\documentclass[12pt]{article}

\usepackage[ansinew]{inputenc}
\usepackage[UKenglish]{babel}
\usepackage{authblk}
\usepackage{textcomp}
\usepackage{mathrsfs}
\usepackage{amsfonts}
\usepackage{mathtools}
\usepackage{amssymb}
\usepackage{amsmath}
\usepackage{amsthm}
\usepackage{amscd}
\usepackage{latexsym}
\usepackage{tensor}
\usepackage{stackengine}
\usepackage{stmaryrd}
\usepackage{rotating}
\usepackage[text={6.8in,8.5in},centering]{geometry}
\usepackage{tikz}
\usetikzlibrary{arrows, calc, decorations.markings, positioning,shapes.geometric}
\usepackage{tikz-cd}
\usepackage{relsize} 
\usepackage[none]{hyphenat}
\usepackage{hyperref}

\usepackage{graphicx}
\graphicspath{ {./Images/} }

\usepackage{hyperref,xcolor}
\usepackage{makeidx}
\definecolor{deepred}{rgb}{0.5,0,0}
\definecolor{deepblue}{rgb}{0,0,0.5}
\definecolor{deepgreen}{rgb}{0,0.5,0}

\hypersetup{
    colorlinks=true,
    linkcolor=deepred,
    filecolor=deepred,      
    urlcolor=deepred,
    citecolor=deepblue,
}

\newcommand{\Id}{\text{id}}

\providecommand{\keywords}[1]
{
  \small	
  \textbf{\textit{Keywords---}} #1
}

\newtheorem{prop}{Proposition}
\numberwithin{prop}{section}

\newtheorem{thm}{Theorem}
\numberwithin{thm}{section}

\usepackage[
backend=biber,
style=alphabetic,
sorting=ynt
]{biblatex}
\begin{document}

\title{\vspace{-1.8cm}Heaps of Fish:\\ arrays, generalized associativity and heapoids}

\author{Carlos Zapata-Carratal\'a$^{1, }$\footnote{lead author, \url{c.zapata.carratala@gmail.com}}  {\,\,\,}   Xerxes D. Arsiwalla$^{2, 3}$\footnote{\url{x.d.arsiwalla@gmail.com}} {\,\,\,} Taliesin Beynon $^{3}$\footnote{\url{tali@taliesin.ai}}\\
{\it \small $^{1}$Edinburgh Mathematical Physics Group, United Kingdom}\\
{\it \small $^{2}$Pompeu Fabra University, Barcelona, Spain}\\
{\it \small $^{3}$Wolfram Research, United States}
}

\date{}

\maketitle

\sloppy

\begin{abstract}
    In this paper we investigate a ternary generalization of associativity by defining a diagrammatic calculus of hypergraphs that extends the usual notions of tensor networks, categories and relational algebras. In doing so we rediscover the ternary structures known as heaps and are able to give a more comprehensive treatment of their emergence in the context of dagger categories and their generalizations. Our key insight is to approach associativity as a confluence property of hypergraph rewrite systems. This approach allows us to define a notion of ternary category and heapoid, where morphisms bind three objects simultaneously, and suggests a systematic study of higher arity forms of associativity.
\end{abstract}


\tableofcontents

\newpage

\section{Introduction} \label{intro}

The present work is part of the ongoing research of the Higher Arity Project collaboration \cite{collab2022arity}. The results shown here are to be taken as the first outcomes of a broader program that aims to better understand compositional aspects of higher order systems and to develop novel mathematical structures that are tailor-made for the modelling of higher order interactions in complex systems.\newline

In this paper we generalize the standard formalisms of multilinear algebra \cite{wald2010general,abraham2012manifolds}, tensor networks \cite{biamonte2017tensor, okunishi2021developments}, relational algebra \cite{burch1991peircean,behrisch2013relational}, and category theory \cite{leinster2014basic} by introducing the concept of array, which encompasses both matrices and relations as indexed objects with values in a commutative semiring. In particular, we focus on the algebra of $3$-index arrays aiming to recover the recent work on cubic matrices \cite{kerner2008ternary,abramov2009algebras} that has identified the following ternary product:
\begin{equation}\label{fishmultplication}
    \sum_{pqr} a_{ijp}\cdot b_{qrp} \cdot c_{qrk} \tag{f}
\end{equation}
which we call the \emph{fish product} for reasons that will become apparent below. Our goal is to investigate the general algebro-compositional properties of this operation.\newline

One of the algebraic structures that appear naturally in our study of the fish product are the ternary algebras known as \emph{semiheaps} which satisfy a generalized version of associativity:
\begin{equation}\label{semiheap}
    ((abc)de)=(a(dcb)e)=(ab(cde)). \tag{sh}
\end{equation}
Semiheaps were introduced by V.V. Wagner \cite{wagner1953theory} (in Russian) during the 1950s while studying partial bijections between sets in an attempt to develop a concise algebraic theory for transition functions of manifold charts in differential geometry. A detailed history of the origins of the ideas that lead to modern semiheap theory, as well as translations of Wagner's original papers, can be found in \cite{hollings2017wagner}. Semiheaps are among the few ternary structures that have received any significant attention from the mathematical community. This can be seen from their recent applications to Morita equivalence \cite{lawson2011generalized}, pseudogroups \cite{kock2007principal}, affine structures \cite{hawthorn2011near,breaz2022heaps}, Lie theory \cite{brzezinski2022trusses} and quantum mechanics \cite{bruce2022semiheaps}.\newline

The fish product appeared in some of the earliest investigations of ternary relational algebra. With roots dating back to the pioneering work of C.S. Peirce \cite{peirce1870description,peirce1880algebra} in the late 1800s, the first recorded use of the fish product appears in a 1967 paper by W.S. McCulloch and R. Moreno-D\'iaz \cite{mcculloch1967triadas} on what they called `triadas', which essentially amount to set-theoretic ternary relations. This is an excerpt from page 340 of \cite{mcculloch1967triadas} showing a multiplication of $3$-index arrays and a supporting diagram in the shape of a fish where the $i$ $j$ indices are placed at the tips of the tail and the $k$ index is placed at the head:
\begin{center}
    \includegraphics[scale=0.35]{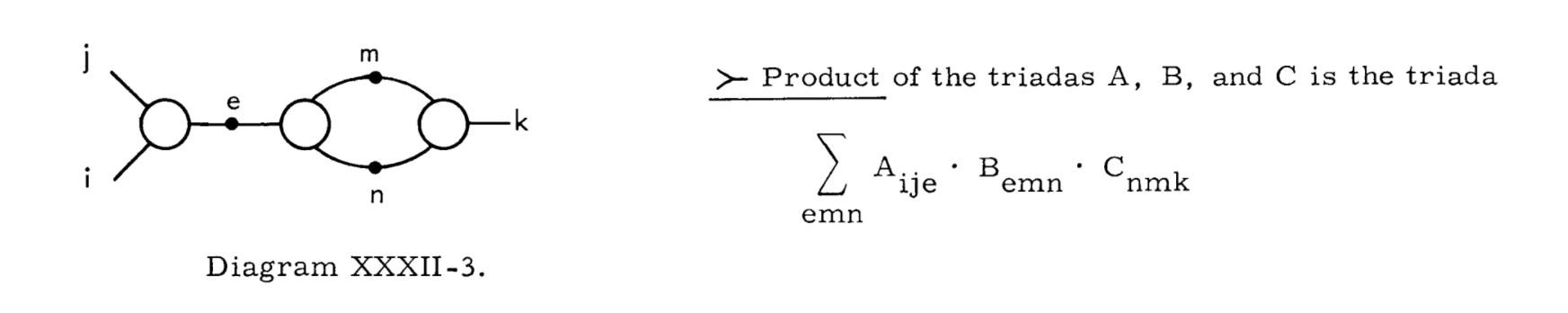}
\end{center}
This line of research gained some traction in the cybernetics community during the 1960s and 1970s. A remarkable example, that in fact came very close to developing the notion of plex algebra we introduce in Section \ref{plex}, is a 1972 paper by C.R. Longyear \cite{longyear1972further} where the fish product is diagrammatically articulated in full detail and the correct sequentialization of the fish product in terms of a ternary multiplication of $3$-index arrays is given for the first time. This is an excerpt from page 13 of \cite{longyear1972further} showing the sequentialization of a fish diagram:
\begin{center}
    \includegraphics[scale=0.3]{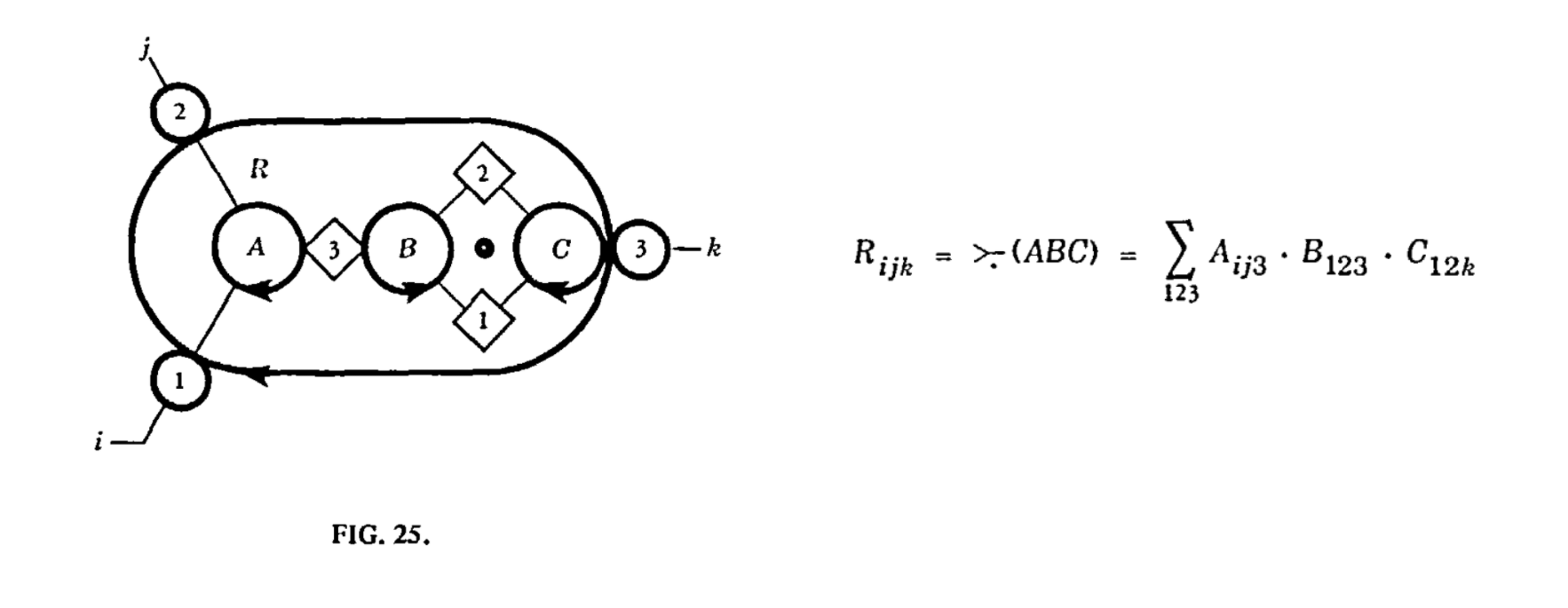}
\end{center}
Unfortunately, due to the decline of the cybernetics movement towards the late 20$^\text{th}$ century, this line of work did not continue to the present day. The fish product remained -- and still largely remains -- unknown to the broader mathematical community. To the best of our knowledge, the next time the fish product occurs in the mathematical literature is in the work of V. Abramov and S. Shitov on cubic matrices \cite{abramov2009algebras}, where it was rediscovered in a systematic computer search for ternary cubic matrix operations satisfying some known forms of generalized associativity. This is an excerpt from page 14 of \cite{abramov2009algebras} showing the possible contraction patterns that define $3$-index array multiplications satisfying the semiheap axiom (\ref{semiheap}):
\begin{center}
    \includegraphics[scale=0.26]{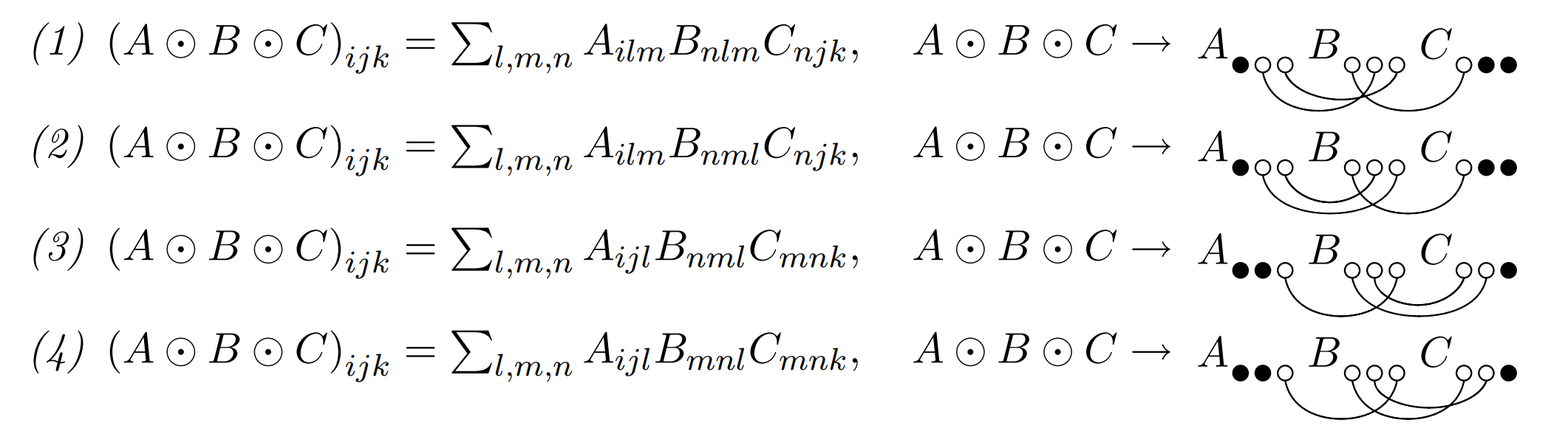}
\end{center}
As far as we can see from the published literature, this was the first time that the fish product and semiheaps were explicitly linked, a connection that does not appear in the research on triadas by the cyberneticians -- semiheap theory was likely unknown to them. We aim to restore the lineage of the fish product by contextualizing it in a general formalism of higher order algebra and by clearly articulating its connection with semiheap theory. In particular, we shall show that the four different multiplications $\textit{(1)}$ - $\textit{(4)}$ above are all instances of a single algebraic operation.\newline

Our main insight comes from the development of a diagrammatic calculus based on hypergraphs that generalizes the diagrams of category theory \cite{leinster2014basic} and tensor networks \cite{okunishi2021developments}. The basic idea is to represent multi-index objects, such as higher order matrices and higher arity relations, as hyperedges of a hypergraph whose vertices represent indices. For example, a $3$-index array is denoted by:
\begin{center}
\begin{tikzpicture}[line join = round, line cap = round]

\begin{scope}[scale=1.2, xshift=0, yshift=-10]

\coordinate [label=above:$k$] (2) at (0.5,{0.5*sqrt(3)});
\coordinate [label= below right:$j$] (1) at (1,0);
\coordinate [label=below left:$i$] (0) at (0,0);

\draw[lightgray, opacity=0.5,fill=lightgray,fill opacity=0.5] (1)--(0)--(2)--cycle;

\end{scope}

\node[] at (0.6,0) {$a$};

\node[] at (2.5,0) {$\sim$};

\node[] at (4,0) {$a_{ijk}$};

\end{tikzpicture}
\end{center}
This notation allows to write the fish product as the following simple hypergraph:
\begin{center}
\begin{tikzpicture}[line join = round, line cap = round]

\begin{scope}[xshift=0, yshift=0]

\begin{scope}[scale=0.7, xshift=0, yshift=0]

\coordinate [label=right:$k$] (6) at ({3*sqrt(3)},0);
\coordinate [label=below:$r$] (5) at ({2*sqrt(3)},-1);
\coordinate [label=above:$q$] (4) at ({2*sqrt(3)},1);
\coordinate [label=above:$p$] (3) at ({sqrt(3)},0);
\coordinate [label=below left:$j$] (2) at (0,-1);
\coordinate [label=above left:$i$] (1) at (0,1);
\coordinate [] (0) at (0,0);

\draw[gray, opacity=0.5,fill=lightgray,fill opacity=0.5] (1)--(2)--(3)--cycle;
\draw[gray, opacity=0.5,fill=lightgray,fill opacity=0.5] (3)--(4)--(5)--cycle;
\draw[gray, opacity=0.5,fill=lightgray,fill opacity=0.5] (4)--(5)--(6)--cycle;

\end{scope}

\node[] at ({0.7*0.6},0) {$a$};
\node[] at ({0.7*(2*sqrt(3)-0.6)},0.05) {$b$};
\node[] at ({0.7*(2*sqrt(3)+0.6)},0) {$c$};

\filldraw[darkgray] (5) circle (0.07);
\filldraw[darkgray] (4) circle (0.07);
\filldraw[darkgray] (3) circle (0.07);

\end{scope}

\node[] at (5.2,0) {$\sim$};

\node[] at (7,0) {$(abc)_{ijk}$};

\node[] at (8,0) {$=$};

\node[] at (10,-0.1) {$\displaystyle\sum_{pqr} a_{ijp}\cdot b_{qrp} \cdot c_{qrk}$};

\end{tikzpicture}
\end{center}
where the vertices marked with a black dot represent index contractions. Our diagrammatic calculus will reveal that the fish product satisfies a form of generalized associativity.\newline

This observation stems from the interpretation that associativity, as an algebraic axiom, is the consequence of the confluence property of a rewrite system. Such a viewpoint has a longstanding tradition in computer science \cite{barendsen2003term}, with D. Knuth being one of its forerunners \cite{knuth1983simple}, and has seen some applications to abstract algebra in recent years \cite{poinsot2010partial,endrullis2019confluence}. Our approach takes direct inspiration from the standard formalism of category theory \cite{leinster2014basic}, the recent work of E. Gnang \cite{gnang2014combinatorial,gnang2020bhattacharya} and T. Beynon \cite{beynon2022projects}, and S. Wolfram's multiway systems \cite{wolfram2002new,wolfram2020wpp,wolfram2020class}. Our central result is summarized in the following confluent hypergraph rewrite system:
\begin{center}
\begin{tikzpicture}[line join = round, line cap = round, scale=0.8]

\begin{scope}[xshift=-45, yshift=30]

\begin{scope}[scale=0.7, xshift=0, yshift=0]

\coordinate [] (9) at ({5*sqrt(3)},0);
\coordinate [] (8) at ({4*sqrt(3)},-1);
\coordinate [] (7) at ({4*sqrt(3)},1);
\coordinate [] (6) at ({3*sqrt(3)},0);
\coordinate [] (5) at ({2*sqrt(3)},-1);
\coordinate [] (4) at ({2*sqrt(3)},1);
\coordinate [] (3) at ({sqrt(3)},0);
\coordinate [] (2) at (0,-1);
\coordinate [] (1) at (0,1);
\coordinate [] (0) at (0,0);

\draw[gray, opacity=0.5,fill=lightgray,fill opacity=0.5] (1)--(2)--(3)--cycle;
\draw[gray, opacity=0.5,fill=lightgray,fill opacity=0.5] (3)--(4)--(5)--cycle;
\draw[gray, opacity=0.5,fill=lightgray,fill opacity=0.5] (4)--(5)--(6)--cycle;
\draw[gray, opacity=0.5,fill=lightgray,fill opacity=0.5] (6)--(7)--(8)--cycle;
\draw[gray, opacity=0.5,fill=lightgray,fill opacity=0.5] (8)--(7)--(9)--cycle;

\end{scope}

\node[] at ({0.7*0.6},0) {$a$};
\node[] at ({0.7*(2*sqrt(3)-0.6)},0.05) {$b$};
\node[] at ({0.7*(2*sqrt(3)+0.6)},0) {$c$};
\node[] at ({0.7*(4*sqrt(3)-0.6)},0.05) {$d$};
\node[] at ({0.7*(4*sqrt(3)+0.6)},0) {$e$};

\filldraw[darkgray] (8) circle (0.07);
\filldraw[darkgray] (7) circle (0.07);
\filldraw[darkgray] (6) circle (0.07);
\filldraw[darkgray] (5) circle (0.07);
\filldraw[darkgray] (4) circle (0.07);
\filldraw[darkgray] (3) circle (0.07);

\end{scope}

\begin{scope}[xshift=-170, yshift=-91]

\begin{scope}[scale=0.7, xshift=0, yshift=0]

\coordinate [] (6) at ({3*sqrt(3)},0);
\coordinate [] (5) at ({2*sqrt(3)},-1);
\coordinate [] (4) at ({2*sqrt(3)},1);
\coordinate [] (3) at ({sqrt(3)},0);
\coordinate [] (2) at (0,-1);
\coordinate [] (1) at (0,1);
\coordinate [] (0) at (0,0);

\draw[gray, opacity=0.5,fill=lightgray,fill opacity=0.5] (1)--(2)--(3)--cycle;
\draw[gray, opacity=0.5,fill=lightgray,fill opacity=0.5] (3)--(4)--(5)--cycle;
\draw[gray, opacity=0.5,fill=lightgray,fill opacity=0.5] (4)--(5)--(6)--cycle;

\end{scope}

\node[] at ({0.7*0.6},0.05) {$abc$};
\node[] at ({0.7*(2*sqrt(3)-0.6)},0.05) {$d$};
\node[] at ({0.7*(2*sqrt(3)+0.6)},0) {$e$};

\filldraw[darkgray] (5) circle (0.07);
\filldraw[darkgray] (4) circle (0.07);
\filldraw[darkgray] (3) circle (0.07);

\end{scope}

\begin{scope}[xshift=-20, yshift=-91]

\begin{scope}[scale=0.7, xshift=0, yshift=0]

\coordinate [] (6) at ({3*sqrt(3)},0);
\coordinate [] (5) at ({2*sqrt(3)},-1);
\coordinate [] (4) at ({2*sqrt(3)},1);
\coordinate [] (3) at ({sqrt(3)},0);
\coordinate [] (2) at (0,-1);
\coordinate [] (1) at (0,1);
\coordinate [] (0) at (0,0);

\draw[gray, opacity=0.5,fill=lightgray,fill opacity=0.5] (1)--(2)--(3)--cycle;
\draw[gray, opacity=0.5,fill=lightgray,fill opacity=0.5] (3)--(4)--(5)--cycle;
\draw[gray, opacity=0.5,fill=lightgray,fill opacity=0.5] (4)--(5)--(6)--cycle;

\end{scope}

\node[] at ({0.7*0.6},0) {$a$};
\node[] at ({0.7*(2*sqrt(3)-0.6)},0.05) {$bcd$};
\node[] at ({0.7*(2*sqrt(3)+0.6)},0) {$e$};

\filldraw[darkgray] (5) circle (0.07);
\filldraw[darkgray] (4) circle (0.07);
\filldraw[darkgray] (3) circle (0.07);

\end{scope}

\begin{scope}[xshift=130, yshift=-91]

\begin{scope}[scale=0.7, xshift=0, yshift=0]

\coordinate [] (6) at ({3*sqrt(3)},0);
\coordinate [] (5) at ({2*sqrt(3)},-1);
\coordinate [] (4) at ({2*sqrt(3)},1);
\coordinate [] (3) at ({sqrt(3)},0);
\coordinate [] (2) at (0,-1);
\coordinate [] (1) at (0,1);
\coordinate [] (0) at (0,0);

\draw[gray, opacity=0.5,fill=lightgray,fill opacity=0.5] (1)--(2)--(3)--cycle;
\draw[gray, opacity=0.5,fill=lightgray,fill opacity=0.5] (3)--(4)--(5)--cycle;
\draw[gray, opacity=0.5,fill=lightgray,fill opacity=0.5] (4)--(5)--(6)--cycle;

\end{scope}

\node[] at ({0.7*0.6},0) {$a$};
\node[] at ({0.7*(2*sqrt(3)-0.6)},0.05) {$b$};
\node[] at ({0.7*(2*sqrt(3)+0.6)},0.05) {$cde$};

\filldraw[darkgray] (5) circle (0.07);
\filldraw[darkgray] (4) circle (0.07);
\filldraw[darkgray] (3) circle (0.07);

\end{scope}

\begin{scope}[xshift=20, yshift=-220]

\coordinate [] (6) at ({3*sqrt(3)},0);
\coordinate [] (5) at ({2*sqrt(3)},-1);
\coordinate [] (4) at ({2*sqrt(3)},1);
\coordinate [] (3) at ({sqrt(3)},0);
\coordinate [] (2) at (0,-1);
\coordinate [] (1) at (0,1);
\coordinate [] (0) at (0,0);

\draw[gray, opacity=0.5,fill=lightgray,fill opacity=0.5] (1)--(2)--(3)--cycle;

\node[] at (0.62,0.05) {$abcde$};

\end{scope}

\draw[-stealth, line width=1pt] (-1.5,-0.5) to[out=210, in=80] (-2.5,-2);
\draw[-stealth, line width=1pt] (4.2,-0.5) to[out=-30, in=100] (5.2,-2);
\draw[-stealth, line width=1pt] (-2.5,-4.5) to[out=-80, in=150] (-1.5,-6);
\draw[-stealth, line width=1pt] (5.2,-4.5) to[out=-100, in=30] (4.2,-6);
\draw[-stealth, line width=1pt] (1,-0.5) to (1,-2);
\draw[-stealth, line width=1pt] (1,-4.5) to (1,-6);

\begin{scope}[scale=0.15, xshift=-700, yshift=-150]

\coordinate [] (9) at ({5*sqrt(3)},0);
\coordinate [] (8) at ({4*sqrt(3)},-1);
\coordinate [] (7) at ({4*sqrt(3)},1);
\coordinate [] (6) at ({3*sqrt(3)},0);
\coordinate [] (5) at ({2*sqrt(3)},-1);
\coordinate [] (4) at ({2*sqrt(3)},1);
\coordinate [] (3) at ({sqrt(3)},0);
\coordinate [] (2) at (0,-1);
\coordinate [] (1) at (0,1);
\coordinate [] (0) at (0,0);

\draw[gray, opacity=0.5,fill=darkgray,fill opacity=0.5] (1)--(2)--(3)--cycle;
\draw[gray, opacity=0.5,fill=darkgray,fill opacity=0.5] (3)--(4)--(5)--cycle;
\draw[gray, opacity=0.5,fill=darkgray,fill opacity=0.5] (4)--(5)--(6)--cycle;
\draw[gray, opacity=0.5,fill=lightgray,fill opacity=0.5] (6)--(7)--(8)--cycle;
\draw[gray, opacity=0.5,fill=lightgray,fill opacity=0.5] (8)--(7)--(9)--cycle;

\end{scope}

\begin{scope}[scale=0.15, xshift=-700, yshift=-1000]

\coordinate [] (6) at ({3*sqrt(3)},0);
\coordinate [] (5) at ({2*sqrt(3)},-1);
\coordinate [] (4) at ({2*sqrt(3)},1);
\coordinate [] (3) at ({sqrt(3)},0);
\coordinate [] (2) at (0,-1);
\coordinate [] (1) at (0,1);
\coordinate [] (0) at (0,0);

\draw[gray, opacity=0.5,fill=darkgray,fill opacity=0.5] (1)--(2)--(3)--cycle;
\draw[gray, opacity=0.5,fill=darkgray,fill opacity=0.5] (3)--(4)--(5)--cycle;
\draw[gray, opacity=0.5,fill=darkgray,fill opacity=0.5] (4)--(5)--(6)--cycle;

\end{scope}

\begin{scope}[scale=0.15, xshift=270, yshift=-150]

\coordinate [] (9) at ({5*sqrt(3)},0);
\coordinate [] (8) at ({4*sqrt(3)},-1);
\coordinate [] (7) at ({4*sqrt(3)},1);
\coordinate [] (6) at ({3*sqrt(3)},0);
\coordinate [] (5) at ({2*sqrt(3)},-1);
\coordinate [] (4) at ({2*sqrt(3)},1);
\coordinate [] (3) at ({sqrt(3)},0);
\coordinate [] (2) at (0,-1);
\coordinate [] (1) at (0,1);
\coordinate [] (0) at (0,0);

\draw[gray, opacity=0.5,fill=lightgray,fill opacity=0.5] (1)--(2)--(3)--cycle;
\draw[gray, opacity=0.5,fill=darkgray,fill opacity=0.5] (3)--(4)--(5)--cycle;
\draw[gray, opacity=0.5,fill=darkgray,fill opacity=0.5] (4)--(5)--(6)--cycle;
\draw[gray, opacity=0.5,fill=darkgray,fill opacity=0.5] (6)--(7)--(8)--cycle;
\draw[gray, opacity=0.5,fill=lightgray,fill opacity=0.5] (8)--(7)--(9)--cycle;

\end{scope}

\begin{scope}[scale=0.15, xshift=270, yshift=-1000]

\coordinate [] (6) at ({3*sqrt(3)},0);
\coordinate [] (5) at ({2*sqrt(3)},-1);
\coordinate [] (4) at ({2*sqrt(3)},1);
\coordinate [] (3) at ({sqrt(3)},0);
\coordinate [] (2) at (0,-1);
\coordinate [] (1) at (0,1);
\coordinate [] (0) at (0,0);

\draw[gray, opacity=0.5,fill=darkgray,fill opacity=0.5] (1)--(2)--(3)--cycle;
\draw[gray, opacity=0.5,fill=darkgray,fill opacity=0.5] (3)--(4)--(5)--cycle;
\draw[gray, opacity=0.5,fill=darkgray,fill opacity=0.5] (4)--(5)--(6)--cycle;

\end{scope}

\begin{scope}[scale=0.15, xshift=1000, yshift=-150]

\coordinate [] (9) at ({5*sqrt(3)},0);
\coordinate [] (8) at ({4*sqrt(3)},-1);
\coordinate [] (7) at ({4*sqrt(3)},1);
\coordinate [] (6) at ({3*sqrt(3)},0);
\coordinate [] (5) at ({2*sqrt(3)},-1);
\coordinate [] (4) at ({2*sqrt(3)},1);
\coordinate [] (3) at ({sqrt(3)},0);
\coordinate [] (2) at (0,-1);
\coordinate [] (1) at (0,1);
\coordinate [] (0) at (0,0);

\draw[gray, opacity=0.5,fill=lightgray,fill opacity=0.5] (1)--(2)--(3)--cycle;
\draw[gray, opacity=0.5,fill=lightgray,fill opacity=0.5] (3)--(4)--(5)--cycle;
\draw[gray, opacity=0.5,fill=darkgray,fill opacity=0.5] (4)--(5)--(6)--cycle;
\draw[gray, opacity=0.5,fill=darkgray,fill opacity=0.5] (6)--(7)--(8)--cycle;
\draw[gray, opacity=0.5,fill=darkgray,fill opacity=0.5] (8)--(7)--(9)--cycle;

\end{scope}

\begin{scope}[scale=0.15, xshift=1000, yshift=-1000]

\coordinate [] (6) at ({3*sqrt(3)},0);
\coordinate [] (5) at ({2*sqrt(3)},-1);
\coordinate [] (4) at ({2*sqrt(3)},1);
\coordinate [] (3) at ({sqrt(3)},0);
\coordinate [] (2) at (0,-1);
\coordinate [] (1) at (0,1);
\coordinate [] (0) at (0,0);

\draw[gray, opacity=0.5,fill=darkgray,fill opacity=0.5] (1)--(2)--(3)--cycle;
\draw[gray, opacity=0.5,fill=darkgray,fill opacity=0.5] (3)--(4)--(5)--cycle;
\draw[gray, opacity=0.5,fill=darkgray,fill opacity=0.5] (4)--(5)--(6)--cycle;

\end{scope}

\end{tikzpicture}
\end{center}
Since the order of alphabetic labels should be read from \emph{tail-to-head} in the anatomy of the fish product, the three middle terms of this rewrite system give the symbolic equations (\ref{semiheap}), which are precisely the defining axiom of a semiheap.\newline

The contents are organized as follows: we begin by presenting a concise summary of semiheap theory in Section \ref{heaps}; the general formalism of array algebra and the corresponding hypergraph diagram calculus is developed in Sections \ref{array} and \ref{plex}; we make our interpretation of associativity as rewrite system confluence precise in Section \ref{confl}; the fish product and its connection with semiheap theory is investigated in detail in Section \ref{fish}; a definition of heapoid and a ternary notion of category are proposed in Section \ref{heapoids}; and we conclude with some remarks on future work towards higher arity generalizations of associativity in Section \ref{towards}.\newline

\section{Heaps}\label{heaps}

Let $H$ be a set and $\eta : H \times H \times H \to H$ a ternary operation, we write $\eta (a,b,c) \equiv (abc)$. The pair $(H,\eta)$ is called a \textbf{semiheap} when the ternary operation satisfies the equations (\ref{semiheap}), that is
\begin{equation*}
    ((abc)de)=(a(dcb)e)=(ab(cde))
\end{equation*}
for all $a,b,c,d,e\in H$. Note that the middle term in (\ref{semiheap}) implies a departure from sequential associativity understood as the property satisfied by concatenation of a binary associative operation, i.e. the property that notationally allows for brackets to be dropped. Consequently, property (\ref{semiheap}) has been called weak associativity \cite{wagner1951ternary}, associativity of the second kind \cite{carlsson1976cohomology}, quasi-associativity \cite{kolar2000heap}, para-associativity \cite{hawthorn2009radical}, type B associativity \cite{kerner2008ternary} and pseudo-associativity \cite{hollings2017wagner}. Given a semiheap structure on a set $(H,\eta)$ it is possible to define another semiheap structure on the same set by reversing the sequential order of arguments: if the ternary operation $\eta$ satisfies (\ref{semiheap}) then so does $\overline{\eta}(a,b,c):=\eta(c,b,a)$ and $(H,\overline{\eta})$ is called the \textbf{reverse semiheap}. A map $\varphi:(H,\eta) \to (K,\kappa)$ satisfying
\begin{equation*}
    \varphi(\eta(a,b,c))=\kappa(\varphi(a),\varphi(b),\varphi(c))
\end{equation*}
for all $a,b,c\in H$ is called a \textbf{semiheap homomorphism}. Although largely unknown beyond specialist circles, semiheaps turn out to be minimal algebraic structures that appear naturally in several areas of mathematics such as manifold theory, relational calculus and category theory.\newline

The standard example of a semiheap structure is found in the set of binary relations between two arbitrary sets $\textsf{Rel}(A,B)$ by defining the following ternary operation:
\begin{equation*}
    \eta_{\textsf{Rel}}(R_1,R_2,R_3) := R_1 \circ R_2^\top \circ R_3
\end{equation*}
where $R_1,R_2,R_3\subset A\times B$ are binary relations, $\circ$ is the usual composition of relations and $\top$ denotes the transpose or converse relation. It is easy to see that (\ref{semiheap}) holds as a direct consequence of the associativity property of $\circ$ and the fact that $\top$ is idempotent and a $\circ$-antihomomorphism. More generally, the morphism sets of any dagger category $(\mathcal{C},\dag)$ carry a semiheap structure via the analogous construction:
\begin{equation*}
    \eta_{\mathcal{C}}(f,g,h) := f \circ g^\dag \circ h
\end{equation*}
for $f,g,h\in \mathcal{C}(A,B)$ and any two objects $A,B\in \mathcal{C}$. The commutative diagram that realizes this ternary operation in a dagger category displays a characteristic \emph{back-and-forth} pattern:
\begin{equation}\label{dag}
\begin{tikzcd}[sep=large]
A \arrow[rr, "f", bend left] \arrow[rr, "h", bend right] & & B \arrow[ll, "g^\dag"']
\end{tikzcd} \tag{dag}
\end{equation}

Another important class of examples of semiheaps are involuted semigroups. Let a set $G$ together with an associative binary operation $\cdot$ and a map $*:G\to G$ satisfying $(a^*)^*=a$ and $(a\cdot b)^*=b^*\cdot a^*$ for all $a,b\in G$, then it is easy to show that the ternary operation defined by
\begin{equation*}
    \eta_G(a,b,c) := a \cdot b^*\cdot c
\end{equation*}
satisfies (\ref{semiheap}) thus making $(G,\eta_G)$ into a semiheap. In particular, all groups carry a canonical semiheap structure via the above construction by considering the inversion involution $*=\,^{-1}$.\newline

The generalization of a neutral element for a ternary operation motivates the following notion: let a semiheap $(H,\eta)$, an element $e\in H$ is called a \textbf{left biunit} when
\begin{equation*}
    (eea)=a,
\end{equation*}
a \textbf{right biunit} when
\begin{equation*}
    (aee)=a
\end{equation*}
for all $a\in H$. When $e\in H$ is both a left and a right biunit it is simply called a \textbf{biunit}. A ternary operation for which all elements are biunits is sometimes called a \textbf{Malcev} operation. Biunits correspond to familiar situations in the examples of semiheaps presented above: bijections of sets are biunits in the semiheap of heterogeneous binary relations, unitary morphisms are biunits in the semiheaps of dagger categories, neutral elements in an involuted semigroup are biunits of the corresponding semiheap and all elements of a group are biunits of the corresponding semiheap.\newline

The existence of biunits in a semiheap allows to construct an involuted monoid.
\begin{thm}[Involuted Monoid of a Semiheap with Biunit, Thm. 8.2.8-9 \cite{hollings2017wagner}] \label{monoidsemiheap}
Let a semiheap $(H,\eta)$ and a biunit $e\in H$, then the binary operation
\begin{equation*}
    a\cdot_e b := (aeb)
\end{equation*}
and the map
\begin{equation*}
    a^{*_e} := (eae)
\end{equation*}
make $(H,\cdot_e,*_e)$ into an involuted monoid with $e$ as identity element. Furthermore, given another biunit $e'\in H$, there is an isomorphism of involuted monoids $\varphi_{ee'}:(H,\cdot_e,*_e)\to (H,\cdot_{e'},*_{e'})$ given by
\begin{equation*}
    \varphi_{ee'}(a) := (aee').
\end{equation*}
\end{thm}
This construction is easily shown to be the converse of the semiheap structure corresponding to any involuted semigroup and thus the above theorem establishes a one-to-one correspondence between involuted monoids and semiheaps with a single biunit. The connection between semiheaps and involuted monoids is, in fact, more general.
\begin{thm}[Semiheaps Embed into Involution Monoids, Thm. 8.2.10-11 \cite{hollings2017wagner}] \label{embedsemiheap}
Any semiheap can be homomorphically embedded into a semiheap with a single biunit. Therefore, it follows from Theorem \ref{monoidsemiheap} that any semiheap can be embedded into an involuted monoid.
\end{thm}

A semiheap $(H,\eta)$ whose elements are all biunits is called a \textbf{heap}. It is easily checked that Theorems \ref{monoidsemiheap} and \ref{embedsemiheap} imply that any heap can be homomorphically embedded into a group. The condition that all elements of a semiheap be biunits is rather strong: it turns out that, under this condition, the equalities involving the non-associative term in (\ref{semiheap}) can be relaxed.
\begin{thm}[Heaps, Thm. 8.2.13 \cite{hollings2017wagner}] \label{heapaxiom}
A set $H$ with a ternary operation satisfying
\begin{equation}\label{heap}
    ((abc)de)=(ab(cde)) \tag{h}
\end{equation}
and
\begin{equation}\label{malcev}
    (aab)=b=(baa) \tag{m}
\end{equation}
for all $a,b,c,d,e \in H$, is a heap.
\end{thm}

A first example of heap is the semiheap structure on morphism sets of dagger categories restricted to unitary morphisms. In particular, bijections between sets naturally form a heap with the standard ternary operation defined for binary relations above. This was, in fact, the example that motivated the original definition of heaps by V.V. Wagner in \cite{wagner1951ternary} and the subsequent development of the theory of semiheaps and generalized heaps \cite{hollings2017wagner}.\newline

Heaps offer a compelling perspective that reframes the notion of groupoid in concise graph theoretic terms. This was first noted by A. Kock in \cite{kock1982algebraic} and we recast it here with a focus on the interpretation of groupoids as generalizations of group actions. Recall that a groupoid is a category $\mathcal{G}$ whose morphisms are all invertible. Given any object $A\in \mathcal{G}$, it is easy to see that the morphism set $\mathcal{G}(A,A)=:\mathcal{G}_A$ forms a group under composition of morphisms $(\mathcal{G}_A,\circ,1_A)$ and is thus called the \textbf{isotropy group} at $A$. Similarly, taking two objects $A,B\in \mathcal{G}$, the morphism set $\mathcal{G}(A,B)=:\mathcal{G}_{AB}$ inherits a heap structure since groupoids are a particular case of a dagger category where $\dag=^{-1}$. The morphism set between two objects of a groupoid together with the heap structure $(\mathcal{G}_{AB},\eta_{AB})$ is called the \textbf{translation heap}. Our first observation is that the translation heap of an object into itself $(\mathcal{G}_{AA},\eta_{AA})$ corresponds to the heap constructed from the isotropy group $(\mathcal{G}_A,\circ,1_A)$ regarded as an involuted semigroup:
\begin{equation*}
    (\mathcal{G}_{AA},\eta_{AA}) \cong (\mathcal{G}_A,\eta_{\mathcal{G}_A}).
\end{equation*}
Isotropy groups and translation heaps interact via left and right actions.
\begin{prop}[Isotropy Actions] \label{isotropyactions}
Let two objects of a groupoid $A,B\in \mathcal{G}$, then composition of morphisms defines commuting right and left actions of the isotropy groups on the translation heap
\begin{equation*}
    \mathcal{G}_A \circlearrowright (\mathcal{G}_{AB},\eta_{AB}) \circlearrowleft \mathcal{G}_B
\end{equation*}
that satisfy the following bi-invariance property with respect to the heap operation:
\begin{equation*}
    \eta_{AB}(f\circ a, b \circ g \circ a, b \circ h)=\eta_{AB}(f,g,h)
\end{equation*}
for all $a \in \mathcal{G}_A$, $b \in \mathcal{G}_B$ and $f,g,h\in \mathcal{G}_{AB}$.
\end{prop}
\begin{proof}
This follows directly from the definitions of isotropy groups and translation heaps, and the basic properties of morphism composition in categories.
\end{proof}
Note that translation heaps are formally dependent on the orientation of the objects, in particular, $\mathcal{G}_{AB}\neq \mathcal{G}_{BA}$ as sets. Nevertheless, the inverse involution of a groupoid induces a heap isomorphism between translation heaps in opposite orientations.
\begin{prop}[Reverse Translation Heaps] \label{reverseheaps}
Let two objects of a groupoid $A,B\in \mathcal{G}$, then there is a canonical isomorphism of heaps
\begin{equation*}
    (\mathcal{G}_{AB},\eta_{AB}) \cong (\mathcal{G}_{BA},\overline{\eta}_{BA}).
\end{equation*}
\end{prop}
\begin{proof}
The isomorphism is given by the inversion involution which restricts to a bijection between morphism sets
\begin{equation*}
    ^{-1}:\mathcal{G}_{AB} \to \mathcal{G}_{BA}.
\end{equation*}
We easily check
\begin{equation*}
    \eta_{AB}(f,g,h)^{-1}=(f\circ g^{-1} \circ h)^{-1} = h^{-1} \circ (g^{-1})^{-1} \circ f^{-1}=\overline{\eta}_{BA}(f^{-1},g^{-1},h^{-1})
\end{equation*}
for all $f,g,h\in \mathcal{G}_{AB}$, thus giving the heap homomorphism.
\end{proof}
These results lead to a reinterpretation of groupoids, regarded as partial algebraic structures, as networks of total algebraic structures: a groupoid defines a graph whose vertices are isotropy groups and whose edges are translation heaps together with the isotropy actions. The resulting graph is simple from the fact that morphism sets in a category are unique, undirected as a consequence of Proposition \ref{reverseheaps} and component-complete since all objects linked via some morphisms are also linked with all adjacent objects by composition and invertibility.\newline

\begin{center}
\begin{tikzpicture}[node distance={35mm}, thick, main/.style = {draw, circle}]

\node[main] (1) {group};
\node[main] (2) [above right of=1] {group};
\node[main] (3) [below right of=1] {group};
\node[main] (4) [above right of=3] {group};
\node[main] (5) [above right of=4] {group};
\node[main] (6) [below right of=4] {group};

\draw (1) -- node[sloped,fill=white] {heap} (2);
\draw (1) -- node[sloped,fill=white] {heap} (3);
\draw (1) -- (4);
\draw (1) -- (5);
\draw (1) -- (6);
\draw (2) -- node[sloped,fill=white] {heap} (3);
\draw (2) -- node[sloped,fill=white] {heap} (4);
\draw (2) -- node[sloped,fill=white] {heap} (5);
\draw (3) -- node[sloped,fill=white] {heap} (4);
\draw (3) -- node[sloped,fill=white] {heap} (6);
\draw (4) -- node[sloped,fill=white] {heap} (5);
\draw (4) -- node[sloped,fill=white] {heap} (6);
\draw (5) -- node[sloped,fill=white] {heap} (6);

\draw (5.5,0) ellipse (6.5 and 4);

\node[] at (9.7,0) {groupoid};

\end{tikzpicture}
\end{center}

A concrete instance of the heap approach to groupoids can be seen in the recent application of heap theory to the study affine structures \cite{breaz2022heaps,brzezinski2022trusses}. In its simplest form, the connection between heaps and affine structures works as follows: the point set of an affine space $A$ can be regarded as the objects of a groupoid whose morphisms are the translations represented by the free vectors of the associated vector space $V$. The vector space $V$ can be regarded as an abelian group or, equivalently, as a heap with ternary operation given as above:
\begin{equation*}
    [v,u,w]:=v-u+w
\end{equation*}
for any vectors vectors $v,u,w\in V$. Translations then correspond to single vectors by the trivial fact that the heap operation is idempotent, thus recovering the standard notion that vectors determine translations between pairs of points of an affine space.

\section{Array Algebra} \label{array}

In this section we give a precise definition of the general notion of \emph{array} with the aim to encompass set-theoretic relations and higher order matrices -- also (inaccurately) called `tensors' in the computer science literature. Our focus will be on highlighting compositional properties of some natural operations that can be defined on arrays.\newline

Consider some collection of sets $\mathcal{U}$, we call $\mathcal{U}$ the \textbf{index universe} and its members $I\in \mathcal{U}$ \textbf{index sets} or simply \textbf{indices}. A finite multiset of indices $\{I_1,I_2,\dots,I_n\}$ is called a \textbf{constellation} and the cardinality $n$ of a constellation is called its \textbf{order}. Consider a set $V$ with some internal algebraic structure, we call $V$ the \textbf{value set} and its elements \textbf{values}. For the sake of concreteness, and looking to recover the desired examples of relations and matrices, the value set will be assumed to be a commutative semiring $(V,+,\cdot)$. A map from sequentially ordered tuples of a constellation into the value set
\begin{equation*}
    a: I_1 \times I_2 \times \cdots \times I_n \to V
\end{equation*}
is called a \textbf{$V$-valued array of order $n$ in $\mathcal{U}$}. We adopt the convention by which order $0$ arrays correspond to elements of $V$. By setting $\mathcal{U}$ to be the universe of all sets and the binary Boolean algebra as the value set $V=\mathbb{B}:=(\{0,1\},\vee, \wedge)$, we obtain ordinary set-theoretic relations. Similarly, by setting $\mathcal{U}_{\text{fin}}$ to be the universe of finite sets and a ring as the value set $V=(R,+,\cdot)$, we obtain matrices over $R$ scalars. Unless otherwise specified, the index universe $\mathcal{U}$ and the value set $V$ will always be fixed in our subsequent discussions, so we will simply use the term \textbf{array} to refer to a map $a$ as above or \textbf{$n$-array} to further specify the order. The evaluations of an array are called its \textbf{entries} and they are notated indistinctly by:
\begin{equation*}
    a(i_1,i_2,\dots, i_n) \equiv a_{i_1i_2\cdots i_n}.
\end{equation*}
A collection of arrays $\{a(I_1,\dots,K,\dots,I_n), b(J_1,\dots,K,\dots,J_m),\dots\}$ sharing a common index set $K$ in their constellation are said to be \textbf{incident} on the index $K$.\newline

We note in passing that the necessary structure required on an index universe for arrays to be defined is that of a symmetric monoidal category and thus we could take a universe to be any symmetric monoidal category. However, in order to stay close to our motivating examples, we shall limit our present discussion to sets and ordered tuples defined via the standard Cartesian product of sets. The symmetric monoidal structure of the index universe induces the following natural index manipulation operations on arrays:
\begin{itemize}
    \item Reordering. The braiding isomorphisms of the index universe induce natural bijections of Cartesian products of indices:
    \begin{equation*}
        \varphi_\sigma : I_1 \times I_2 \times \cdots \times I_n \to I_{\sigma(1)} \times I_{\sigma(2)} \times \cdots \times I_{\sigma(n)}
    \end{equation*}
    where $\sigma\in S_n$ is a permutation of the sequentially ordered indices. Given an $n$-array $a$ as above, the bijections $\varphi_\sigma$ define new $n$-arrays via pullback:
    \begin{equation*}
        a^\sigma := a \circ \varphi_\sigma^{-1}.
    \end{equation*}
    The $|S_n|$ $n$-arrays $a^\sigma$ are called the \textbf{reorderings} of $a$ and they all effectively carry the same information. When some of the indices are identical $I_1=I_2=\dots =I$ the reordering operation together with the automorphism group of the value set give the basis for the discussion of \textbf{array symmetry} as equivariance. The general study of array symmetry will be pursued in future work.
    \item Flattening. The associativity isomorphisms of the Cartesian product of sets allow to formally rewrite arrays as lower order arrays via the definition of multi-indices. For instance, given a $n$-array $a$ as above, we could regard $I_1\times I_2=:I$ as a multi-index set and define the \textbf{flattened} $(n-1)$-array:
    \begin{equation*}
        \underline{a}:I\times I_3 \times \cdots \times I_n \to V
    \end{equation*}
    which formally carries the same information as $a$.
    \item Broadening. The canonical projection $\text{pr}_1 : I_1 \times I_2 \to I_1$ allows to extend, or broadcast, as it is known in the computer science literature \cite{abadi2016tensorflow}, any given array to higher order arrays with redundant entries. For example, given a $n$-array $a$ as above and an index set $J\in \mathcal{U}$, we can define the \textbf{broadened} array
    \begin{equation*}
        \overline{a}: I_1 \times I_2 \times \cdots \times I_n \times J \to V
    \end{equation*}
    simply by pullback with the projection
    \begin{equation*}
        \overline{a}:=a \circ \text{pr}_1
    \end{equation*}
    where the $n$ first indices of the Cartesian product are grouped together as a single factor.
    \item Slicing. Given an $n$-array $a$ as above we can obtain lower order arrays by currying index values. For instance, let some element of the first index set $j\in I_1$, then the map
    \begin{equation*}
        a(j,-,\dots,-):I_2\times \cdots \times I_n \to V
    \end{equation*}
    is clearly a $(n-1)$-array. In general, currying $m\leq n$ index values produces a $(n-m)$-array called a \textbf{$m$-slice} of the original array $a$. The $n$-slices of a $n$-array correspond to its entries.
\end{itemize}
Combining these index manipulations gives considerable freedom in the explicit indexation of high order arrays. In practice, however, we shall only use these operations for the sake of notational concision and to highlight compositional patterns that could otherwise be obscured.\newline

The singled-out values corresponding to the additive and multiplicative identities of the semiring $(V,+,\cdot)$ allow to define special arrays: an \textbf{empty array} or \textbf{zero array} is a constant array with value $0\in V$ and a \textbf{full array} or \textbf{Hadamard array} is a constant array with value $1\in V$, by slight abuse of notation we may denote them by $0$ and $1$ respectively if the context does not lead to confusion. Consider a single index set $I\in \mathcal{U}$, $n$-arrays of the form
\begin{equation*}
    a:I\times I \times \stackrel{n}{\cdots} \times I \to V
\end{equation*}
are called \textbf{regular arrays}. The regular empty and full $n$-arrays are denoted $0_n$ and $1_n$ respectively. Exploiting the fact that all indices of a regular array are identical, we define the \textbf{identity array} or \textbf{Kronecker array} $\delta_n$ as follows:
\begin{equation*}
    \delta_n(i_1,i_2,\dots,i_n)=
    \begin{cases}
    1 \quad \text{if} \quad i_1=i_2=\cdots=i_n\\
    0 \quad \text{otherwise}.
    \end{cases}
\end{equation*}

The semiring binary operations on the value set $(V,+,\cdot)$ induce array operations of varying degree of compatibility with the index structure of arrays. The most strictly compatible operations are \textbf{entry-wise addition $+$} and \textbf{entry-wise multiplication $\cdot$} defined in the obvious way as binary operations of arrays over the same constellation of indices. The zero and Hadamard arrays act as the additive and multiplicative neutral elements, respectively, with respect to these operations. More generally, given a collection of incident arrays
\begin{equation*}
    a_1,a_2,\dots,a_N
\end{equation*}
over constellations of indices
\begin{equation*}
    \{I_1,\dots, P, \dots,I_{n_1}\}, \{J_1,\dots, P, \dots, J_{n_2}\}, \dots, \{K_1,\dots, P, \dots, K_{n_N}\}
\end{equation*}
all sharing at least one common incident index $P$, an \textbf{additive incidence} is a new array whose entries are given by:
\begin{equation*}
    a_1(i_1,\dots, l, \dots,i_{n_1}) + a_2(j_1, \dots, l, \dots, j_{n_2}) + \cdots + a_N(k_1,\dots, l, \dots, k_{n_N}).
\end{equation*}
Similarly, a \textbf{multiplicative incidence} is an array whose entries are:
\begin{equation*}
    a_1(i_1,\dots, l, \dots,i_{n_1}) \cdot a_2(j_1, \dots, l, \dots, j_{n_2}) \cdot \cdots \cdot a_N(k_1,\dots, l, \dots, k_{n_N}).
\end{equation*}
The order of the array resulting from either incidence operation is:
\begin{equation*}
    \sum_{s=1}^Nn_s - N + 1.
\end{equation*}
Unless otherwise specified, an \textbf{incidence} of arrays will refer to a multiplicative incidence as defined above. The number of incident arrays $N$ is called the \textbf{arity} of the incidence. Unary incidences, that is, when only a single array is considered, i.e. $N=1$, correspond to the identity operation on arrays.\newline

Despite there being multiple ways to define array operations that combine semiring addition and multiplication, we shall restrict our attention to the operations that directly generalize binary relation composition and matrix multiplication. Consider two binary relations $a\subset X\times Y$ and $b\subset Y\times Z$, then the standard definition of composition of relations $c= b\circ a$ via the existence of intermediary elements can be formally expressed in terms of the operations of the Boolean algebra $(\mathbb{B},\vee,\wedge)$ as follows:
\begin{equation*}
    c_{xz}:= \bigvee_{y\in Y} a_{xy} \wedge b_{yz}.
\end{equation*}
Similarly, in the case of matrices over some ring $(R,+,\cdot)$, given a matrix $a_{ij}$ over indices $I\times J$ and a matrix $b_{jk}$ over indices $J\times K$, matrix multiplication is defined by the formula:
\begin{equation*}
    c_{ik}:=\sum_{j\in J} a_{ij}\cdot b_{jk}.
\end{equation*}
When relations and matrices are regarded as arrays over a semiring, the above formulas are formally identical. This operation for a pair of $2$-arrays, commonly called a contraction of indices in the physics literature \cite{wald2010general}, suggests the definition of the following more general operation: given a collection arrays
\begin{equation*}
    a_1,a_2,\dots,a_N
\end{equation*}
over constellations of indices
\begin{equation*}
    \{I_1,\dots, P, \dots,I_{n_1}\}, \{J_1,\dots, P, \dots, J_{n_2}\}, \dots, \{K_1,\dots, P, \dots, K_{n_N}\}
\end{equation*}
all sharing at least one common index $P$, a \textbf{contraction} is a new array whose entries are given by:
\begin{equation*}
    \sum_{p\in P} a_1(i_1,\dots, p, \dots,i_{n_1}) \cdot a_2(j_1, \dots, p, \dots, j_{n_2}) \cdot \cdots \cdot a_N(k_1,\dots, p, \dots, k_{n_N}).
\end{equation*}
The order of the array resulting from the contraction is:
\begin{equation*}
    \sum_{s=1}^Nn_s - N.
\end{equation*}
Similarly to incidences, we shall call the number arrays $N$ the \textbf{arity} of the contraction. Note that our definition of contraction generalizes the standard pair-wise contraction of tensor components commonly used in multilinear algebra \cite{wald2010general,abraham2012manifolds} and tensor networks \cite{biamonte2017tensor, okunishi2021developments}, and the pair-wise composition of relations \cite{burch1991peircean,behrisch2013relational}, by allowing for several arrays to be simultaneously contracted over a single common index. In particular, just as pair-wise contraction of indices is intimately connected to the notion of duality, i.e. distinctions of covariant and contravariant indices, manifesting at the level of $1$-arrays as the usual dot product of vectors
\begin{equation*}
    \sum_{p\in I} v_p \cdot w_p,
\end{equation*}
a ternary contraction of indices corresponds to the generally poorly understood notion of triality \cite{baez2001triality}, manifesting at the level of $1$-arrays as the ternary dot product of vectors
\begin{equation*}
    \sum_{p\in I} v_p \cdot w_p \cdot u_p.
\end{equation*}
Furthermore, our definition allows for unary contractions, i.e. $N=1$, where a single index of a $n$-array $a$ is summed over
\begin{equation*}
    \sum_{p\in I_k} a(i_1,\dots,i_{k-1}, p,i_{k+1} \dots, i_n)
\end{equation*}
to give a $(n-1)$-array. Unary contractions for $\mathbb{B}$-valued $2$-arrays correspond to the standard projections of binary relations, particularly, the domain and image of functions between sets. An obvious but important observation is that a general $N$-ary contraction of arrays reduces to a (multiplicative) $N$-ary incidence and a unary contraction of the incident index. Conceptually, the difference between an incidence and a contraction is whether the information of common indices is preserved or collapsed. This is clearly illustrated in the case of binary relations sharing a common set: given two binary relations $a\subset X\times Y$ and $b\subset Y\times Z$, the incidence and contraction operations define the relations
\begin{equation*}
    i_{xyz}:= a_{xy} \wedge b_{yz} \qquad \text{and} \qquad c_{xz}:= \bigvee_{y\in Y} a_{xy} \wedge b_{yz}.
\end{equation*}
respectively. The relation resulting from the incidence operation is a ternary relation $i\subset X \times Y \times Z$ collecting all the trios $(x,y,z)$ such that $y$ is an intermediary element between $x$ and $z$, whereas the relation resulting from the contraction is a binary relation $c\subset X\times Z$ collecting the pairs $(x,z)$ for which there exists at least one intermediary element between $x$ and $z$.\newline

Our definitions of array incidence and contraction may be slightly generalized to account for standard matrix operations such as traces or tensor products. We define \textbf{self-incidences} and \textbf{self-contractions} by allowing more than one index in each input array to be operated simultaneously in a single index incidence or contraction. A trace is then a two-index self contraction. The \textbf{tensor product} of arrays is simply defined by parallel multiplication, i.e. an incidence product with no common indices. For instance, the tensor product of a $1$-array $v$, a $2$-array $m$ and a $3$-array $a$ results in a $6$-array $t$ as follows:
\begin{equation*}
    t(i,j,k,l,n,m):=v(i)\cdot m(j,k)\cdot a(l,n,m).
\end{equation*}

Our choice of terminology for identity arrays is vindicated by their behaviour with respect to array operations. Firstly, identity arrays form incidence and contraction subalgebras, i.e. any incidence or contraction of several identity arrays gives an identity array of the appropriate order, and they act as neutral elements in contractions with generic arrays. To illustrate these properties we note:
\begin{align*}
    \delta_2(i,j) \cdot \delta_2(j,k) &=\delta_3(i,j,k)\\
    \sum_{p\in I} \delta_3(i,j,p)\cdot \delta_3(p,k,l) &= \delta_4(i,j,k,l)
\end{align*}
Secondly, identity arrays act as neutral elements with respect to contraction. For example, given an arbitrary $n$-array $a$, it follows that
\begin{equation*}
    \sum_{p\in I} \delta_2(j,p)\cdot a(i_1, \dots,p,\dots,i_n)=a(i_1, \dots,j,\dots,i_n).
\end{equation*}
Furthermore, contractions with identity arrays recover self-contractions:
\begin{equation*}
    \sum_{p\in I} \sum_{q\in I} \delta_2(p,q)\cdot a(i_1, \dots,p,\dots, q,\dots,i_n)= \sum_{p\in I} a(i_1, \dots,p,\dots, p,\dots,i_n).
\end{equation*}
Identity arrays define higher order \textbf{diagonal extensions} of arbitrary arrays via incidence or contraction. For example, given a $1$-array $v$, regarded as a component vector, it is easy to see that the $2$-array resulting from incidence with the identity $2$-array
\begin{equation*}
    m(i,j):=v(i) \cdot \delta_2(i,j)
\end{equation*}
corresponds to the components of a square diagonal matrix $m$ with diagonal entries given by the components of $v$. The array $m$ is called a diagonal extension of the array $v$. Unary contractions reverse diagonal extensions, as seen in the example above:
\begin{equation*}
    v(i)=\sum_{p\in I} v(p) \cdot \delta_2(p,i) = \sum_{p\in I} v(i) \cdot \delta_2(i,p).
\end{equation*}
Equivalently, the diagonal extension $m$ can be defined via contraction with the identity $3$-array
\begin{equation*}
    m(i,j)=\sum_{p} v(p) \cdot \delta_3(i,j,p)
\end{equation*}
which can be similarly reverted by unary contractions to give back the $1$-array $v$. Combining simultaneous contraction and incidence operations with identity arrays recovers self-incidences: for a $3$-array $a$ on indices $I\times I \times J$ we can readily check:
\begin{equation*}
    a(i,i,j)=\sum_{p\in I} \delta_2(i,p)\cdot a(i,p,k).
\end{equation*}

The distributivity property of the value semiring $(V,+,\cdot)$ ensures that more general operations involving multiple simultaneous incidences and contractions between several arrays may be constructed. Any well-defined procedure that outputs a single array from a given collection of arrays is called an array \textbf{multiplication} and the number of arrays involved in the multiplication is called its \textbf{arity}. Keeping track of the index structure of general multiplications quickly becomes a notational challenge as the order of arrays increases. The study of higher order array operations calls for a specific notational framework that can enable the systematic treatment of conformability conditions between arrays and the positions of common indices.

\section{The Plex Formalism} \label{plex}

In this section we shall develop the graphical notation that allows us to visualize array algebra operations. Our general approach is motivated by higher order generalizations of the notion of category via what we call `plex algebra'. Although the scope of the present work will only require a very limited use of plex algebra, we shall briefly outline the general ideas here and leave their concrete application to the study of associativity and semiheaps for later sections.\newline

We start by articulating the central notion of `plex'. Consider the array algebra setting presented at the start of Section \ref{array} where we take a universe of index sets $\mathcal{U}$ and a value semiring $(V,+,\cdot)$. Given a constellation of indices $I_1,\dots,I_n\in\mathcal{U}$, a \textbf{plex of order $n$} or \textbf{$n$-plex} is a formal assignment of values to each collection of $n$ separate index elements $\{i_1,\dots,i_n\}$. By using the braiding isomorphisms of the symmetric monoidal structure on the universe of index sets a plex can be defined as an equivalence class of an array under reorderings. Accordingly, `plex' is a slightly more elementary notion than `array', as it encodes the same entry-value assignment but without the information of the particular sequential order of its indices. Under this interpretation, plexes could also be called pre-sequential arrays. The choice of a representative array in a reordering equivalence class is called a \textbf{sequentialization} of a plex. Evidently, there are $|S_n|$ generically distinct sequentializations of a $n$-plex.\newline

The main motivation to de-emphasize array sequentiality is notational parsimony when describing general index operations. Our goal is to establish a diagrammatic notation for array multiplication that faithfully captures the information of \textbf{conformability}, i.e. minimal conditions of existence of common indices among the arrays being operated, and \textbf{topology}, i.e. the precise locations of the common indices. The standard set-theoretic notion of hypergraph \cite{bretto2013hypergraph} suggests a clear way to proceed: index sets are represented by vertices and $n$-plexes are represented by hyperedges on $n$ vertices. Conformability data is captured in a hypergraph and operations are specified by labels on the vertices and hyperedges of the hypergraph. Such representations are called \textbf{plex diagrams}. Let us articulate the details of this diagrammatic formalism.\newline

A plex is graphically depicted by a polytope whose the interior represents the entry-value assignment, typically with an alphabetic label that identifies the plex, and whose vertices represent the indices. Here are some examples of low order plexes:
\begin{center}
\begin{tikzpicture}[line join = round, line cap = round]

\begin{scope}[xshift=0, yshift=0]

\coordinate [label=left:$i$] (0) at (0.5,0.5);

\filldraw [lightgray] (0) circle (2pt);

\node[] at (0.7,0.7) {$v$};

\node[] at (0.65,-1.5) {$1$-plex};

\end{scope}

\begin{scope}[xshift=80, yshift=0]

\coordinate [label=right:$j$] (1) at (1.5,0.5);
\coordinate [label=left:$i$] (0) at (0.5,0.5);

\draw[opacity=0.5, lightgray, ultra thick] (1)--(0);

\node[] at (1,0.8) {$m$};

\node[] at (0.9,-1.5) {$2$-plex};

\end{scope}

\begin{scope}[xshift=200, yshift=0]

\coordinate [label=above:$k$] (2) at (0.5,{0.5*sqrt(3)});
\coordinate [label=right:$j$] (1) at (1,0);
\coordinate [label=left:$i$] (0) at (0,0);

\draw[lightgray, opacity=0.5,fill=lightgray,fill opacity=0.5] (1)--(0)--(2)--cycle;

\node[] at (0.5,0.3) {$a$};

\node[] at (0.45,-1.5) {$3$-plex};

\end{scope}

\begin{scope}[xshift=330, yshift=0]

\coordinate [label=above:$l$] (3) at (0,{sqrt(2)},0);
\coordinate [label=left:$i$] (2) at ({-.5*sqrt(3)},0,-.5);
\coordinate [label=below:$j$] (1) at (0,0,1);
\coordinate [label=right:$k$] (0) at ({.5*sqrt(3)},0,-.5);

\draw[densely dotted,postaction={decorate}] (0)--(2);
\draw[fill=lightgray,fill opacity=.5] (1)--(0)--(3)--cycle;
\draw[fill=gray,fill opacity=.5] (2)--(1)--(3)--cycle;
\draw[gray, opacity=0.5, postaction={decorate}] (1)--(0);
\draw[gray, opacity=0.5, postaction={decorate}] (1)--(2);
\draw[gray, opacity=0.5, postaction={decorate}] (2)--(3);
\draw[gray, opacity=0.5, postaction={decorate}] (1)--(3);
\draw[gray, opacity=0.5, postaction={decorate}] (0)--(3);

\node[] at (0.2,0.5) {$c$};

\node[] at (0.1,-1.5) {$4$-plex};

\end{scope}

\end{tikzpicture}
\end{center}
Note that only interiors and vertices of polytopes represent relevant information in plex diagrams, highlighted edges or faces, like in the $4$-plex above, will only be used as visual aid. When working with low order plexes ($n\leq 3$) we use polytope representations that correspond to topological simplices, as shown in the four examples above where a $n$-plex is graphically depicted as a $(n-1)$-simplex. Due to the inherent limitations of $2$- and $3$-dimensional graphics, non-simplicial representations are preferred for plexes of order $4$ and higher. For instance, a $4$-plex $c$ can be represented in two equivalent ways:
\begin{center}
\begin{tikzpicture}[line join = round, line cap = round]

\begin{scope}[xshift=0, yshift=0]

\coordinate [label=above:$l$] (3) at (0,{sqrt(2)},0);
\coordinate [label=left:$i$] (2) at ({-.5*sqrt(3)},0,-.5);
\coordinate [label=below:$j$] (1) at (0,0,1);
\coordinate [label=right:$k$] (0) at ({.5*sqrt(3)},0,-.5);

\draw[densely dotted,postaction={decorate}] (0)--(2);
\draw[fill=lightgray,fill opacity=.5] (1)--(0)--(3)--cycle;
\draw[fill=gray,fill opacity=.5] (2)--(1)--(3)--cycle;
\draw[gray, opacity=0.5, postaction={decorate}] (1)--(0);
\draw[gray, opacity=0.5, postaction={decorate}] (1)--(2);
\draw[gray, opacity=0.5, postaction={decorate}] (2)--(3);
\draw[gray, opacity=0.5, postaction={decorate}] (1)--(3);
\draw[gray, opacity=0.5, postaction={decorate}] (0)--(3);

\node[] at (0.2,0.5) {$c$};

\end{scope}

\node[] at (2.4,0.5) {$\equiv$};

\begin{scope}[xshift=100, yshift=0]

\coordinate [label=above left:$l$] (3) at (0,1.2);
\coordinate [label=above right:$i$] (2) at (1.2,1.2);
\coordinate [label=below right:$j$] (1) at (1.2,0);
\coordinate [label=below left:$k$] (0) at (0,0);

\draw[lightgray, opacity=0.5, fill=lightgray,fill opacity=.5] (3)--(2)--(1)--(0)--cycle;

\node[] at (0.6,0.55) {$c$};

\end{scope}

\end{tikzpicture}
\end{center}
Labels will be used flexibly: they will often be omitted and, depending on the particular emphasis of the plex diagram, vertices may indistinctly carry index element labels or index set labels.\newline

Index manipulations can be conveniently expressed by graphically altering plexes. Order-decreasing operations such as \textbf{flattening} and \textbf{slicing} are expressed by collapsing a plex along some of its vertices or edges, respectively, and order-increasing operations such as \textbf{broadening} are expressed by inscribing a plex onto a \emph{ghost} plex of a larger order. Here is a diagrammatic example of a $3$-plex flattening into a $2$-plex and a $2$-plex broadened into a $3$-plex:
\begin{center}
\begin{tikzpicture}[line join = round, line cap = round]

\begin{scope}[xshift=0, yshift=0]

\coordinate [] (2) at (0.5,{0.5*sqrt(3)});
\coordinate [] (1) at (1,0);
\coordinate [] (0) at (0,0);

\draw[lightgray, opacity=0.5, fill=lightgray, fill opacity=0.5] (1)--(0)--(2)--cycle;

\end{scope}

\begin{scope}[xshift=70, yshift=0]

\coordinate [] (2) at (0.5,{0.5*sqrt(3)});
\coordinate [] (1) at (1,0);
\coordinate [] (0) at (0,0);

\draw[opacity=0.5, lightgray, ultra thick] (2)--(0);

\end{scope}

\begin{scope}[xshift=160, yshift=0]

\coordinate [] (2) at (0.5,{0.5*sqrt(3)});
\coordinate [] (1) at (1,0);
\coordinate [] (0) at (0,0);

\draw[opacity=0.5, lightgray, ultra thick] (2)--(0);

\end{scope}

\begin{scope}[xshift=220, yshift=0]

\coordinate [] (2) at (0.5,{0.5*sqrt(3)});
\coordinate [] (1) at (1,0);
\coordinate [] (0) at (0,0);

\draw[lightgray, dashed] (1)--(0)--(2)--cycle;
\draw[opacity=0.5, lightgray, ultra thick] (2)--(0);

\end{scope}

\node[] at (1.75,0.3) {$\mapsto$};
\node[] at (7.05,0.3) {$\mapsto$};

\end{tikzpicture}
\end{center}

The array operations introduced in Section \ref{array} generalize to plexes in a natural way via quotienting by reordering isomorphisms. An operation defined on plexes as the pre-sequential manifestation of some array multiplication will be called a \textbf{plex product} and the number of input plexes will be called the \textbf{arity} of the plex product. A plex product that is order-preserving, i.e. $n$-plex inputs give a $n$-plex output, will sometimes be called a \textbf{plex composition}. For the remainder of this section we show how plex diagrams capture plex products.\newline

We begin by considering entry-wise operations, which can be simply denoted at the level of labels. For example, \textbf{entry-wise addition} of $2$-arrays can be written as
\begin{center}
\begin{tikzpicture}[line join = round, line cap = round]

\coordinate [label=below left:$j$] (1) at (0,0);
\coordinate [label=above right:$i$] (0) at (1,1);
\coordinate [label=below left:$j$] (2) at (2,0);
\coordinate [label=above right:$i$] (3) at (3,1);
\coordinate [label=below left:$j$] (4) at (5,0);
\coordinate [label=above right:$i$] (5) at (6,1);

\draw[opacity=0.5, lightgray, ultra thick] (1)--(0);
\draw[opacity=0.5, lightgray, ultra thick] (2)--(3);

\node[] at (1,0.4) {$m$};

\node[] at (3,0.4) {$n$};

\node[] at (4.2,0.4) {$\mapsto$};

\draw[opacity=0.5, lightgray, ultra thick] (4)--(5);

\node[] at (6.5,0.4) {$m+n$};

\end{tikzpicture}
\end{center}
similarly, \textbf{entry-wise multiplication} of $3$-arrays can be written as
\begin{center}
\begin{tikzpicture}[line join = round, line cap = round]

\coordinate [label=above:$k$] (2) at (0.5,{0.5*sqrt(3)});
\coordinate [label=right:$j$] (1) at (1,0);
\coordinate [label=left:$i$] (0) at (0,0);

\coordinate [label=above:$k$] (5) at (3,{0.5*sqrt(3)});
\coordinate [label=right:$j$] (4) at (3.5,0);
\coordinate [label=left:$i$] (3) at (2.5,0);

\coordinate [label=above:$k$] (8) at (6+0.3,{0.5*sqrt(3)});
\coordinate [label=right:$j$] (7) at (6.5+0.3,0);
\coordinate [label=left:$i$] (6) at (5.5+0.3,0);

\draw[lightgray, opacity=0.5, fill=lightgray, fill opacity=0.5] (1)--(0)--(2)--cycle;
\draw[lightgray, opacity=0.5, fill=lightgray, fill opacity=0.5] (4)--(3)--(5)--cycle;
\draw[lightgray, opacity=0.5, fill=lightgray, fill opacity=0.5] (7)--(6)--(8)--cycle;

\node[] at (0.5,0.3) {$a$};

\node[] at (3,0.3) {$b$};

\node[] at (4.7,0.3) {$\mapsto$};

\node[] at (5.97+0.3,0.3) {$a \! \cdot \! b$};

\node[] at (8.5,0.3) { };

\end{tikzpicture}
\end{center}
The main use case for plex diagrams are general array multiplications which may involve several simultaneous incidences and contractions. The notational challenge that we face is to faithfully represent common indices among a given collection of plexes or arrays. Our approach is to portray index incidence data as the adjacency data of a hypergraph whose hyperedges represent plexes and whose vertices represent indices. Diagrammatically, we shall render hypergraphs as collections of juxtaposed polytopes where bare vertices, i.e. not carrying any graphical markings, shall represent distinct indices and polytope interiors shall represent plexes. Following this recipe, the \textbf{incidence} operation is represented simply by attaching plexes by their common indices. For instance, given a $2$-plex $m$ and a $3$-plex $a$ with one index in common their incidence product is written as
\begin{center}
\begin{tikzpicture}[line join = round, line cap = round]

\coordinate [label=below left:$i$] (1) at (0,0);
\coordinate [label=above right:$j$] (0) at (1,1);

\coordinate [label=above:$l$] (4) at (3,{0.5*sqrt(3)});
\coordinate [label=right:$k$] (3) at (3.5,0);
\coordinate [label=left:$j$] (2) at (2.5,0);

\coordinate [label=above:$l$] (8) at (7.366,0.8);
\coordinate [label=right:$k$] (7) at (7.366,-0.2);
\coordinate [label=below:$j$] (6) at (6.5,0.3);
\coordinate [label=left:$i$] (5) at (5.6,0.3);

\draw[lightgray, opacity=0.5, ultra thick] (1)--(0);

\node[] at (0.8,0.35) {$m$};

\draw[lightgray, opacity=0.5, fill=lightgray,fill opacity=0.5] (2)--(3)--(4)--cycle;

\node[] at (3,0.3) {$a$};

\node[] at (4.7,0.3) {$\mapsto$};

\draw[lightgray, opacity=0.5, ultra thick] (5)--(6);
\draw[lightgray, opacity=0.5, fill=lightgray,fill opacity=0.5] (7)--(6)--(8)--cycle;

\node[] at (6,0.5) {$m$};
\node[] at (7.05,0.3) {$a$};

\end{tikzpicture}
\end{center}
the output is a $4$-plex, as explicitly indicated by the four vertex labels. To recover an incidence array operation from this plex diagram we must choose sequentializations for the plexes $m$ and $a$. By setting $m_{ij}$ as a $2$-array and $a_{jlk}$ as a $3$-array we have the following equivalence:
\begin{center}
\begin{tikzpicture}[line join = round, line cap = round]

\coordinate [label=above:$l$] (4) at (1.366,0.8);
\coordinate [label=right:$k$] (3) at (1.366,-0.2);
\coordinate [label=below:$j$] (2) at (0.5,0.3);
\coordinate [label=left:$i$] (1) at (-0.3,0.3);

\draw[lightgray, opacity=0.5, ultra thick] (1)--(2);
\draw[lightgray, opacity=0.5, fill=lightgray,fill opacity=0.5] (2)--(3)--(4)--cycle;

\node[] at (0.11,0.5) {$m$};
\node[] at (1.05,0.3) {$a$};

\node[] at (3,0.3) {$=$};

\node[] at (4.6,0.26) {$m_{ij}\cdot a_{jlk}$};

\end{tikzpicture}
\end{center}
If no sequentializations are specified, the plex diagram alone encapsulates as many sequentially inequivalent array operations as there are index permutations of each array involved in the operation
\begin{center}
\begin{tikzpicture}[line join = round, line cap = round]

\coordinate [label=above:$l$] (4) at (1.366,0.8);
\coordinate [label=right:$k$] (3) at (1.366,-0.2);
\coordinate [label=below:$j$] (2) at (0.5,0.3);
\coordinate [label=left:$i$] (1) at (-0.3,0.3);

\draw[lightgray, opacity=0.5, ultra thick] (1)--(2);
\draw[lightgray, opacity=0.5, fill=lightgray,fill opacity=0.5] (2)--(3)--(4)--cycle;

\node[] at (0.15,0.5) {$m$};
\node[] at (1.05,0.3) {$a$};

\node[] at (2.3,0.3) {$=$};

\node[] at (7,0.3) {$\{m_{ij}\cdot a_{jlk} \,,\, m_{ji}\cdot a_{jlk} \,,\, m_{ij}\cdot a_{jkl} \,,\, m_{ij}\cdot a_{ljk} \,,\, \dots\}$};

\end{tikzpicture}
\end{center}
The $12=|S_2|\cdot |S_3|$ sequentially inequivalent array operations in brackets above all share two features: conformability, i.e. the fact that there is one common index between $m$ and $a$, and topology, i.e. the fact that a single index of $m$ is incident with a single index of $a$; this is precisely the data that plex diagrams aim to encode. Here are some examples of plex diagrams representing multiple incidence operations together with their corresponding array expressions in a particular sequentialization, and the total order of the resulting plex:
\begin{center}
\begin{tikzpicture}[line join = round, line cap = round]

\begin{scope}[xshift=0, yshift=0]

\coordinate [label=above right:$u$] (3) at ({0.5*sqrt(3)},-0.5);
\coordinate [label=below right:$n$] (2) at (-{0.5*sqrt(3)},-0.5);
\coordinate [label=below left:$m$] (1) at (0,1);
\coordinate [] (0) at (0,0);

\draw[opacity=0.5, lightgray, ultra thick] (1)--(0);
\draw[opacity=0.5, lightgray, ultra thick] (2)--(0);
\draw[opacity=0.5, lightgray, ultra thick] (3)--(0);

\node[] at (0,-1.5) {$m_{il}\cdot n_{jl}\cdot u_{kl}$};
\node[] at (0,-2.3) {$4$-plex};

\end{scope}

\begin{scope}[xshift=100, yshift=8]

\coordinate [] (3) at (-1,0.5);
\coordinate [] (2) at (1,0.5);
\coordinate [] (1) at (0,-1);
\coordinate [] (0) at (0,0);

\draw[gray, opacity=0.5,fill=lightgray,fill opacity=0.5] (1)--(0)--(2)--cycle;
\draw[gray, opacity=0.5,fill=lightgray,fill opacity=0.5] (1)--(0)--(3)--cycle;
\draw[opacity=0.5, lightgray, ultra thick] (3)--(2);

\node[] at (0,0.7) {$m$};
\node[] at (-0.31,-0.15) {$a$};
\node[] at (0.33,-0.1) {$b$};

\node[] at (0,-1.75) {$m_{ij}\cdot a_{ipq}\cdot b_{pqj}$};
\node[] at (0,-2.56) {$4$-plex};

\end{scope}

\begin{scope}[xshift=230, yshift=0]

\coordinate [] (4) at (1.8,0.4);
\coordinate [] (5) at (2,-0.7);
\coordinate [] (3) at ({0.5*sqrt(3)},-0.5);
\coordinate [] (2) at (-{0.5*sqrt(3)},-0.5);
\coordinate [] (1) at (0,1);
\coordinate [] (0) at (0,0);

\draw[gray, opacity=0.5,fill=lightgray,fill opacity=0.5] (1)--(0)--(2)--cycle;
\draw[gray, opacity=0.5,fill=lightgray,fill opacity=0.5] (2)--(0)--(3)--cycle;
\draw[gray, opacity=0.5,fill=lightgray,fill opacity=0.5] (1)--(0)--(3)--cycle;

\draw[opacity=0.5, lightgray, ultra thick] (1) to[out=210, in=90] (2);
\draw[opacity=0.5, lightgray, ultra thick] (3)--(4);
\draw[opacity=0.5, lightgray, ultra thick] (3)--(5);

\node[] at (-0.22,0.17) {$a$};
\node[] at (0.2,0.2) {$b$};
\node[] at (0,-0.3) {$c$};
\node[] at (1.15,0.2) {$m$};
\node[] at (1.6,-0.4) {$n$};
\node[] at (-0.9,0.5) {$u$};

\node[] at (0,-1.5) {$u_{ij}\cdot a_{ijl}\cdot b_{ilk} \cdot c_{jlk} \cdot m_{kp} \cdot n_{kq}$};
\node[] at (0.1,-2.3) {$6$-plex};

\end{scope}

\end{tikzpicture}
\end{center}
Note that index labels at the vertices could be omitted without loss of information, this will always be the case as long as plex diagrams depict simple hypergraphs, i.e. when no two distinct hyperedges are allowed to be defined by the same set of vertices. Unless otherwise indicated, we shall always consider plex diagrams of this kind.\newline

Following our observation in Section \ref{array} that a general array contraction can be specified by an incidence followed by a unary contraction on the incident index, a \textbf{contraction} of plexes is depicted simply by adding a marking on the incident index to be contracted. For example, the contraction product of a $2$-plex $m$ and a $3$-plex $a$ with one index in common is written as
\begin{center}
\begin{tikzpicture}[line join = round, line cap = round]

\coordinate [label=below left:$i$] (1) at (0,0);
\coordinate [label=above right:$j$] (0) at (1,1);

\coordinate [label=above:$l$] (4) at (3,{0.5*sqrt(3)});
\coordinate [label=right:$k$] (3) at (3.5,0);
\coordinate [label=left:$j$] (2) at (2.5,0);

\coordinate [label=above:$l$] (8) at (7.366,0.8);
\coordinate [label=right:$k$] (7) at (7.366,-0.2);
\coordinate [label=below:$j$] (6) at (6.5,0.3);
\coordinate [label=left:$i$] (5) at (5.6,0.3);

\draw[lightgray, opacity=0.5, ultra thick] (1)--(0);

\node[] at (0.8,0.35) {$m$};

\draw[lightgray, opacity=0.5, fill=lightgray,fill opacity=0.5] (2)--(3)--(4)--cycle;

\node[] at (3,0.3) {$a$};

\node[] at (4.7,0.3) {$\mapsto$};

\draw[lightgray, opacity=0.5, ultra thick] (5)--(6);
\draw[lightgray, opacity=0.5, fill=lightgray,fill opacity=0.5] (7)--(6)--(8)--cycle;

\filldraw[darkgray] (6) circle (0.07);

\node[] at (6,0.5) {$m$};
\node[] at (7.05,0.3) {$a$};

\end{tikzpicture}
\end{center}
and by choosing a sequentialization we get
\begin{center}
\begin{tikzpicture}[line join = round, line cap = round]

\coordinate [label=above:$l$] (4) at (1.366,0.8);
\coordinate [label=right:$k$] (3) at (1.366,-0.2);
\coordinate [label=below:$j$] (2) at (0.5,0.3);
\coordinate [label=left:$i$] (1) at (-0.3,0.3);

\draw[lightgray, opacity=0.5, ultra thick] (1)--(2);
\draw[lightgray, opacity=0.5, fill=lightgray,fill opacity=0.5] (2)--(3)--(4)--cycle;

\filldraw[darkgray] (2) circle (0.07);

\node[] at (0.11,0.5) {$m$};
\node[] at (1.05,0.3) {$a$};

\node[] at (2.7,0.3) {$=$};

\node[] at (4.6,0.19) {$\displaystyle \sum_j m_{ij}\cdot a_{jlk}$};

\end{tikzpicture}
\end{center}
The plex diagram of the contraction product encodes all the sequentially inequivalent array contractions resulting from the different choices of sequentializations of $m$ and $a$. Contraction, indicated by the dot marking on a vertex, effectively removes an index in a plex diagram. This contrasts with incidence which preserves all indices of the input plexes. This can be clearly illustrated by the following diagrammatic equations for $2$-plex products:
\begin{center}
\begin{tikzpicture}[line join = round, line cap = round]

\begin{scope}[xshift=0, yshift=0]

\coordinate [] (2) at (0.5,{0.5*sqrt(3)});
\coordinate [] (1) at (1,0);
\coordinate [] (0) at (0,0);

\draw[lightgray, opacity=0.5, ultra thick] (0)--(2);
\draw[lightgray, opacity=0.5, ultra thick] (2)--(1);

\filldraw[darkgray] (2) circle (0.07);

\end{scope}

\begin{scope}[xshift=50, yshift=0]

\node[] at (0,0.3) {$=$};

\end{scope}

\begin{scope}[xshift=70, yshift=0]

\coordinate [] (2) at (0.5,{0.5*sqrt(3)});
\coordinate [] (1) at (1,0);
\coordinate [] (0) at (0,0);

\draw[lightgray, opacity=0.5, ultra thick] (0)--(1);

\end{scope}

\begin{scope}[xshift=190, yshift=0]

\coordinate [] (2) at (0.5,{0.5*sqrt(3)});
\coordinate [] (1) at (1,0);
\coordinate [] (0) at (0,0);

\draw[lightgray, opacity=0.5, ultra thick] (0)--(2);
\draw[lightgray, opacity=0.5, ultra thick] (2)--(1);

\end{scope}

\begin{scope}[xshift=240, yshift=0]

\node[] at (0,0.3) {$=$};

\end{scope}

\begin{scope}[xshift=260, yshift=0]

\coordinate [] (2) at (0.5,{0.5*sqrt(3)});
\coordinate [] (1) at (1,0);
\coordinate [] (0) at (0,0);

\draw[lightgray, opacity=0.5,fill=lightgray,fill opacity=0.5] (1)--(0)--(2)--cycle;

\end{scope}

\end{tikzpicture}
\end{center}
Unary contractions are simply denoted by a plex with a marked vertex, as illustrated by the following diagrammatic equations:
\begin{center}
\begin{tikzpicture}[line join = round, line cap = round]

\begin{scope}[xshift=0, yshift=0]

\coordinate [] (1) at (1,0.3);
\coordinate [] (0) at (0,0.3);

\draw[lightgray, opacity=0.5, ultra thick] (0)--(1);

\filldraw[darkgray] (1) circle (0.07);

\end{scope}

\begin{scope}[xshift=50, yshift=0]

\node[] at (0,0.3) {$=$};

\end{scope}

\begin{scope}[xshift=70, yshift=0]

\coordinate [] (0) at (0,0.3);

\filldraw[lightgray, opacity=0.5] (0) circle (0.09);

\end{scope}

\begin{scope}[xshift=190, yshift=0]

\coordinate [] (2) at (0.5,{0.5*sqrt(3)});
\coordinate [] (1) at (1,0);
\coordinate [] (0) at (0,0);

\draw[lightgray, opacity=0.5,fill=lightgray,fill opacity=0.5] (1)--(0)--(2)--cycle;

\filldraw[darkgray] (1) circle (0.07);

\end{scope}

\begin{scope}[xshift=240, yshift=0]

\node[] at (0,0.3) {$=$};

\end{scope}

\begin{scope}[xshift=260, yshift=0]

\coordinate [] (2) at (0.5,{0.5*sqrt(3)});
\coordinate [] (1) at (1,0);
\coordinate [] (0) at (0,0);

\draw[lightgray, opacity=0.5, ultra thick] (0)--(2);

\end{scope}

\end{tikzpicture}
\end{center}

A general plex product is thus uniquely specified by a plex diagram corresponding to a simple connected hypergraph with vertex markings. The order of the resulting plex is easy to assess at a glance since the indices correspond to the unmarked vertices i.e. unaltered and incident indices, which we call the \textbf{free indices} of the plex diagram. Sometimes it will also be useful to distinguish between the indices that are operated in some way, i.e. incident or contracted, from the unaltered ones, the former are called \textbf{internal indices} and the latter are called \textbf{external indices}. Here are some plex diagrams representing general plex products together with their corresponding array expressions in a particular sequentialization, and the total order of the resulting plex:
\begin{center}
\begin{tikzpicture}[line join = round, line cap = round]

\begin{scope}[xshift=0, yshift=0]

\coordinate [] (3) at ({0.5*sqrt(3)},-0.5);
\coordinate [] (2) at (-{0.5*sqrt(3)},-0.5);
\coordinate [] (1) at (0,1);
\coordinate [] (0) at (0,0);

\draw[gray, opacity=0.5,fill=lightgray,fill opacity=0.5] (1)--(0)--(2)--cycle;
\draw[gray, opacity=0.5,fill=lightgray,fill opacity=0.5] (2)--(0)--(3)--cycle;
\draw[gray, opacity=0.5,fill=lightgray,fill opacity=0.5] (1)--(0)--(3)--cycle;

\filldraw[darkgray] (0) circle (0.07);

\node[] at (-0.22,0.17) {$a$};
\node[] at (0.2,0.2) {$b$};
\node[] at (0,-0.3) {$c$};

\node[] at (0,-1.8) {$\displaystyle \sum_p a_{ijp} \cdot b_{ipk} \cdot c_{pjk}$};
\node[] at (0,-2.6) {$3$-plex};

\end{scope}

\begin{scope}[xshift=140, yshift=0]

\coordinate [] (7) at (1,{-0.5*(1+sqrt(3))});
\coordinate [] (6) at (0,{-0.5*(1+sqrt(3))});
\coordinate [] (5) at (-1,{0.5*(1+sqrt(3))});
\coordinate [] (4) at (0,{0.5*(1+sqrt(3))});
\coordinate [] (3) at (1.5,0.5);
\coordinate [] (2) at (0.5,-0.5);
\coordinate [] (1) at (-0.5,0.5);
\coordinate [] (0) at (-1.5,-0.5);

\draw[opacity=0.5, lightgray, ultra thick] (0)--(1);
\draw[opacity=0.5, lightgray, ultra thick] (1)--(2);
\draw[opacity=0.5, lightgray, ultra thick] (2)--(3);

\draw[gray, opacity=0.5,fill=lightgray,fill opacity=0.5] (1)--(4)--(5)--cycle;
\draw[gray, opacity=0.5,fill=lightgray,fill opacity=0.5] (2)--(6)--(7)--cycle;

\filldraw[darkgray] (1) circle (0.07);
\filldraw[darkgray] (2) circle (0.07);

\node[] at (-1.3,0.2) {$m$};
\node[] at (-0.25,-0.1) {$n$};
\node[] at (0.8,0.2) {$u$};
\node[] at (-0.5,{0.5*(1+0.65*sqrt(3))}) {$a$};
\node[] at (0.5,{-0.5*(1+0.65*sqrt(3))}) {$b$};

\node[] at (0,-1.8) {$\displaystyle \sum_{p} \sum_{q} m_{ip}\cdot a_{kpl}\cdot n_{pq} \cdot b_{sqr} \cdot u_{qj}$};
\node[] at (0,-2.6) {$6$-plex};

\end{scope}

\begin{scope}[xshift=320, yshift=0]

\coordinate [] (6) at (0,0);
\coordinate [] (5) at (2,0);
\coordinate [] (4) at (1,-1);
\coordinate [] (3) at (1,1);
\coordinate [] (2) at (-1,-1);
\coordinate [] (1) at (-1,1);
\coordinate [] (0) at (-2,0);

\draw[gray, opacity=0.5,fill=lightgray,fill opacity=0.5] (1)--(0)--(2)--cycle;
\draw[gray, opacity=0.5,fill=lightgray,fill opacity=0.5] (4)--(6)--(3)--cycle;
\draw[gray, opacity=0.5,fill=lightgray,fill opacity=0.5] (5)--(4)--(3)--cycle;

\draw[opacity=0.5, lightgray, ultra thick] (1)--(6);
\draw[opacity=0.5, lightgray, ultra thick] (2)--(6);

\filldraw[darkgray] (3) circle (0.07);
\filldraw[darkgray] (4) circle (0.07);
\filldraw[darkgray] (6) circle (0.07);

\node[] at (-0.25,0.7) {$m$};
\node[] at (-0.25,-0.7) {$n$};
\node[] at (-1.4,0) {$a$};
\node[] at (0.65,0.04) {$b$};
\node[] at (1.35,-0.01) {$c$};

\node[] at (0,-1.8) {$\displaystyle \sum_{p} \sum_{q} \sum_{r} a_{ijk} \cdot m_{jp}\cdot n_{kp} \cdot b_{pqr} \cdot c_{qrl}$};
\node[] at (0,-2.6) {$4$-plex};

\end{scope}

\end{tikzpicture}
\end{center}
We should remark at this point that the fact that plex diagrams faithfully encode general plex products (or array multiplications up to sequentialization) is directly dependent on the algebraic properties of the value commutative semiring $(V,+,\cdot)$. In particular, we note that all the defining axioms of associativity, commutativity and distributivity of the binary operations in $(V,+,\cdot)$ are required for a general plex product to be uniquely defined from a plex diagram.\newline

Distinct vertices in a plex represent separate indices but they may belong to the same index set. In the extreme case of a constellation consisting of multiple copies of the same index set, the reordering class of a regular array defines a \textbf{regular plex}. When a plex has repeated index sets in its constellation we use vertex labels as index sets or index values interchangeably to emphasize domain or index element names respectively. For instance, given a regular $4$-plex $c$ on an index set $I$ we write both
\begin{center}
\begin{tikzpicture}[line join = round, line cap = round]

\begin{scope}[xshift=0, yshift=0]

\coordinate [label=above:$I$] (3) at (0,{sqrt(2)},0);
\coordinate [label=left:$I$] (2) at ({-.5*sqrt(3)},0,-.5);
\coordinate [label=below:$I$] (1) at (0,0,1);
\coordinate [label=right:$I$] (0) at ({.5*sqrt(3)},0,-.5);

\draw[densely dotted,postaction={decorate}] (0)--(2);
\draw[fill=lightgray,fill opacity=.5] (1)--(0)--(3)--cycle;
\draw[fill=gray,fill opacity=.5] (2)--(1)--(3)--cycle;
\draw[gray, opacity=0.5, postaction={decorate}] (1)--(0);
\draw[gray, opacity=0.5, postaction={decorate}] (1)--(2);
\draw[gray, opacity=0.5, postaction={decorate}] (2)--(3);
\draw[gray, opacity=0.5, postaction={decorate}] (1)--(3);
\draw[gray, opacity=0.5, postaction={decorate}] (0)--(3);

\node[] at (0.2,0.5) {$c$};

\end{scope}

\node[] at (2.65,0.5) {and};

\begin{scope}[xshift=140, yshift=0]

\coordinate [label=above:$l$] (3) at (0,{sqrt(2)},0);
\coordinate [label=left:$i$] (2) at ({-.5*sqrt(3)},0,-.5);
\coordinate [label=below:$j$] (1) at (0,0,1);
\coordinate [label=right:$k$] (0) at ({.5*sqrt(3)},0,-.5);

\draw[densely dotted,postaction={decorate}] (0)--(2);
\draw[fill=lightgray,fill opacity=.5] (1)--(0)--(3)--cycle;
\draw[fill=gray,fill opacity=.5] (2)--(1)--(3)--cycle;
\draw[gray, opacity=0.5, postaction={decorate}] (1)--(0);
\draw[gray, opacity=0.5, postaction={decorate}] (1)--(2);
\draw[gray, opacity=0.5, postaction={decorate}] (2)--(3);
\draw[gray, opacity=0.5, postaction={decorate}] (1)--(3);
\draw[gray, opacity=0.5, postaction={decorate}] (0)--(3);

\node[] at (0.2,0.5) {$c$};

\end{scope}

\end{tikzpicture}
\end{center}
Plex diagrams representing products of regular plexes (or any plex product involving several copies of the same index) are subject to additional freedom and several generically distinct algebraic operations can be defined from the same plex diagram. This can be illustrated by the following contraction of a pair of $3$-plexes:
\begin{center}
\begin{tikzpicture}[]

\begin{scope}[xshift=0, yshift=0]

\begin{scope}[scale=0.7, xshift=0, yshift=0]

\coordinate [label=right:$K$] (6) at ({3*sqrt(3)},0);
\coordinate [label=below:$J$] (5) at ({2*sqrt(3)},-1);
\coordinate [label=above:$J$] (4) at ({2*sqrt(3)},1);
\coordinate [label=left:$I$] (3) at ({sqrt(3)},0);
\coordinate [] (2) at (0,-1);
\coordinate [] (1) at (0,1);
\coordinate [] (0) at (0,0);

\draw[gray, opacity=0.5,fill=lightgray,fill opacity=0.5] (3)--(4)--(5)--cycle;
\draw[gray, opacity=0.5,fill=lightgray,fill opacity=0.5] (4)--(5)--(6)--cycle;

\end{scope}

\filldraw[darkgray] (5) circle (0.07);
\filldraw[darkgray] (4) circle (0.07);

\node[] at ({0.7*(2*sqrt(3)-0.6)},0) {$a$};
\node[] at ({0.7*(2*sqrt(3)+0.6)},0.03) {$b$};

\end{scope}

\end{tikzpicture}
\end{center}
Since the $3$-plexes share two indices on the same index set $J$, the plex diagram does not faithfully capture which of the two $J$ indices of each plex are being contracted. In this case, there are two possibilities, which appear as a permutation of contracted indices when expressed in array notation after choosing a sequentialization:
\begin{equation*}
    \sum_{p,q\in J} a_{ipq} \cdot b_{pqk} \qquad \text{ or } \qquad \sum_{p,q\in J} a_{iqp} \cdot b_{pqk}.
\end{equation*}
In general, a plex diagram involving $n$ copies of the same index set in incidences or contractions defines $|S_n|$ generically distinct plex multiplications called \textbf{twistings}.\newline

The identity arrays $\delta_n$ defined in Section \ref{array} generalize to regular $n$-plexes in the obvious way. Identity plexes are special in that they are in 1-to-1 correspondence with their sequentializations since the identity arrays $\delta_n$ are totally symmetric and thus they are the only members of their reordering classes. Following the category theory convention to represent unique identity morphisms as long equals signs, we write the identity plexes as follows:
\begin{center}
\begin{tikzpicture}[]

\begin{scope}[xshift=0, yshift=0]

\coordinate [] (1) at (1.5,0.5);
\coordinate [] (0) at (0.5,0.5);

\draw[lightgray, line width=1pt, double] (1)--(0);

\node[] at (1,-1.1) {$\delta_2$};

\end{scope}

\begin{scope}[xshift=100, yshift=0]

\coordinate [] (2a) at (0.5+0.03,{0.5*sqrt(3)});
\coordinate [] (2b) at (0.5-0.03,{0.5*sqrt(3)});
\coordinate [] (1a) at ({1+0.03*0.5},{0.03*0.5*sqrt(3)});
\coordinate [] (1b) at ({1-0.03*0.5},{-0.03*0.5*sqrt(3)});
\coordinate [] (0a) at ({0.03*0.5},{-0.03*0.5*sqrt(3)});
\coordinate [] (0b) at ({-0.03*0.5},{0.03*0.5*sqrt(3)});

\draw[lightgray, line width=1pt] (0b) to[out=30, in=270] (2b);
\draw[lightgray, line width=1pt] (2a) to[out=270, in=150] (1a);
\draw[lightgray, line width=1pt] (1b) to[out=150, in=30] (0a);

\node[] at (0.5,-1.1) {$\delta_3$};

\end{scope}

\begin{scope}[xshift=190, yshift=0]

\coordinate [] (3a) at (0+0.025,1+0.025);
\coordinate [] (3b) at (0-0.025,1-0.025);
\coordinate [] (2a) at (1+0.025,1-0.025);
\coordinate [] (2b) at (1-0.025,1+0.025);
\coordinate [] (1a) at (1+0.025,0+0.025);
\coordinate [] (1b) at (1-0.025,0-0.025);
\coordinate [] (0a) at (0+0.025,0-0.025);
\coordinate [] (0b) at (0-0.025,0+0.025);

\draw[lightgray, line width=1pt] (0b) to[out=45, in=315] (3b);
\draw[lightgray, line width=1pt] (3a) to[out=315, in=225] (2b);
\draw[lightgray, line width=1pt] (2a) to[out=225, in=135] (1a);
\draw[lightgray, line width=1pt] (1b) to[out=135, in=45] (0a);

\node[] at (0.6,-1.1) {$\delta_4$};

\end{scope}

\begin{scope}[xshift=280, yshift=0]

\node[] at (0,0.5) {$\cdots$};

\end{scope}

\end{tikzpicture}
\end{center}
Identity plexes form a subalgebra under incidence and contraction. Diagrammatically, this simply expressed as the fact that any (connected) plex diagram involving only identity plexes equals the identity plex of order equal to the number of free indices, for instance:
\begin{center}
\begin{tikzpicture}[]

\begin{scope}[xshift=0, yshift=0]

\coordinate [] (2) at (0.5,{0.5*sqrt(3)});
\coordinate [] (1) at (1,0);
\coordinate [] (0) at (0,0);

\draw[lightgray, line width=1pt, double] (2)--(0);
\draw[lightgray, line width=1pt, double] (2)--(1);

\end{scope}

\node[] at (2.3,0.3) {$=$};

\begin{scope}[xshift=100, yshift=0]

\coordinate [] (2a) at (0.5+0.03,{0.5*sqrt(3)});
\coordinate [] (2b) at (0.5-0.03,{0.5*sqrt(3)});
\coordinate [] (1a) at ({1+0.03*0.5},{0.03*0.5*sqrt(3)});
\coordinate [] (1b) at ({1-0.03*0.5},{-0.03*0.5*sqrt(3)});
\coordinate [] (0a) at ({0.03*0.5},{-0.03*0.5*sqrt(3)});
\coordinate [] (0b) at ({-0.03*0.5},{0.03*0.5*sqrt(3)});

\draw[lightgray, line width=1pt] (0b) to[out=30, in=270] (2b);
\draw[lightgray, line width=1pt] (2a) to[out=270, in=150] (1a);
\draw[lightgray, line width=1pt] (1b) to[out=150, in=30] (0a);

\end{scope}

\node[] at (5.7,0.3) {$=$};

\begin{scope}[xshift=190, yshift=0]

\coordinate [] (0) at (0.5,{0.5/sqrt(3)});
\coordinate [] (2) at (0.5,{0.5*sqrt(3)});
\coordinate [] (1) at (1,0);
\coordinate [] (3) at (0,0);

\draw[lightgray, line width=1pt, double] (1)--(0);
\draw[lightgray, line width=1pt, double] (2)--(0);
\draw[lightgray, line width=1pt, double] (3)--(0);

\filldraw[darkgray] (0) circle (0.07);

\end{scope}

\end{tikzpicture}
\end{center}
Attaching identity plexes to a given plex via incidence or contraction generally defines the plex version of \textbf{diagonal extension}: given a plex on a constellation of indices $\{I_1,I_2, \dots, I_n\}$ we can always define a new plex by taking several copies of an index set $\{I_1, I_1, \dots, I_1, I_2, \dots , I_n\}$ and assigning the zero value $0\in V$ to all the non diagonal combinations of $I_1$ index elements. For instance, given a $3$-plex $a$, the contraction with a single index of the identity $3$-plex
\begin{center}
\begin{tikzpicture}[line join = round, line cap = round]

\coordinate [] (4) at ({-0.5*sqrt(3)},-0.5);
\coordinate [] (3) at ({-0.5*sqrt(3)},0.5);
\coordinate [] (2) at ({0.5*sqrt(3)},-0.5);
\coordinate [] (1) at ({0.5*sqrt(3)},0.5);
\coordinate [] (0) at (0,0);

\begin{scope}[xshift=0, yshift=0, rotate=-30]

\coordinate [] (2a) at (0.5+0.03,{0.5*sqrt(3)});
\coordinate [] (2b) at (0.5-0.03,{0.5*sqrt(3)});
\coordinate [] (1a) at ({1+0.03*0.5},{0.03*0.5*sqrt(3)});
\coordinate [] (1b) at ({1-0.03*0.5},{-0.03*0.5*sqrt(3)});
\coordinate [] (0a) at ({0.03*0.5},{-0.03*0.5*sqrt(3)});
\coordinate [] (0b) at ({-0.03*0.5},{0.03*0.5*sqrt(3)});

\draw[lightgray, line width=1pt] (0b) to[out=30, in=270] (2b);
\draw[lightgray, line width=1pt] (2a) to[out=270, in=150] (1a);
\draw[lightgray, line width=1pt] (1b) to[out=150, in=30] (0a);

\end{scope}

\draw[lightgray, opacity=0.5, fill=lightgray,fill opacity=0.5] (0)--(3)--(4)--cycle;

\filldraw[darkgray] (0) circle (0.07);

\node[] at (-0.57,0) {$a$};

\end{tikzpicture}
\end{center}
defines a $4$-plex that encodes the same information as $a$. Identity plexes act as \textbf{compositionally neutral} elements with respect to contraction. Identity plexes can be inserted into any marked vertex of a plex diagram without altering the outcome, as illustrated by the following diagrammatic equations:
\begin{center}
\begin{tikzpicture}[]

\begin{scope}[xshift=0, yshift=0]

\coordinate [] (6) at ({0.5*sqrt(3)},0);
\coordinate [] (4) at ({-0.5*sqrt(3)},-0.5);
\coordinate [] (3) at ({-0.5*sqrt(3)},0.5);
\coordinate [] (2) at ({0.5*sqrt(3)},-0.5);
\coordinate [] (1) at ({0.5*sqrt(3)},0.5);
\coordinate [] (0) at (0,0);

\draw[lightgray, line width=1pt, double] (6)--(0);
\draw[lightgray, opacity=0.5, fill=lightgray,fill opacity=0.5] (0)--(3)--(4)--cycle;

\filldraw[darkgray] (0) circle (0.07);

\end{scope}

\node[] at (1.45,0) {$=$};

\begin{scope}[xshift=85, yshift=0]

\coordinate [] (6) at ({0.5*sqrt(3)},0);
\coordinate [] (4) at ({-0.5*sqrt(3)},-0.5);
\coordinate [] (3) at ({-0.5*sqrt(3)},0.5);
\coordinate [] (2) at ({0.5*sqrt(3)},-0.5);
\coordinate [] (1) at ({0.5*sqrt(3)},0.5);
\coordinate [] (0) at (0,0);

\draw[lightgray, opacity=0.5, fill=lightgray,fill opacity=0.5] (0)--(3)--(4)--cycle;

\end{scope}

\begin{scope}[xshift=180, yshift=0]

\coordinate [] (6) at (-0.5,{0.5*sqrt(3)});
\coordinate [] (5) at (-1,0);
\coordinate [] (4) at (-0.5,{-0.5*sqrt(3)});
\coordinate [] (3) at (0.5,{-0.5*sqrt(3)});
\coordinate [] (2) at (1,0);
\coordinate [] (1) at (0.5,{0.5*sqrt(3)});
\coordinate [] (0) at (0,0);

\draw[lightgray, opacity=0.5, fill=lightgray,fill opacity=0.5] (0)--(1)--(2)--cycle;
\draw[lightgray, opacity=0.5, fill=lightgray,fill opacity=0.5] (0)--(3)--(4)--cycle;
\draw[lightgray, opacity=0.5, fill=lightgray,fill opacity=0.5] (0)--(5)--(6)--cycle;

\filldraw[darkgray] (0) circle (0.07);

\end{scope}

\node[] at (8.05,0) {$=$};

\begin{scope}[xshift=290, yshift=0]

\coordinate [] (9) at (-0.5,{0.5/sqrt(3)});
\coordinate [] (8) at (0,{-1/sqrt(3)});
\coordinate [] (7) at (0.5,{0.5/sqrt(3)});
\coordinate [] (6) at (-1,{0.5*sqrt(3)+0.5/sqrt(3)});
\coordinate [] (5) at (-1.5,{0.5/sqrt(3)});
\coordinate [] (4) at (-0.5,{-0.5*sqrt(3)-1/sqrt(3)});
\coordinate [] (3) at (0.5,{-0.5*sqrt(3)-1/sqrt(3)});
\coordinate [] (2) at (1.5,{0.5/sqrt(3)});
\coordinate [] (1) at (1,{0.5*sqrt(3)+0.5/sqrt(3)});
\coordinate [] (0) at (0,0);

\draw[lightgray, opacity=0.5, fill=lightgray,fill opacity=0.5] (7)--(1)--(2)--cycle;
\draw[lightgray, opacity=0.5, fill=lightgray,fill opacity=0.5] (8)--(3)--(4)--cycle;
\draw[lightgray, opacity=0.5, fill=lightgray,fill opacity=0.5] (9)--(5)--(6)--cycle;

\begin{scope}[xshift=0, yshift=-16.5, rotate=60]

\coordinate [] (2a) at (0.5+0.03,{0.5*sqrt(3)});
\coordinate [] (2b) at (0.5-0.03,{0.5*sqrt(3)});
\coordinate [] (1a) at ({1+0.03*0.5},{0.03*0.5*sqrt(3)});
\coordinate [] (1b) at ({1-0.03*0.5},{-0.03*0.5*sqrt(3)});
\coordinate [] (0a) at ({0.03*0.5},{-0.03*0.5*sqrt(3)});
\coordinate [] (0b) at ({-0.03*0.5},{0.03*0.5*sqrt(3)});

\draw[lightgray, line width=1pt] (0b) to[out=30, in=270] (2b);
\draw[lightgray, line width=1pt] (2a) to[out=270, in=150] (1a);
\draw[lightgray, line width=1pt] (1b) to[out=150, in=30] (0a);

\end{scope}

\filldraw[darkgray] (7) circle (0.07);
\filldraw[darkgray] (8) circle (0.07);
\filldraw[darkgray] (9) circle (0.07);

\end{scope}

\end{tikzpicture}
\end{center}
Identity plexes are also useful to diagrammatically express \textbf{self-incidences} and \textbf{self-contractions} cleanly. For instance, a self-incidence of a $3$-plex $a$ can be written as the simultaneous incidence and contraction with the identity $2$-plex
\begin{center}
\begin{tikzpicture}[]

\coordinate [] (4) at (1.366,0.8);
\coordinate [] (3) at (1.366,-0.2);
\coordinate [] (2) at (0.5,0.3);

\draw[lightgray, line width=1pt, double] (3) to[out=200, in=280] (2);
\draw[lightgray, opacity=0.5, fill=lightgray,fill opacity=0.5] (2)--(3)--(4)--cycle;

\filldraw[darkgray] (2) circle (0.07);

\node[] at (1.05,0.3) {$a$};

\node[] at (2.2,0.2) {$=$};

\node[] at (5.85,0.1) {$\displaystyle \sum_p \delta_{ip}\cdot a_{ipk} \, = \, a_{iik} \, = \, \sum_p \delta_{pi}\cdot a_{pik}$};

\begin{scope}[xshift=280, yshift=0]

\coordinate [] (4) at (1.366,0.8);
\coordinate [] (3) at (1.366,-0.2);
\coordinate [] (2) at (0.5,0.3);

\draw[lightgray, line width=1pt, double] (3) to[out=200, in=280] (2);
\draw[lightgray, opacity=0.5, fill=lightgray,fill opacity=0.5] (2)--(3)--(4)--cycle;

\filldraw[darkgray] (3) circle (0.07);

\node[] at (1.05,0.3) {$a$};

\node[] at (-0.3,0.2) {$=$};

\end{scope}

\end{tikzpicture}
\end{center}
where we have chosen some sequentialization. By further contracting the incident index in the plex diagram we obtain a self-contraction or, in this case, a \textbf{partial trace}, since two indices are being contracted:
\begin{center}
\begin{tikzpicture}[]

\coordinate [] (4) at (1.366,0.8);
\coordinate [] (3) at (1.366,-0.2);
\coordinate [] (2) at (0.5,0.3);

\draw[lightgray, line width=1pt, double] (3) to[out=200, in=280] (2);
\draw[lightgray, opacity=0.5, fill=lightgray,fill opacity=0.5] (2)--(3)--(4)--cycle;

\filldraw[darkgray] (2) circle (0.07);
\filldraw[darkgray] (3) circle (0.07);

\node[] at (1.05,0.3) {$a$};

\node[] at (2.2,0.2) {$=$};

\node[] at (5.3,0.1) {$\displaystyle \sum_p \sum_q \delta_{pq}\cdot a_{pqk} \, = \, \sum_p a_{ppk} $};

\end{tikzpicture}
\end{center}
Alternatively, self-incidences and self-contractions may sometimes be diagrammatically expressed by bending or folding plexes so as to bring the vertices which represent incident indexes into contact. This can be illustrated by the following diagrammatic equations expressing the ordinary notion of \textbf{trace} for a $2$-array $m$:
\begin{center}
\begin{tikzpicture}[]

\begin{scope}[xshift=0, yshift=0]

\coordinate [] (2) at (-0.5,-0.5);
\coordinate [] (1) at (0.5,0.5);
\coordinate [] (0) at (0,0);

\draw[opacity=0.5, lightgray, ultra thick] (2) to[out=90, in=180] (1);
\draw[lightgray, line width=1pt, double] (2) to[out=0, in=270] (1);

\filldraw[darkgray] (2) circle (0.07);

\node[] at (-0.5,0.5) {$m$};

\node[] at (1.2,0) {$=$};

\node[] at (2,0) {$m_{ii}$};

\end{scope}

\begin{scope}[xshift=100, yshift=0]

\coordinate [] (2) at (-0.5,-0.5);
\coordinate [] (1) at (0.5,0.5);
\coordinate [] (0) at (0,0);

\draw[opacity=0.5, lightgray, ultra thick] (0) to[out=90, in=270] (0);

\draw[opacity=0.5, lightgray, ultra thick] (0.5,0) circle (0.5);
\filldraw[opacity=0.5, lightgray] (0) circle (0.07);

\node[] at (-0.7,0) {$=$};

\end{scope}

\begin{scope}[xshift=230, yshift=0]

\coordinate [] (2) at (-0.5,-0.5);
\coordinate [] (1) at (0.5,0.5);
\coordinate [] (0) at (0,0);

\draw[opacity=0.5, lightgray, ultra thick] (2) to[out=90, in=180] (1);
\draw[lightgray, line width=1pt, double] (2) to[out=0, in=270] (1);

\filldraw[darkgray] (2) circle (0.07);
\filldraw[darkgray] (1) circle (0.07);

\node[] at (-0.5,0.5) {$m$};

\node[] at (1.2,0) {$=$};

\node[] at (2.35,-0.1) {$\displaystyle \sum_p m_{pp}$};

\end{scope}

\begin{scope}[xshift=350, yshift=0]

\coordinate [] (2) at (-0.5,-0.5);
\coordinate [] (1) at (0.5,0.5);
\coordinate [] (0) at (0,0);

\draw[opacity=0.5, lightgray, ultra thick] (0) to[out=90, in=270] (0);

\draw[opacity=0.5, lightgray, ultra thick] (0.5,0) circle (0.5);
\filldraw[darkgray] (0) circle (0.07);

\node[] at (-0.7,0) {$=$};

\end{scope}

\end{tikzpicture}
\end{center}

The tensor product of arrays extends to plexes in an obvious way via quotient by reordering. Diagrammatically, the \textbf{plex tensor product} corresponds simply to juxtaposition or grouping of plexes as illustrated here:
\begin{center}
\begin{tikzpicture}[line join = round, line cap = round]

\draw[rounded corners] (0,-0.8) rectangle (4,2.3);

\draw[rounded corners] (4,0.5)--(4.5,0.5)--(4.5,1)--(4,1);

\node[] at (4.23,0.755) {$\otimes$};

\begin{scope}[xshift=0, yshift=0]

\coordinate [label=left:$i$] (0) at (0.5,0.5);

\filldraw [lightgray] (0) circle (2pt);

\node[] at (0.7,0.7) {$v$};

\end{scope}

\begin{scope}[xshift=40, yshift=25]

\coordinate [label=right:$k$] (1) at (1.5,1);
\coordinate [label=left:$j$] (0) at (0.5,0.5);

\draw[opacity=0.5, lightgray, ultra thick] (1)--(0);

\node[] at (0.8,1) {$b$};

\end{scope}

\begin{scope}[xshift=65, yshift=-10]

\coordinate [label=above:$l$] (2) at (0.5,{0.5*sqrt(3)});
\coordinate [label=right:$m$] (1) at (1,0);
\coordinate [label=left:$n$] (0) at (0,0);

\draw[lightgray, opacity=0.5,fill=lightgray,fill opacity=0.5] (1)--(0)--(2)--cycle;

\node[] at (0.5,0.3) {$a$};

\end{scope}

\begin{scope}[xshift=207, yshift=20]

\coordinate [label=above left:$i$] (6) at (-0.5,{0.5*sqrt(3)});
\coordinate [label=left:$j$] (5) at (-1,0);
\coordinate [label=below left:$k$] (4) at (-0.5,{-0.5*sqrt(3)});
\coordinate [label=below right:$l$] (3) at (0.5,{-0.5*sqrt(3)});
\coordinate [label=right:$m$] (2) at (1,0);
\coordinate [label=above right:$n$] (1) at (0.5,{0.5*sqrt(3)});
\coordinate [] (0) at (0,0);

\node[] at (-2.1,0) {$=$};

\filldraw[lightgray, opacity=0.5] (1)--(2)--(3)--(4)--(5)--(6)--cycle;

\node[] at (-0.4,-0.3) {$v$};
\node[] at (0,0.4) {$b$};
\node[] at (0.35,-0.3) {$a$};
\node[] at (0,0) {$\otimes$};

\end{scope}

\begin{scope}[xshift=320, yshift=20]

\node[] at (-1.8,0) {$=$};

\node[] at (0,-0.04) {$v_i \cdot b_{jk} \cdot a_{lmn}$};

\end{scope}

\end{tikzpicture}
\end{center}
where we have chosen some sequentialization. We note again that the commutativity and associativity properties of the semiring multiplication $\cdot$ are crucial for the tensor product of plexes to be diagrammatically well-defined. Recall that in our convention $0$-plexes correspond to elements of the value semiring $(V,+,\cdot)$, therefore, the tensor product of $0$-plexes with arbitrary plexes corresponds to the standard entry-wise \textbf{scalar multiplication}. Together with entry-wise addition of plexes, this operations make plexes of a fixed order $n$ into a module over the semiring $(V,+,\cdot)$. This suggests a compatible use of the notation introduced for tensor products as groupings by simply juxtaposing a semiring element to a plex to indicate scalar multiplication, as illustrated here for the product of a $3$-plex $a$ with a semiring element $\nu\in V$:
\begin{center}
\begin{tikzpicture}[line join = round, line cap = round]

\draw[rounded corners] (-0.4,-0.4) rectangle (1.8,1);

\draw[rounded corners] (1.8,0.07)--(2.3,0.07)--(2.3,0.57)--(1.8,0.57);

\node[] at (2.03,0.31) {$\otimes$};

\begin{scope}[xshift=0, yshift=0]

\coordinate [] (4) at (1.366,0.8);
\coordinate [] (3) at (1.366,-0.2);
\coordinate [] (2) at (0.5,0.3);

\node[] at (0,0.3) {$\nu$};

\draw[lightgray, opacity=0.5, fill=lightgray,fill opacity=0.5] (2)--(3)--(4)--cycle;

\node[] at (1.05,0.3) {$a$};

\node[] at (3,0.3) {$=$};

\end{scope}

\begin{scope}[xshift=110, yshift=0]

\coordinate [] (4) at (1.366,0.8);
\coordinate [] (3) at (1.366,-0.2);
\coordinate [] (2) at (0.5,0.3);

\node[] at (0,0.3) {$\nu$};

\draw[lightgray, opacity=0.5, fill=lightgray,fill opacity=0.5] (2)--(3)--(4)--cycle;

\node[] at (1.05,0.3) {$a$};

\node[] at (2.2,0.3) {$=$};

\end{scope}

\begin{scope}[xshift=200, yshift=0]

\coordinate [] (4) at (1.366,0.8);
\coordinate [] (3) at (1.366,-0.2);
\coordinate [] (2) at (0.5,0.3);

\node[] at (0,0.26) {$\nu \cdot a_{ijk}$};

\end{scope}

\end{tikzpicture}
\end{center}
For plexes of order greater than $0$, the plex tensor product is strictly order-increasing, as illustrated by the following diagrammatic equations
\begin{center}
\begin{tikzpicture}[line join = round, line cap = round]

\begin{scope}[xshift=0, yshift=0]

\draw[rounded corners] (-0.4,-0.4) rectangle (1.8,1);

\draw[rounded corners] (1.8,0.07)--(2.3,0.07)--(2.3,0.57)--(1.8,0.57);

\node[] at (2.03,0.31) {$\otimes$};

\filldraw[opacity=0.5, lightgray] (1.2,0.65) circle (0.07);
\filldraw[opacity=0.5, lightgray] (1,0.05) circle (0.07);
\filldraw[opacity=0.5, lightgray] (0.2,0.35) circle (0.07);

\node[] at (3,0.3) {$=$};

\end{scope}

\begin{scope}[xshift=90, yshift=0]

\coordinate [] (4) at (1.366,0.8);
\coordinate [] (3) at (1.366,-0.2);
\coordinate [] (2) at (0.5,0.3);

\draw[lightgray, opacity=0.5, fill=lightgray,fill opacity=0.5] (2)--(3)--(4)--cycle;

\end{scope}

\begin{scope}[xshift=200, yshift=0]

\draw[rounded corners] (-0.4,-0.4) rectangle (1.8,1);

\draw[rounded corners] (1.8,0.07)--(2.3,0.07)--(2.3,0.57)--(1.8,0.57);

\node[] at (2.03,0.31) {$\otimes$};

\draw[opacity=0.5, lightgray, ultra thick] (0,-0.05)--(0.6,0.55);
\draw[opacity=0.5, lightgray, ultra thick] (0.7,0.1)--(1.5,0.7);

\node[] at (3,0.3) {$=$};

\end{scope}

\begin{scope}[xshift=330, yshift=-7]

\coordinate [] (3) at (0,{sqrt(2)},0);
\coordinate [] (2) at ({-.5*sqrt(3)},0,-.5);
\coordinate [] (1) at (0,0,1);
\coordinate [] (0) at ({.5*sqrt(3)},0,-.5);

\draw[densely dotted,postaction={decorate}] (0)--(2);
\draw[fill=lightgray,fill opacity=.5] (1)--(0)--(3)--cycle;
\draw[fill=gray,fill opacity=.5] (2)--(1)--(3)--cycle;
\draw[gray, opacity=0.5, postaction={decorate}] (1)--(0);
\draw[gray, opacity=0.5, postaction={decorate}] (1)--(2);
\draw[gray, opacity=0.5, postaction={decorate}] (2)--(3);
\draw[gray, opacity=0.5, postaction={decorate}] (1)--(3);
\draw[gray, opacity=0.5, postaction={decorate}] (0)--(3);

\end{scope}

\end{tikzpicture}
\end{center}

\section{Associativity as Confluence} \label{confl}

We now come to the central topic of this paper: associativity. In this section we shall use array algebra and the plex formalism introduced in Sections \ref{array} and \ref{plex} to give a precise diagrammatic characterization of the notion of sequential associativity in preparation for our discussion on ternary generalizations of associativity in Sections \ref{fish} and \ref{towards}.\newline

There are many equivalent ways one can think about binary associativity: an algebraic property of an operation, a compatibility condition for left and right actions, the single nature of a sequence of processes, or a feature of some rewrite systems. Although it is useful to keep all these interpretations in mind throughout our discussion, we shall focus on the rewrite approach to associativity. This is in keeping with our de-emphasis on sequentiality, as we are aiming to establish a notion of associativity that is general enough to apply to non-sequential settings but that recovers the standard definition when applied to binary operations and compositions.\newline

We begin by considering the lowest-arity non-trivial order-preserving plex product that can be defined: the binary composition of $2$-plexes
\begin{center}
\begin{tikzpicture}[line join = round, line cap = round]

\begin{scope}[xshift=0, yshift=0]

\coordinate [] (2) at (1,0);
\coordinate [] (1) at (0,1);
\coordinate [] (0) at (-1,0);

\draw[opacity=0.5, lightgray, ultra thick] (0)--(1);
\draw[opacity=0.5, lightgray, ultra thick] (1)--(2);

\node[] at (-0.8,0.6) {$a$};
\node[] at (0.8,0.65) {$b$};

\filldraw[darkgray] (1) circle (0.07);

\end{scope}

\begin{scope}[xshift=120, yshift=10]

\node[] at (-2,0) {$=$};

\coordinate [] (2) at (1,0);
\coordinate [] (1) at (0,1);
\coordinate [] (0) at (-1,0);

\draw[opacity=0.5, lightgray, ultra thick] (0)--(2);

\node[] at (0,0.4) {$ab$};

\end{scope}

\end{tikzpicture}
\end{center}
where the LHS plex diagram is called a \textbf{vee} of $2$-plexes. This is indeed the pre-sequential manifestation of ordinary composition of functions, relations, and, generally, morphisms in a category \cite{leinster2014basic}. This plex product thus appears as the natural starting point to characterize the notion of associativity. Note that despite the sequential appearance of the standard labelling convention above, where we indicate plex by alphabetic labels and the product by juxtaposition of labels, the diagrammatic correspondence between symbols and plexes is entirely determined by incidence topology, i.e. it only depends on the relative positions of common indices. This can be seen explicitly if we label indices
\begin{center}
\begin{tikzpicture}[line join = round, line cap = round]

\begin{scope}[xshift=0, yshift=0]

\coordinate [label=right:$k$] (2) at (1,0);
\coordinate [label=above:$j$] (1) at (0,1);
\coordinate [label=left:$i$] (0) at (-1,0);

\draw[opacity=0.5, lightgray, ultra thick] (0)--(1);
\draw[opacity=0.5, lightgray, ultra thick] (1)--(2);

\node[] at (-0.8,0.6) {$a$};
\node[] at (0.8,0.65) {$b$};

\filldraw[darkgray] (1) circle (0.07);

\end{scope}

\begin{scope}[xshift=120, yshift=10]

\node[] at (-2,0) {$=$};

\coordinate [label=right:$k$] (2) at (1,0);
\coordinate [label=left:$i$] (0) at (-1,0);

\draw[opacity=0.5, lightgray, ultra thick] (0)--(2);

\node[] at (0,0.4) {$ab$};

\end{scope}

\end{tikzpicture}
\end{center}
and observe that the label `$a$' is adjacent to the index $i$ and the label `$b$' is adjacent to the index $k$ both in the plex diagram representing the plex product (LHS) and in the label `$ab$' for the resulting plex (RHS). Consequently, we have the following pair of equivalent labellings of the composition $2$-plex:
\begin{center}
\begin{tikzpicture}[line join = round, line cap = round]

\begin{scope}[xshift=0, yshift=0]

\coordinate [label=right:$k$] (2) at (1,0);
\coordinate [label=left:$i$] (0) at (-1,0);

\draw[opacity=0.5, lightgray, ultra thick] (0)--(2);

\node[] at (0,0.4) {$ab$};

\end{scope}

\begin{scope}[xshift=120, yshift=0]

\node[] at (-2,0) {$=$};

\coordinate [label=right:$i$] (2) at (1,0);
\coordinate [label=left:$k$] (0) at (-1,0);

\draw[opacity=0.5, lightgray, ultra thick] (0)--(2);

\node[] at (0,0.4) {$ba$};

\end{scope}

\end{tikzpicture}
\end{center}

The associativity property of the binary composition of $2$-plexes amounts to the fact that the composition of three $2$-plexes evaluates uniquely to a well-defined $2$-plex
\begin{center}
\begin{tikzpicture}[line join = round, line cap = round]

\begin{scope}[xshift=0, yshift=0]

\coordinate [] (3) at (2,1);
\coordinate [] (2) at (1,0);
\coordinate [] (1) at (0,1);
\coordinate [] (0) at (-1,0);

\draw[opacity=0.5, lightgray, ultra thick] (0)--(1);
\draw[opacity=0.5, lightgray, ultra thick] (1)--(2);
\draw[opacity=0.5, lightgray, ultra thick] (2)--(3);

\node[] at (-0.8,0.6) {$a$};
\node[] at (0.8,0.65) {$b$};
\node[] at (2,0.6) {$c$};

\filldraw[darkgray] (1) circle (0.07);
\filldraw[darkgray] (2) circle (0.07);

\end{scope}

\begin{scope}[xshift=150, yshift=10]

\node[] at (-2,0) {$=$};

\coordinate [] (2) at (1,0);
\coordinate [] (1) at (0,1);
\coordinate [] (0) at (-1,0);

\draw[opacity=0.5, lightgray, ultra thick] (0)--(2);

\node[] at (0,0.4) {$abc$};

\end{scope}

\end{tikzpicture}
\end{center}
This plex diagram is called a \textbf{zee} of $2$-plexes. However, as remarked in Section \ref{plex} when we gave the general definition of a plex product, this is a general property of all plex diagrams by construction: a simple connected hypergraph uniquely specifies (up to sequentialization) a multiplication of arrays with values in a commutative semiring. Consequently, in order to isolate what is special about the binary composition of $2$-plexes we must find a characterization of associativity articulated purely within the hypergraph formalism of plex diagrams.\newline

The key concept that we need to borrow from hypergraph theory is that of motif rewrite. Given a hypergraph $H$, a \textbf{motif} is a choice of isomorphism class of simple connected hypergraphs $[M]$ that is used to probe the subhypergraphs of $H$. \textbf{Motif detection} is the process by which subhypergraphs $S\subset H$ such that $M\cong S$ are found. In the context of the plex formalism, a motif is simply a choice of plex diagram (without labels) representing a fixed placeholder product of plexes. For instance, we can choose as our motif the following ternary composition of $3$-plexes
\begin{center}
\begin{tikzpicture}[line join = round, line cap = round]

\begin{scope}[xshift=0, yshift=0]

\coordinate [] (3) at ({0.5*sqrt(3)},-0.5);
\coordinate [] (2) at (-{0.5*sqrt(3)},-0.5);
\coordinate [] (1) at (0,1);
\coordinate [] (0) at (0,0);

\draw[gray, opacity=0.5,fill=lightgray,fill opacity=0.5] (1)--(0)--(2)--cycle;
\draw[gray, opacity=0.5,fill=lightgray,fill opacity=0.5] (2)--(0)--(3)--cycle;
\draw[gray, opacity=0.5,fill=lightgray,fill opacity=0.5] (1)--(0)--(3)--cycle;

\filldraw[darkgray] (0) circle (0.07);

\end{scope}

\node[] at (1.8,0) {$=$};

\begin{scope}[xshift=75, yshift=-8]

\coordinate [] (2) at (0.5,{0.5*sqrt(3)});
\coordinate [] (1) at (1,0);
\coordinate [] (0) at (0,0);

\draw[lightgray, opacity=0.5,fill=lightgray,fill opacity=0.5] (1)--(0)--(2)--cycle;

\end{scope}

\end{tikzpicture}
\end{center}
known as the Bhattacharya-Mesner product \cite{mesner1990association,gnang2020bhattacharya}. Then, given an arbitrary plex diagram we can perform a motif detection procedure that identifies the subhypergraphs which correspond to the above plex diagram. Since motifs correspond to products, the plexes involved in each particular instance of the the motif can be operated to output a single plex thus effectively rewriting the plex diagram. An example with the Bhattacharya-Mesner motif
\begin{center}
\begin{tikzpicture}[line join = round, line cap = round]

\begin{scope}[xshift=0, yshift=0]

\coordinate [] (4) at (1.8,0.4);
\coordinate [] (5) at (2,-0.7);
\coordinate [] (3) at ({0.5*sqrt(3)},-0.5);
\coordinate [] (2) at (-{0.5*sqrt(3)},-0.5);
\coordinate [] (1) at (0,1);
\coordinate [] (0) at (0,0);

\draw[gray, opacity=0.5,fill=lightgray,fill opacity=0.5] (1)--(0)--(2)--cycle;
\draw[gray, opacity=0.5,fill=lightgray,fill opacity=0.5] (2)--(0)--(3)--cycle;
\draw[gray, opacity=0.5,fill=lightgray,fill opacity=0.5] (1)--(0)--(3)--cycle;

\draw[opacity=0.5, lightgray, ultra thick] (1) to[out=210, in=90] (2);
\draw[opacity=0.5, lightgray, ultra thick] (3)--(4);
\draw[opacity=0.5, lightgray, ultra thick] (3)--(5);

\filldraw[darkgray] (0) circle (0.07);

\end{scope}

\node[] at (3.2,0) {$\mapsto$};

\begin{scope}[xshift=150, yshift=0]

\coordinate [] (4) at (1.8,0.4);
\coordinate [] (5) at (2,-0.7);
\coordinate [] (3) at ({0.5*sqrt(3)},-0.5);
\coordinate [] (2) at (-{0.5*sqrt(3)},-0.5);
\coordinate [] (1) at (0,1);
\coordinate [] (0) at (0,0);

\draw[gray, opacity=0.5,fill=lightgray,fill opacity=0.5] (1)--(2)--(3)--cycle;

\draw[opacity=0.5, lightgray, ultra thick] (1) to[out=210, in=90] (2);
\draw[opacity=0.5, lightgray, ultra thick] (3)--(4);
\draw[opacity=0.5, lightgray, ultra thick] (3)--(5);

\end{scope}

\end{tikzpicture}
\end{center}
illustrates an instance of hypergraph \textbf{motif rewrite}. In general, motif detection and rewrites define a multiway rewrite system \cite{wolfram2020class} of hypergraphs.\newline

Returning to the binary composition of $2$-plexes, we can now express associativity as a property of the rewrite system resulting from doing vee motif rewrites on the zee diagram:
\begin{center}
\begin{tikzpicture}[line join = round, line cap = round]

\begin{scope}[xshift=0, yshift=0]

\coordinate [] (3) at (2,1);
\coordinate [] (2) at (1,0);
\coordinate [] (1) at (0,1);
\coordinate [] (0) at (-1,0);

\draw[opacity=0.5, lightgray, ultra thick] (0)--(1);
\draw[opacity=0.5, lightgray, ultra thick] (1)--(2);
\draw[opacity=0.5, lightgray, ultra thick] (2)--(3);

\node[] at (-0.8,0.6) {$a$};
\node[] at (0.8,0.65) {$b$};
\node[] at (2,0.6) {$c$};

\filldraw[darkgray] (1) circle (0.07);
\filldraw[darkgray] (2) circle (0.07);

\end{scope}

\begin{scope}[xshift=-100, yshift=-105]

\coordinate [] (3) at (2,1);
\coordinate [] (2) at (1,0);
\coordinate [] (1) at (0,1);
\coordinate [] (0) at (-1,0);

\draw[opacity=0.5, lightgray, ultra thick] (0)--(2);
\draw[opacity=0.5, lightgray, ultra thick] (2)--(3);

\node[] at (0,0.4) {$ab$};
\node[] at (2,0.6) {$c$};

\filldraw[darkgray] (2) circle (0.07);

\end{scope}

\begin{scope}[xshift=100, yshift=-105]

\coordinate [] (3) at (2,1);
\coordinate [] (2) at (1,0);
\coordinate [] (1) at (0,1);
\coordinate [] (0) at (-1,0);

\draw[opacity=0.5, lightgray, ultra thick] (0)--(1);
\draw[opacity=0.5, lightgray, ultra thick] (1)--(3);

\node[] at (-0.8,0.6) {$a$};
\node[] at (1,1.4) {$bc$};

\filldraw[darkgray] (1) circle (0.07);

\end{scope}

\begin{scope}[xshift=0, yshift=-190]

\coordinate [] (3) at (2,0);
\coordinate [] (2) at (1,0);
\coordinate [] (1) at (0,1);
\coordinate [] (0) at (-1,0);

\draw[opacity=0.5, lightgray, ultra thick] (0)--(3);

\node[] at (0.5,0.4) {$abc$};

\end{scope}

\draw[-stealth, line width=1pt] (-1.5,-0.5) to[out=210, in=90] (-2.5,-2);
\draw[-stealth, line width=1pt] (2.2,-0.5) to[out=-30, in=90] (3.2,-2);
\draw[-stealth, line width=1pt] (-2.5,-4.5) to[out=-90, in=150] (-1.5,-6);
\draw[-stealth, line width=1pt] (3.2,-4.5) to[out=-90, in=30] (2.2,-6);

\begin{scope}[scale=0.4, xshift=-250, yshift=-90]

\coordinate [] (3) at (2,1);
\coordinate [] (2) at (1,0);
\coordinate [] (1) at (0,1);
\coordinate [] (0) at (-1,0);

\draw[opacity=0.5, darkgray, ultra thick] (0)--(1);
\draw[opacity=0.5, darkgray, ultra thick] (1)--(2);
\draw[opacity=0.5, lightgray, ultra thick] (2)--(3);

\end{scope}

\begin{scope}[scale=0.4, xshift=275, yshift=-90]

\coordinate [] (3) at (2,1);
\coordinate [] (2) at (1,0);
\coordinate [] (1) at (0,1);
\coordinate [] (0) at (-1,0);

\draw[opacity=0.5, lightgray, ultra thick] (0)--(1);
\draw[opacity=0.5, darkgray, ultra thick] (1)--(2);
\draw[opacity=0.5, darkgray, ultra thick] (2)--(3);

\end{scope}

\begin{scope}[scale=0.4, xshift=-250, yshift=-400]

\coordinate [] (3) at (2,1);
\coordinate [] (2) at (1,0);
\coordinate [] (1) at (0,1);
\coordinate [] (0) at (-1,0);

\draw[opacity=0.5, darkgray, ultra thick] (0)--(2);
\draw[opacity=0.5, darkgray, ultra thick] (2)--(3);

\end{scope}

\begin{scope}[scale=0.4, xshift=250, yshift=-400]

\coordinate [] (3) at (2,1);
\coordinate [] (2) at (1,0);
\coordinate [] (1) at (0,1);
\coordinate [] (0) at (-1,0);

\draw[opacity=0.5, darkgray, ultra thick] (0)--(1);
\draw[opacity=0.5, darkgray, ultra thick] (1)--(3);

\end{scope}

\end{tikzpicture}
\end{center}
This is, of course, a simple instance of a \textbf{confluent} rewrite system. More specifically, we say that the zee hypergraph above is \textbf{vee-confluent} or \textbf{confluent under vee rewrites}. In terms of plex products, the vee-confluence of the zee plex diagram results in the ordinary axiom of \textbf{binary associativity} by writing the labels of composed plexes in brackets:
\begin{equation}\label{assoc}
    ((ab)c)=(a(bc)). \tag{assoc}
\end{equation}

Aiming to abstract the property of associativity for general plex products, confluence thus appears as a desirable necessary condition. However it cannot be sufficient as a property of general plex rewrite systems since all plex diagrams are confluent under sufficiently many rewrite rules, again, by construction, since every plex diagram evaluates to a single plex. This prompts us to find a refined notion of confluence that holds in the case at hand of binary composition of $2$-plexes but that is general enough to apply in arbitrary plex products. We note three further features of the rewrite system that produces binary associativity:
\begin{itemize}
    \item \textbf{Single motif.} The rewrite system is generated by a single rewrite rule defined by detection of a single motif (the vee).
    \item \textbf{Regularity.} The initial hypergraph (the zee) only contains plexes of order $2$. Furthermore, the initial hypergraph is built from piece-wise plex compositions so all intermediate rewrites are also comprised of plexes of order $2$. We can say that the rewrite system is order-regular both vertically and horizontally.
    \item \textbf{Overlap.} The motifs that occur in the initial hypergraph (the zee) have non-empty intersection.
\end{itemize}
A confluent rewrite system of these characteristics is called \textbf{concurrent}. We have thus built a template for a generalized notion of associativity: a plex product defined by some plex diagram (simple connected hypergraph) is called \textbf{associative} if there exists a concurrent rewrite system taking the plex product as motif.\newline

As an immediate consequence of this definition we find that the rewrite system that produces binary associativity is not unique. Indeed, any \textbf{chain} of $2$-plexes:
\begin{center}
\begin{tikzpicture}[line join = round, line cap = round]

\begin{scope}[xshift=0, yshift=0]

\coordinate [] (4) at (3,0);
\coordinate [] (3) at (2,1);
\coordinate [] (2) at (1,0);
\coordinate [] (1) at (0,1);
\coordinate [] (0) at (-1,0);

\draw[opacity=0.5, lightgray, ultra thick] (0)--(1);
\draw[opacity=0.5, lightgray, ultra thick] (1)--(2);
\draw[opacity=0.5, lightgray, ultra thick] (2)--(3);
\draw[opacity=0.5, lightgray, ultra thick, dashed] (4)--(3);

\filldraw[darkgray] (1) circle (0.07);
\filldraw[darkgray] (2) circle (0.07);
\filldraw[darkgray] (3) circle (0.07);

\end{scope}

\node[] at (3.9,0.5) {$\cdots$};

\begin{scope}[xshift=200, yshift=0]

\coordinate [] (4) at (-2,1);
\coordinate [] (3) at (2,1);
\coordinate [] (2) at (1,0);
\coordinate [] (1) at (0,1);
\coordinate [] (0) at (-1,0);

\draw[opacity=0.5, lightgray, ultra thick, dashed] (0)--(4);
\draw[opacity=0.5, lightgray, ultra thick] (0)--(1);
\draw[opacity=0.5, lightgray, ultra thick] (1)--(2);
\draw[opacity=0.5, lightgray, ultra thick] (2)--(3);

\filldraw[darkgray] (0) circle (0.07);
\filldraw[darkgray] (1) circle (0.07);
\filldraw[darkgray] (2) circle (0.07);

\end{scope}

\end{tikzpicture}
\end{center}
gives a concurrent rewrite system by taking the vee as motif. This simply corresponds to concatenation of binary operations. The symbolic manifestation of this concurrent rewrite system is the property that allows to denote a series of binary compositions by a sequence of symbols. More generally, given a chain of length $n$ as initial plex diagram and a chain of length $m<n$ as motif, the resulting rewrite system is concurrent. The symbolic axiom resulting from it:
\begin{equation*}
    ((a_1a_2\cdots a_m)a_{m+1}\cdots a_n)=(a_1(a_2\cdots a_ma_{m+1})\cdots a_n) = \cdots = (a_1a_2\cdots (a_{n-m}\cdots a_n))
\end{equation*}
is called \textbf{sequential associativity} or \textbf{chain associativity}. By choosing $n=3$ and $m=2$ we find ordinary binary associativity as a particular case of sequential associativity. 

\section{The Fish Product} \label{fish}

One of the central goals of the Higher Arity Project \cite{zapata2022invitation,collab2022arity} is to find algebraic structures that generalize associative binary operations. Plex algebra, articulated in Sections \ref{array} and \ref{plex}, offers a rich sprawling landscape of compositional structures and thus appears as a natural starting point in the search for higher arity operations.\newline

In Section \ref{confl} we showed how the binary composition of $2$-plexes
\begin{center}
\begin{tikzpicture}[line join = round, line cap = round]

\begin{scope}[xshift=0, yshift=0]

\coordinate [] (2) at (1,0);
\coordinate [] (1) at (0,1);
\coordinate [] (0) at (-1,0);

\draw[opacity=0.5, lightgray, ultra thick] (0)--(1);
\draw[opacity=0.5, lightgray, ultra thick] (1)--(2);

\node[] at (-0.8,0.6) {$a$};
\node[] at (0.8,0.65) {$b$};

\filldraw[darkgray] (1) circle (0.07);

\end{scope}

\begin{scope}[xshift=120, yshift=10]

\node[] at (-2,0) {$=$};

\coordinate [] (2) at (1,0);
\coordinate [] (1) at (0,1);
\coordinate [] (0) at (-1,0);

\draw[opacity=0.5, lightgray, ultra thick] (0)--(2);

\node[] at (0,0.4) {$ab$};

\end{scope}

\end{tikzpicture}
\end{center}
embodies binary associativity in the context of plex algebra. We further noted that this plex product of order $2$ plexes is, in fact, the smallest plex composition. This compels us to look for higher order plex compositions. At one order higher, this means looking for plex products that take a number of $3$-plexes as input and give a $3$-plex as output. Enumerating all such plex products is a simple combinatorial task. The smallest instances of $3$-plex compositions are ternary, i.e. they take three $3$-plexes as input and give a $3$-plex as output, and there are $10$ of them, as shown by the authors in \cite[10-18]{zapata2022invitation}. Among these, we shall focus on the following ternary plex composition:
\begin{center}
\begin{tikzpicture}[line join = round, line cap = round]

\begin{scope}[xshift=0, yshift=0]

\begin{scope}[scale=0.7, xshift=0, yshift=0]

\coordinate [] (6) at ({3*sqrt(3)},0);
\coordinate [] (5) at ({2*sqrt(3)},-1);
\coordinate [] (4) at ({2*sqrt(3)},1);
\coordinate [] (3) at ({sqrt(3)},0);
\coordinate [] (2) at (0,-1);
\coordinate [] (1) at (0,1);
\coordinate [] (0) at (0,0);

\draw[gray, opacity=0.5,fill=lightgray,fill opacity=0.5] (1)--(2)--(3)--cycle;
\draw[gray, opacity=0.5,fill=lightgray,fill opacity=0.5] (3)--(4)--(5)--cycle;
\draw[gray, opacity=0.5,fill=lightgray,fill opacity=0.5] (4)--(5)--(6)--cycle;

\end{scope}

\node[] at ({0.7*0.6},0) {$a$};
\node[] at ({0.7*(2*sqrt(3)-0.6)},0.05) {$b$};
\node[] at ({0.7*(2*sqrt(3)+0.6)},0) {$c$};

\filldraw[darkgray] (5) circle (0.07);
\filldraw[darkgray] (4) circle (0.07);
\filldraw[darkgray] (3) circle (0.07);

\end{scope}

\begin{scope}[xshift=155, yshift=0]

\node[] at (-0.9,0) {$=$};

\begin{scope}[scale=0.7, xshift=0, yshift=0]

\coordinate [] (6) at ({3*sqrt(3)},0);
\coordinate [] (5) at ({2*sqrt(3)},-1);
\coordinate [] (4) at ({2*sqrt(3)},1);
\coordinate [] (3) at ({sqrt(3)},0);
\coordinate [] (2) at (0,-1);
\coordinate [] (1) at (0,1);
\coordinate [] (0) at (0,0);

\draw[gray, opacity=0.5,fill=lightgray,fill opacity=0.5] (1)--(2)--(3)--cycle;

\end{scope}

\node[] at ({0.7*0.6},0.04) {$abc$};

\end{scope}

\end{tikzpicture}
\end{center}
The simple connected hypergraph in the LHS is given the evocative name `\textbf{fish}' and so this operation is called the \textbf{fish product} of $3$-plexes. Note that this symbolic labelling convention follows the same adjacency criteria we adopted for the labelling of the binary composition at the start of Section \ref{confl} where the incidence topology uniquely determines the sequential order of plex labels. This ternary composition stands out on the set of ternary $3$-plex compositions for its quasi-sequential nature. The $a$, $b$, $c$ plexes act as a manner of input, throughput and output in the sequential operation of pair-wise contraction of $3$-plexes. For later ease of reference, we adopt some suggestive diagram terminology: we call $a$ the \textbf{tail} plex, $b$ the \textbf{body} plex and $c$ the \textbf{head} plex. Note that only $a$ and $c$ carry free indices as all the indices of $b$ are contracted. The two free indices of the tail plex are called \textbf{tips} and the free index of the head plex is called \textbf{mouth}. We summarize the fish hypergraph anatomy as follows:
\begin{center}
\begin{tikzpicture}[line join = round, line cap = round]

\begin{scope}[xshift=0, yshift=0]

\begin{scope}[scale=1, xshift=0, yshift=0]

\coordinate [label=right:mouth] (6) at ({3*sqrt(3)},0);
\coordinate [] (5) at ({2*sqrt(3)},-1);
\coordinate [] (4) at ({2*sqrt(3)},1);
\coordinate [] (3) at ({sqrt(3)},0);
\coordinate [label=below left:tip$_2$] (2) at (0,-1);
\coordinate [label=above left:tip$_1$] (1) at (0,1);
\coordinate [] (0) at (0,0);

\draw[gray, opacity=0.5,fill=lightgray,fill opacity=0.5] (1)--(2)--(3)--cycle;
\draw[gray, opacity=0.5,fill=lightgray,fill opacity=0.5] (3)--(4)--(5)--cycle;
\draw[gray, opacity=0.5,fill=lightgray,fill opacity=0.5] (4)--(5)--(6)--cycle;

\end{scope}

\node[] at (0.5,0) {tail};
\node[] at (2.8,-0.03) {body};
\node[] at (4.1,0) {head};

\filldraw[darkgray] (5) circle (0.07);
\filldraw[darkgray] (4) circle (0.07);
\filldraw[darkgray] (3) circle (0.07);

\end{scope}

\end{tikzpicture}
\end{center}

We begin by stating our central result.

\begin{thm}[Fish Associativity] \label{fishassoc}
The fish product is associative in the sense defined in Section \ref{confl}, that is, there exists a concurrent plex rewrite system with the fish as motif. Furthermore, the (minimal) symbolic condition that derives from the confluence of rewrites is the defining axiom of semiheaps (\ref{semiheap}).
\end{thm}
\begin{proof}
To prove this statement we need to find an initial plex diagram that is confluent under fish rewrites. More concretely, if we aim to recover the semiheap identity (\ref{semiheap}) the initial plex diagram must involve $5$ plexes that we shall label $a$, $b$, $c$, $d$ and $e$. By direct inspection of the fish product it is easy to conclude that the initial plex diagram must be the following simple connected hypergraph
\begin{center}
\begin{tikzpicture}[line join = round, line cap = round]

\begin{scope}[scale=0.7, xshift=0, yshift=0]

\coordinate [] (9) at ({5*sqrt(3)},0);
\coordinate [] (8) at ({4*sqrt(3)},-1);
\coordinate [] (7) at ({4*sqrt(3)},1);
\coordinate [] (6) at ({3*sqrt(3)},0);
\coordinate [] (5) at ({2*sqrt(3)},-1);
\coordinate [] (4) at ({2*sqrt(3)},1);
\coordinate [] (3) at ({sqrt(3)},0);
\coordinate [] (2) at (0,-1);
\coordinate [] (1) at (0,1);
\coordinate [] (0) at (0,0);

\draw[gray, opacity=0.5,fill=lightgray,fill opacity=0.5] (1)--(2)--(3)--cycle;
\draw[gray, opacity=0.5,fill=lightgray,fill opacity=0.5] (3)--(4)--(5)--cycle;
\draw[gray, opacity=0.5,fill=lightgray,fill opacity=0.5] (4)--(5)--(6)--cycle;
\draw[gray, opacity=0.5,fill=lightgray,fill opacity=0.5] (6)--(7)--(8)--cycle;
\draw[gray, opacity=0.5,fill=lightgray,fill opacity=0.5] (8)--(7)--(9)--cycle;

\end{scope}

\node[] at ({0.7*0.6},0) {$a$};
\node[] at ({0.7*(2*sqrt(3)-0.6)},0.05) {$b$};
\node[] at ({0.7*(2*sqrt(3)+0.6)},0) {$c$};
\node[] at ({0.7*(4*sqrt(3)-0.6)},0.05) {$d$};
\node[] at ({0.7*(4*sqrt(3)+0.6)},0) {$e$};

\filldraw[darkgray] (8) circle (0.07);
\filldraw[darkgray] (7) circle (0.07);
\filldraw[darkgray] (6) circle (0.07);
\filldraw[darkgray] (5) circle (0.07);
\filldraw[darkgray] (4) circle (0.07);
\filldraw[darkgray] (3) circle (0.07);

\end{tikzpicture}
\end{center}
which we call the \textbf{long fish} for later reference. Taking the fish plex diagram as a motif we obtain the following rewrite system:
\begin{center}
\begin{tikzpicture}[line join = round, line cap = round]

\begin{scope}[xshift=-45, yshift=30]

\begin{scope}[scale=0.7, xshift=0, yshift=0]

\coordinate [] (9) at ({5*sqrt(3)},0);
\coordinate [] (8) at ({4*sqrt(3)},-1);
\coordinate [] (7) at ({4*sqrt(3)},1);
\coordinate [] (6) at ({3*sqrt(3)},0);
\coordinate [] (5) at ({2*sqrt(3)},-1);
\coordinate [] (4) at ({2*sqrt(3)},1);
\coordinate [] (3) at ({sqrt(3)},0);
\coordinate [] (2) at (0,-1);
\coordinate [] (1) at (0,1);
\coordinate [] (0) at (0,0);

\draw[gray, opacity=0.5,fill=lightgray,fill opacity=0.5] (1)--(2)--(3)--cycle;
\draw[gray, opacity=0.5,fill=lightgray,fill opacity=0.5] (3)--(4)--(5)--cycle;
\draw[gray, opacity=0.5,fill=lightgray,fill opacity=0.5] (4)--(5)--(6)--cycle;
\draw[gray, opacity=0.5,fill=lightgray,fill opacity=0.5] (6)--(7)--(8)--cycle;
\draw[gray, opacity=0.5,fill=lightgray,fill opacity=0.5] (8)--(7)--(9)--cycle;

\end{scope}

\node[] at ({0.7*0.6},0) {$a$};
\node[] at ({0.7*(2*sqrt(3)-0.6)},0.05) {$b$};
\node[] at ({0.7*(2*sqrt(3)+0.6)},0) {$c$};
\node[] at ({0.7*(4*sqrt(3)-0.6)},0.05) {$d$};
\node[] at ({0.7*(4*sqrt(3)+0.6)},0) {$e$};

\filldraw[darkgray] (8) circle (0.07);
\filldraw[darkgray] (7) circle (0.07);
\filldraw[darkgray] (6) circle (0.07);
\filldraw[darkgray] (5) circle (0.07);
\filldraw[darkgray] (4) circle (0.07);
\filldraw[darkgray] (3) circle (0.07);

\end{scope}

\begin{scope}[xshift=-170, yshift=-91]

\begin{scope}[scale=0.7, xshift=0, yshift=0]

\coordinate [] (6) at ({3*sqrt(3)},0);
\coordinate [] (5) at ({2*sqrt(3)},-1);
\coordinate [] (4) at ({2*sqrt(3)},1);
\coordinate [] (3) at ({sqrt(3)},0);
\coordinate [] (2) at (0,-1);
\coordinate [] (1) at (0,1);
\coordinate [] (0) at (0,0);

\draw[gray, opacity=0.5,fill=lightgray,fill opacity=0.5] (1)--(2)--(3)--cycle;
\draw[gray, opacity=0.5,fill=lightgray,fill opacity=0.5] (3)--(4)--(5)--cycle;
\draw[gray, opacity=0.5,fill=lightgray,fill opacity=0.5] (4)--(5)--(6)--cycle;

\end{scope}

\node[] at ({0.7*0.6},0.05) {$abc$};
\node[] at ({0.7*(2*sqrt(3)-0.6)},0.05) {$d$};
\node[] at ({0.7*(2*sqrt(3)+0.6)},0) {$e$};

\filldraw[darkgray] (5) circle (0.07);
\filldraw[darkgray] (4) circle (0.07);
\filldraw[darkgray] (3) circle (0.07);

\end{scope}

\begin{scope}[xshift=-20, yshift=-91]

\begin{scope}[scale=0.7, xshift=0, yshift=0]

\coordinate [] (6) at ({3*sqrt(3)},0);
\coordinate [] (5) at ({2*sqrt(3)},-1);
\coordinate [] (4) at ({2*sqrt(3)},1);
\coordinate [] (3) at ({sqrt(3)},0);
\coordinate [] (2) at (0,-1);
\coordinate [] (1) at (0,1);
\coordinate [] (0) at (0,0);

\draw[gray, opacity=0.5,fill=lightgray,fill opacity=0.5] (1)--(2)--(3)--cycle;
\draw[gray, opacity=0.5,fill=lightgray,fill opacity=0.5] (3)--(4)--(5)--cycle;
\draw[gray, opacity=0.5,fill=lightgray,fill opacity=0.5] (4)--(5)--(6)--cycle;

\end{scope}

\node[] at ({0.7*0.6},0) {$a$};
\node[] at ({0.7*(2*sqrt(3)-0.6)},0.05) {$bcd$};
\node[] at ({0.7*(2*sqrt(3)+0.6)},0) {$e$};

\filldraw[darkgray] (5) circle (0.07);
\filldraw[darkgray] (4) circle (0.07);
\filldraw[darkgray] (3) circle (0.07);

\end{scope}

\begin{scope}[xshift=130, yshift=-91]

\begin{scope}[scale=0.7, xshift=0, yshift=0]

\coordinate [] (6) at ({3*sqrt(3)},0);
\coordinate [] (5) at ({2*sqrt(3)},-1);
\coordinate [] (4) at ({2*sqrt(3)},1);
\coordinate [] (3) at ({sqrt(3)},0);
\coordinate [] (2) at (0,-1);
\coordinate [] (1) at (0,1);
\coordinate [] (0) at (0,0);

\draw[gray, opacity=0.5,fill=lightgray,fill opacity=0.5] (1)--(2)--(3)--cycle;
\draw[gray, opacity=0.5,fill=lightgray,fill opacity=0.5] (3)--(4)--(5)--cycle;
\draw[gray, opacity=0.5,fill=lightgray,fill opacity=0.5] (4)--(5)--(6)--cycle;

\end{scope}

\node[] at ({0.7*0.6},0) {$a$};
\node[] at ({0.7*(2*sqrt(3)-0.6)},0.05) {$b$};
\node[] at ({0.7*(2*sqrt(3)+0.6)},0.05) {$cde$};

\filldraw[darkgray] (5) circle (0.07);
\filldraw[darkgray] (4) circle (0.07);
\filldraw[darkgray] (3) circle (0.07);

\end{scope}

\begin{scope}[xshift=20, yshift=-220]

\coordinate [] (6) at ({3*sqrt(3)},0);
\coordinate [] (5) at ({2*sqrt(3)},-1);
\coordinate [] (4) at ({2*sqrt(3)},1);
\coordinate [] (3) at ({sqrt(3)},0);
\coordinate [] (2) at (0,-1);
\coordinate [] (1) at (0,1);
\coordinate [] (0) at (0,0);

\draw[gray, opacity=0.5,fill=lightgray,fill opacity=0.5] (1)--(2)--(3)--cycle;

\node[] at (0.58,0.05) {$abcde$};

\end{scope}

\draw[-stealth, line width=1pt] (-1.5,-0.5) to[out=210, in=80] (-2.5,-2);
\draw[-stealth, line width=1pt] (4.2,-0.5) to[out=-30, in=100] (5.2,-2);
\draw[-stealth, line width=1pt] (-2.5,-4.5) to[out=-80, in=150] (-1.5,-6);
\draw[-stealth, line width=1pt] (5.2,-4.5) to[out=-100, in=30] (4.2,-6);
\draw[-stealth, line width=1pt] (1,-0.5) to (1,-2);
\draw[-stealth, line width=1pt] (1,-4.5) to (1,-6);

\begin{scope}[scale=0.15, xshift=-700, yshift=-150]

\coordinate [] (9) at ({5*sqrt(3)},0);
\coordinate [] (8) at ({4*sqrt(3)},-1);
\coordinate [] (7) at ({4*sqrt(3)},1);
\coordinate [] (6) at ({3*sqrt(3)},0);
\coordinate [] (5) at ({2*sqrt(3)},-1);
\coordinate [] (4) at ({2*sqrt(3)},1);
\coordinate [] (3) at ({sqrt(3)},0);
\coordinate [] (2) at (0,-1);
\coordinate [] (1) at (0,1);
\coordinate [] (0) at (0,0);

\draw[gray, opacity=0.5,fill=darkgray,fill opacity=0.5] (1)--(2)--(3)--cycle;
\draw[gray, opacity=0.5,fill=darkgray,fill opacity=0.5] (3)--(4)--(5)--cycle;
\draw[gray, opacity=0.5,fill=darkgray,fill opacity=0.5] (4)--(5)--(6)--cycle;
\draw[gray, opacity=0.5,fill=lightgray,fill opacity=0.5] (6)--(7)--(8)--cycle;
\draw[gray, opacity=0.5,fill=lightgray,fill opacity=0.5] (8)--(7)--(9)--cycle;

\end{scope}

\begin{scope}[scale=0.15, xshift=-700, yshift=-1000]

\coordinate [] (6) at ({3*sqrt(3)},0);
\coordinate [] (5) at ({2*sqrt(3)},-1);
\coordinate [] (4) at ({2*sqrt(3)},1);
\coordinate [] (3) at ({sqrt(3)},0);
\coordinate [] (2) at (0,-1);
\coordinate [] (1) at (0,1);
\coordinate [] (0) at (0,0);

\draw[gray, opacity=0.5,fill=darkgray,fill opacity=0.5] (1)--(2)--(3)--cycle;
\draw[gray, opacity=0.5,fill=darkgray,fill opacity=0.5] (3)--(4)--(5)--cycle;
\draw[gray, opacity=0.5,fill=darkgray,fill opacity=0.5] (4)--(5)--(6)--cycle;

\end{scope}

\begin{scope}[scale=0.15, xshift=270, yshift=-150]

\coordinate [] (9) at ({5*sqrt(3)},0);
\coordinate [] (8) at ({4*sqrt(3)},-1);
\coordinate [] (7) at ({4*sqrt(3)},1);
\coordinate [] (6) at ({3*sqrt(3)},0);
\coordinate [] (5) at ({2*sqrt(3)},-1);
\coordinate [] (4) at ({2*sqrt(3)},1);
\coordinate [] (3) at ({sqrt(3)},0);
\coordinate [] (2) at (0,-1);
\coordinate [] (1) at (0,1);
\coordinate [] (0) at (0,0);

\draw[gray, opacity=0.5,fill=lightgray,fill opacity=0.5] (1)--(2)--(3)--cycle;
\draw[gray, opacity=0.5,fill=darkgray,fill opacity=0.5] (3)--(4)--(5)--cycle;
\draw[gray, opacity=0.5,fill=darkgray,fill opacity=0.5] (4)--(5)--(6)--cycle;
\draw[gray, opacity=0.5,fill=darkgray,fill opacity=0.5] (6)--(7)--(8)--cycle;
\draw[gray, opacity=0.5,fill=lightgray,fill opacity=0.5] (8)--(7)--(9)--cycle;

\end{scope}

\begin{scope}[scale=0.15, xshift=270, yshift=-1000]

\coordinate [] (6) at ({3*sqrt(3)},0);
\coordinate [] (5) at ({2*sqrt(3)},-1);
\coordinate [] (4) at ({2*sqrt(3)},1);
\coordinate [] (3) at ({sqrt(3)},0);
\coordinate [] (2) at (0,-1);
\coordinate [] (1) at (0,1);
\coordinate [] (0) at (0,0);

\draw[gray, opacity=0.5,fill=darkgray,fill opacity=0.5] (1)--(2)--(3)--cycle;
\draw[gray, opacity=0.5,fill=darkgray,fill opacity=0.5] (3)--(4)--(5)--cycle;
\draw[gray, opacity=0.5,fill=darkgray,fill opacity=0.5] (4)--(5)--(6)--cycle;

\end{scope}

\begin{scope}[scale=0.15, xshift=1000, yshift=-150]

\coordinate [] (9) at ({5*sqrt(3)},0);
\coordinate [] (8) at ({4*sqrt(3)},-1);
\coordinate [] (7) at ({4*sqrt(3)},1);
\coordinate [] (6) at ({3*sqrt(3)},0);
\coordinate [] (5) at ({2*sqrt(3)},-1);
\coordinate [] (4) at ({2*sqrt(3)},1);
\coordinate [] (3) at ({sqrt(3)},0);
\coordinate [] (2) at (0,-1);
\coordinate [] (1) at (0,1);
\coordinate [] (0) at (0,0);

\draw[gray, opacity=0.5,fill=lightgray,fill opacity=0.5] (1)--(2)--(3)--cycle;
\draw[gray, opacity=0.5,fill=lightgray,fill opacity=0.5] (3)--(4)--(5)--cycle;
\draw[gray, opacity=0.5,fill=darkgray,fill opacity=0.5] (4)--(5)--(6)--cycle;
\draw[gray, opacity=0.5,fill=darkgray,fill opacity=0.5] (6)--(7)--(8)--cycle;
\draw[gray, opacity=0.5,fill=darkgray,fill opacity=0.5] (8)--(7)--(9)--cycle;

\end{scope}

\begin{scope}[scale=0.15, xshift=1000, yshift=-1000]

\coordinate [] (6) at ({3*sqrt(3)},0);
\coordinate [] (5) at ({2*sqrt(3)},-1);
\coordinate [] (4) at ({2*sqrt(3)},1);
\coordinate [] (3) at ({sqrt(3)},0);
\coordinate [] (2) at (0,-1);
\coordinate [] (1) at (0,1);
\coordinate [] (0) at (0,0);

\draw[gray, opacity=0.5,fill=darkgray,fill opacity=0.5] (1)--(2)--(3)--cycle;
\draw[gray, opacity=0.5,fill=darkgray,fill opacity=0.5] (3)--(4)--(5)--cycle;
\draw[gray, opacity=0.5,fill=darkgray,fill opacity=0.5] (4)--(5)--(6)--cycle;

\end{scope}

\end{tikzpicture}
\end{center}
Ignoring the labels for a moment, simply at the level of plex diagrams (hypergraphs), the long fish is clearly a confluent rewrite system with a single rewrite rule given by the fish product. Furthermore, plexes are of order $3$ throughout the rewrite system, hence it is an order-regular rewrite system, and the three fish motifs that can be detected on the long fish (initial hypergraph) have overlapping plexes. These are precisely the requirements we demanded for a concurrent rewrite system in Section \ref{confl}. Recall now that the labelling convention for the fish product has been chosen so that adjacency is preserved. This means that only the label placement `from-side-to-vertex' matters in the triangle representing the plex product, in correspondence with the sequential orientation of the tail-body-head of the fish plex diagram. A diagrammatic consequence of this labelling convention is the following:
\begin{center}
\begin{tikzpicture}[line join = round, line cap = round]

\begin{scope}[xshift=0, yshift=0]

\begin{scope}[rotate=180, scale=0.7, xshift=0, yshift=0]

\coordinate [] (6) at ({3*sqrt(3)},0);
\coordinate [] (5) at ({2*sqrt(3)},-1);
\coordinate [] (4) at ({2*sqrt(3)},1);
\coordinate [] (3) at ({sqrt(3)},0);
\coordinate [] (2) at (0,-1);
\coordinate [] (1) at (0,1);
\coordinate [] (0) at (0,0);

\draw[gray, opacity=0.5,fill=lightgray,fill opacity=0.5] (1)--(2)--(3)--cycle;

\end{scope}

\node[] at (-0.4,0.04) {$bcd$};

\end{scope}

\begin{scope}[xshift=50, yshift=0]

\node[] at (-0.9,0) {$=$};

\begin{scope}[scale=0.7, xshift=0, yshift=0]

\coordinate [] (6) at ({3*sqrt(3)},0);
\coordinate [] (5) at ({2*sqrt(3)},-1);
\coordinate [] (4) at ({2*sqrt(3)},1);
\coordinate [] (3) at ({sqrt(3)},0);
\coordinate [] (2) at (0,-1);
\coordinate [] (1) at (0,1);
\coordinate [] (0) at (0,0);

\draw[gray, opacity=0.5,fill=lightgray,fill opacity=0.5] (1)--(2)--(3)--cycle;

\end{scope}

\node[] at ({0.7*0.6},0.04) {$dcb$};

\end{scope}

\end{tikzpicture}
\end{center}
If we write the three labels of the plexes in the fish product in a sequential order
\begin{center}
\begin{tikzpicture}[line join = round, line cap = round]

\begin{scope}[xshift=0, yshift=0]

\begin{scope}[scale=0.7, xshift=0, yshift=0]

\coordinate [] (6) at ({3*sqrt(3)},0);
\coordinate [] (5) at ({2*sqrt(3)},-1);
\coordinate [] (4) at ({2*sqrt(3)},1);
\coordinate [] (3) at ({sqrt(3)},0);
\coordinate [] (2) at (0,-1);
\coordinate [] (1) at (0,1);
\coordinate [] (0) at (0,0);

\draw[gray, opacity=0.5,fill=lightgray,fill opacity=0.5] (1)--(2)--(3)--cycle;

\end{scope}

\node[] at ({0.7*0.6},0.04) {$abc$};

\end{scope}

\begin{scope}[xshift=85, yshift=0]

\node[] at (-0.9,0) {$\equiv$};

\node[] at ({0.7*0.6},0) {$(abc)$};

\end{scope}

\end{tikzpicture}
\end{center}
then, the fact that the three intermediate rewrite steps evaluate to the same final plex gives the desired identity:
\begin{equation*}
    ((abc)de)=(a(dcb)e)=(ab(cde)).
\end{equation*}

\end{proof}
When we consider the universe of index sets $\mathcal{U}$ and the set of $3$-plexes on constellations thereof, the notion of fish composition gives a ternary analogue of the composition operation in partial algebraic structures such as semigroupoids or categories. This will lead to the definition of what will be called a \textbf{semiheapoid} in Section \ref{heapoids}. In this view, objects are identified with index sets, the morphisms are identified with $3$-plexes and the analogue of the associativity axiom for composition of morphisms is the semiheap property (\ref{semiheap}) which holds for concatenations of an odd number of $3$-plexes forming the long fish pattern. This analogy is cemented by the following result.

\begin{prop}[Fish Semiheaps] \label{fishsemiheap}
Let three index sets $I, J, K\in \mathcal{U}$ and denote the set of all $3$-plexes on them by $A_{IJK}$, then $A_{IJK}$ is closed under the fish product and it inherits $6=|S_3|$ semiheap structures $\eta_{IJK}$ related by index permutations.
\end{prop}
\begin{proof}
This result follows from the observation that $3$-plexes on $3$ distinct indices can generically satisfy the conformability conditions of the fish product. This is easily seen diagrammatically, for instance, given the $3$-plexes $a, b, c \in A_{IJK}$ we can form the following fish product:
\begin{center}
\begin{tikzpicture}[line join = round, line cap = round]

\begin{scope}[scale=1.2, xshift=0, yshift=-10]

\coordinate [label=above:$K$] (2) at (0.5,{0.5*sqrt(3)});
\coordinate [label= below right:$J$] (1) at (1,0);
\coordinate [label=below left:$I$] (0) at (0,0);

\draw[lightgray, opacity=0.5,fill=lightgray,fill opacity=0.5] (1)--(0)--(2)--cycle;

\end{scope}

\begin{scope}[scale=1.2, xshift=60, yshift=-10]

\coordinate [label=above:$K$] (2) at (0.5,{0.5*sqrt(3)});
\coordinate [label= below right:$J$] (1) at (1,0);
\coordinate [label=below left:$I$] (0) at (0,0);

\draw[lightgray, opacity=0.5,fill=lightgray,fill opacity=0.5] (1)--(0)--(2)--cycle;

\end{scope}

\begin{scope}[scale=1.2, xshift=120, yshift=-10]

\coordinate [label=above:$K$] (2) at (0.5,{0.5*sqrt(3)});
\coordinate [label= below right:$J$] (1) at (1,0);
\coordinate [label=below left:$I$] (0) at (0,0);

\draw[lightgray, opacity=0.5,fill=lightgray,fill opacity=0.5] (1)--(0)--(2)--cycle;

\end{scope}

\begin{scope}[xshift=250, yshift=0]

\begin{scope}[scale=0.7, xshift=0, yshift=0]

\coordinate [label=right:$K$] (6) at ({3*sqrt(3)},0);
\coordinate [label=below:$J$] (5) at ({2*sqrt(3)},-1);
\coordinate [label=above:$I$] (4) at ({2*sqrt(3)},1);
\coordinate [label=above:$K$] (3) at ({sqrt(3)},0);
\coordinate [label=below left:$J$] (2) at (0,-1);
\coordinate [label=above left:$I$] (1) at (0,1);
\coordinate [] (0) at (0,0);

\draw[gray, opacity=0.5,fill=lightgray,fill opacity=0.5] (1)--(2)--(3)--cycle;
\draw[gray, opacity=0.5,fill=lightgray,fill opacity=0.5] (3)--(4)--(5)--cycle;
\draw[gray, opacity=0.5,fill=lightgray,fill opacity=0.5] (4)--(5)--(6)--cycle;

\end{scope}

\node[] at ({0.7*0.6},0) {$a$};
\node[] at ({0.7*(2*sqrt(3)-0.6)},0.05) {$b$};
\node[] at ({0.7*(2*sqrt(3)+0.6)},0) {$c$};

\filldraw[darkgray] (5) circle (0.07);
\filldraw[darkgray] (4) circle (0.07);
\filldraw[darkgray] (3) circle (0.07);

\end{scope}

\node[] at (0.6,0) {$a$};
\node[] at (3.15,0.05) {$b$};
\node[] at (5.65,0) {$c$};

\node[] at (7.51,0) {$\mapsto$};

\end{tikzpicture}
\end{center}
which is clearly again a $3$-plex on the same indices and thus an element of $A_{IJK}$. The ternary operation defined by this particular assignment of plexes is denoted as $\eta_{IJK}: A_{IJK}^3 \to A_{IJK}$. There are as many ternary operations of this kind on $A_{IJK}$ as there are permutations of the external indices in the fish plex diagram. In other words, we can define as many formally distinct ternary operations on $A_{IJK}$ as there are ways to place the index sets $I$, $J$, $K$ on the outer vertices of the fish hypergraph. These correspond to all the bijections of the form
\begin{equation*}
    \{I,J,K\} \to \{\text{tip}_1, \text{tip}_2, \text{mouth}\}.
\end{equation*}
and therefore there are a total of $|S_3|=6$. We can account for al the different compositional topologies resulting from the permutations of indices diagrammatically. A convenient way to enumerate the different fish products that are allowed is by writing all the possible fish diagrams centered on a single generic plex that acts as head or tail and progressively attach two more plexes to complete a fish outwards:
\begin{center}
\begin{tikzpicture}[line join = round, line cap = round]

\begin{scope}[rotate=-30, xshift=0, yshift=0]

\coordinate [label=below:$J$] (9) at ({5*sqrt(3)},0);
\coordinate [] (8) at ({4*sqrt(3)},-1);
\coordinate [] (7) at ({4*sqrt(3)},1);
\coordinate [] (6) at ({3*sqrt(3)},0);
\coordinate [] (5) at ({2*sqrt(3)},-1);
\coordinate [] (4) at ({2*sqrt(3)},1);
\coordinate [] (3) at ({sqrt(3)},0);
\coordinate [label=left:$I$] (2) at (0,-1);
\coordinate [label=above left:$K$] (1) at (0,1);
\coordinate [] (0) at (0,0);
\coordinate [label=above left:$\eta_{IKJ}$] (A) at ({-0.5*sqrt(3)},0);
\coordinate [label=below right:$\eta_{KIJ}$] (B) at ({5.5*sqrt(3)},0);

\draw[gray, opacity=0.5,fill=lightgray,fill opacity=0.5] (1)--(2)--(3)--cycle;
\draw[gray, opacity=0.5,fill=lightgray,fill opacity=0.5] (3)--(4)--(5)--cycle;
\draw[gray, opacity=0.5,fill=lightgray,fill opacity=0.5] (4)--(5)--(6)--cycle;
\draw[gray, opacity=0.5,fill=lightgray,fill opacity=0.5] (6)--(7)--(8)--cycle;
\draw[gray, opacity=0.5,fill=lightgray,fill opacity=0.5] (8)--(7)--(9)--cycle;

\draw[gray, thick, dashed, stealth-stealth] (B)--(A);

\end{scope}

\begin{scope}[rotate=90, xshift=-172.5, yshift=-99.5]

\coordinate [label=above right:$K$] (9) at ({5*sqrt(3)},0);
\coordinate [] (8) at ({4*sqrt(3)},-1);
\coordinate [] (7) at ({4*sqrt(3)},1);
\coordinate [] (6) at ({3*sqrt(3)},0);
\coordinate [] (5) at ({2*sqrt(3)},-1);
\coordinate [] (4) at ({2*sqrt(3)},1);
\coordinate [] (3) at ({sqrt(3)},0);
\coordinate [label=below right:$J$] (2) at (0,-1);
\coordinate [label=below left:$I$] (1) at (0,1);
\coordinate [] (0) at (0,0);
\coordinate [label=below:$\eta_{JIK}$] (A) at ({-0.5*sqrt(3)},0);
\coordinate [label=above:$\eta_{IJK}$] (B) at ({5.5*sqrt(3)},0);

\draw[gray, opacity=0.5,fill=lightgray,fill opacity=0.5] (1)--(2)--(3)--cycle;
\draw[gray, opacity=0.5,fill=lightgray,fill opacity=0.5] (3)--(4)--(5)--cycle;

\draw[gray, opacity=0.5,fill=lightgray,fill opacity=0.5] (6)--(7)--(8)--cycle;
\draw[gray, opacity=0.5,fill=lightgray,fill opacity=0.5] (8)--(7)--(9)--cycle;

\draw[gray, thick, dashed, stealth-stealth] (B)--(A);

\end{scope}

\begin{scope}[rotate=-150, xshift=-172.5, yshift=99.5]

\coordinate [label=above left:$I$] (9) at ({5*sqrt(3)},0);
\coordinate [] (8) at ({4*sqrt(3)},-1);
\coordinate [] (7) at ({4*sqrt(3)},1);
\coordinate [] (6) at ({3*sqrt(3)},0);
\coordinate [] (5) at ({2*sqrt(3)},-1);
\coordinate [] (4) at ({2*sqrt(3)},1);
\coordinate [] (3) at ({sqrt(3)},0);
\coordinate [label=above:$J$] (2) at (0,-1);
\coordinate [label=right:$K$] (1) at (0,1);
\coordinate [] (0) at (0,0);
\coordinate [label=above right:$\eta_{KJI}$] (A) at ({-0.5*sqrt(3)},0);
\coordinate [label=below left:$\eta_{JKI}$] (B) at ({5.5*sqrt(3)},0);

\draw[gray, opacity=0.5,fill=lightgray,fill opacity=0.5] (1)--(2)--(3)--cycle;
\draw[gray, opacity=0.5,fill=lightgray,fill opacity=0.5] (3)--(4)--(5)--cycle;

\draw[gray, opacity=0.5,fill=lightgray,fill opacity=0.5] (6)--(7)--(8)--cycle;
\draw[gray, opacity=0.5,fill=lightgray,fill opacity=0.5] (8)--(7)--(9)--cycle;

\draw[gray, thick, dashed, stealth-stealth] (B)--(A);

\end{scope}

\end{tikzpicture}
\end{center}
where we have notated the sufficient indices that determine the index labels of the remaining unspecified vertices of the diagram by the conformability constraints of the fish product. This diagram should be read as follows: start at the central plex $a$ and take one of the dashed arrows pointing outwards, then attach a second plex $b$ in the position of the first triangle that the dashed arrow intersects in the outwards direction, and, lastly, attach a third plex $c$ in the position of the triangle that dashed arrow next intersects in the outwards direction. It follows from Theorem \ref{fishassoc} that each of these ternary operations $\eta_{\cdot \cdot \cdot}$ satisfies the semiheap property (\ref{semiheap}) and thus we find $6$ generically distinct semiheap structures on $A_{IJK}$.
\end{proof}
An important observation that follows from Proposition \ref{fishsemiheap} is that the sequential order of plexes in a fish product can be taken head-to-tail or tail-to-head resulting in two formally distinct semiheap structures. However, the fact that they share the same underlying compositional topology, i.e. they are computed by the same plex diagram only read in one of two possible orientations, manifests as the fact that the two resulting semiheap structures are related by reversal. Indeed, following the notation above, it is easy to check that:
\begin{equation*}
    \eta_{IJK}=\overline{\eta}_{JIK} \qquad \eta_{KIJ}=\overline{\eta}_{IKJ} \qquad \eta_{JKI}=\overline{\eta}_{KJI}
\end{equation*}
Up to reversal, however, there are $3$ generically distinct semiheap structures on the set of $3$-plexes on a fixed constellation of $3$ indices which all differ by the topology of internally contracted indices in the fish diagram and thus cannot be related in any natural way:
\begin{equation*}
    \eta_{IJK} \neq \eta_{JKI} \neq \eta_{KIJ}.
\end{equation*}
If we further assume the regularity condition that two indices are equal, e.g. $I=J$, the internal contractions of the body plex with the head plex in the fish diagram are subject to twisting freedom as described in Section \ref{plex}. This implies that for each fish semiheap structure of the form $\eta_{JJK}$ there is another generically inequivalent semiheap structure $\eta_{JJK}'$ given by the twisting.\newline

The sequential nature of the fish product stems from the fact that it can be obtained as a concatenation of two binary plex operations:
\begin{center}
\begin{tikzpicture}[line join = round, line cap = round]

\begin{scope}[scale=1.2, xshift=0, yshift=-10]

\coordinate [] (2) at (0.5,{0.5*sqrt(3)});
\coordinate [] (1) at (1,0);
\coordinate [] (0) at (0,0);

\draw[lightgray, opacity=0.5,fill=lightgray,fill opacity=0.5] (1)--(0)--(2)--cycle;

\end{scope}

\begin{scope}[scale=1.2, xshift=40, yshift=-10]

\coordinate [] (2) at (0.5,{0.5*sqrt(3)});
\coordinate [] (1) at (1,0);
\coordinate [] (0) at (0,0);

\draw[lightgray, opacity=0.5,fill=lightgray,fill opacity=0.5] (1)--(0)--(2)--cycle;

\end{scope}

\begin{scope}[scale=1.2, xshift=80, yshift=-10]

\coordinate [] (2) at (0.5,{0.5*sqrt(3)});
\coordinate [] (1) at (1,0);
\coordinate [] (0) at (0,0);

\draw[lightgray, opacity=0.5,fill=lightgray,fill opacity=0.5] (1)--(0)--(2)--cycle;

\end{scope}

\begin{scope}[xshift=175, yshift=0]

\begin{scope}[scale=0.7, xshift=0, yshift=0]

\coordinate [] (6) at ({3*sqrt(3)},0);
\coordinate [] (5) at ({2*sqrt(3)},-1);
\coordinate [] (4) at ({2*sqrt(3)},1);
\coordinate [] (3) at ({sqrt(3)},0);
\coordinate [] (2) at (0,-1);
\coordinate [] (1) at (0,1);
\coordinate [] (0) at (0,0);

\draw[gray, opacity=0.5,fill=lightgray,fill opacity=0.5] (1)--(2)--(3)--cycle;
\draw[gray, opacity=0.5,fill=lightgray,fill opacity=0.5] (3)--(4)--(5)--cycle;

\end{scope}

\begin{scope}[scale=1.2, xshift=75, yshift=-10]

\coordinate [] (2) at (0.5,{0.5*sqrt(3)});
\coordinate [] (1) at (1,0);
\coordinate [] (0) at (0,0);

\draw[lightgray, opacity=0.5,fill=lightgray,fill opacity=0.5] (1)--(0)--(2)--cycle;

\end{scope}

\node[] at ({0.7*0.6},0) {$a$};
\node[] at ({0.7*(2*sqrt(3)-0.6)},0.05) {$b$};
\node[] at (3.75,0) {$c$};

\filldraw[darkgray] (3) circle (0.07);

\end{scope}

\begin{scope}[xshift=350, yshift=0]

\begin{scope}[scale=0.7, xshift=0, yshift=0]

\coordinate [] (6) at ({3*sqrt(3)},0);
\coordinate [] (5) at ({2*sqrt(3)},-1);
\coordinate [] (4) at ({2*sqrt(3)},1);
\coordinate [] (3) at ({sqrt(3)},0);
\coordinate [] (2) at (0,-1);
\coordinate [] (1) at (0,1);
\coordinate [] (0) at (0,0);

\draw[gray, opacity=0.5,fill=lightgray,fill opacity=0.5] (1)--(2)--(3)--cycle;
\draw[gray, opacity=0.5,fill=lightgray,fill opacity=0.5] (3)--(4)--(5)--cycle;
\draw[gray, opacity=0.5,fill=lightgray,fill opacity=0.5] (4)--(5)--(6)--cycle;

\end{scope}

\node[] at ({0.7*0.6},0) {$a$};
\node[] at ({0.7*(2*sqrt(3)-0.6)},0.05) {$b$};
\node[] at ({0.7*(2*sqrt(3)+0.6)},0) {$c$};

\filldraw[darkgray] (5) circle (0.07);
\filldraw[darkgray] (4) circle (0.07);
\filldraw[darkgray] (3) circle (0.07);

\end{scope}

\node[] at (0.6,0) {$a$};
\node[] at (2.3,0.05) {$b$};
\node[] at (3.95,0) {$c$};

\node[] at (5.4,0) {$\mapsto$};
\node[] at (11.4,0) {$\mapsto$};

\end{tikzpicture}
\end{center}
The binary products of $3$-plexes used in each step correspond to the following plex diagrams:
\begin{center}
\begin{tikzpicture}[line join = round, line cap = round]

\begin{scope}[xshift=0, yshift=0]

\begin{scope}[scale=0.7, xshift=0, yshift=0]

\coordinate [] (6) at ({3*sqrt(3)},0);
\coordinate [] (5) at ({2*sqrt(3)},-1);
\coordinate [] (4) at ({2*sqrt(3)},1);
\coordinate [] (3) at ({sqrt(3)},0);
\coordinate [] (2) at (0,-1);
\coordinate [] (1) at (0,1);
\coordinate [] (0) at (0,0);

\draw[gray, opacity=0.5,fill=lightgray,fill opacity=0.5] (1)--(2)--(3)--cycle;
\draw[gray, opacity=0.5,fill=lightgray,fill opacity=0.5] (3)--(4)--(5)--cycle;

\end{scope}

\filldraw[darkgray] (3) circle (0.07);

\end{scope}

\begin{scope}[xshift=130, yshift=0]

\begin{scope}[scale=0.7, xshift=0, yshift=0]

\coordinate [] (6) at ({3*sqrt(3)},0);
\coordinate [] (5) at ({2*sqrt(3)},-1);
\coordinate [] (4) at ({2*sqrt(3)},1);
\coordinate [] (3) at ({sqrt(3)},0);
\coordinate [] (2) at (0,-1);
\coordinate [] (1) at (0,1);
\coordinate [] (0) at (0,0);

\draw[gray, opacity=0.5,fill=lightgray,fill opacity=0.5] (3)--(4)--(5)--cycle;
\draw[gray, opacity=0.5,fill=lightgray,fill opacity=0.5] (4)--(5)--(6)--cycle;

\end{scope}

\filldraw[darkgray] (5) circle (0.07);
\filldraw[darkgray] (4) circle (0.07);

\end{scope}

\end{tikzpicture}
\end{center}
which we identify as the closest analogues of binary composition that can be defined for $3$-plexes. We make this idea precise via the index flattening operation: given a $3$-plex $a$ on indices $\{I,J,K\}$ we can consider the pair of indices $J,K$ as a multi-index $M$ so $a$ is recaptured in a $2$-plex $\underline{a}$ on indices $\{I,M\}$
\begin{center}
\begin{tikzpicture}[line join = round, line cap = round]

\begin{scope}[xshift=0, yshift=0]

\begin{scope}[rotate=180, scale=0.7, xshift=0, yshift=0]

\coordinate [] (6) at ({3*sqrt(3)},0);
\coordinate [] (5) at ({2*sqrt(3)},-1);
\coordinate [] (4) at ({2*sqrt(3)},1);
\coordinate [label=left:$I$] (3) at ({sqrt(3)},0);
\coordinate [label=above right:$J$] (2) at (0,-1);
\coordinate [label=below right:$K$] (1) at (0,1);
\coordinate [] (0) at (0,0);

\draw[gray, opacity=0.5,fill=lightgray,fill opacity=0.5] (1)--(2)--(3)--cycle;

\end{scope}

\node[] at (-0.4,0.04) {$a$};

\end{scope}

\begin{scope}[xshift=85, yshift=0]

\node[] at (-2,0) {$\leadsto$};

\coordinate [label=right:$M$] (2) at (1,0);
\coordinate [label=left:$I$] (0) at (-1,0);

\draw[lightgray, ultra thick] (0)--(2);

\node[] at (0,0.4) {$\underline{a}$};

\end{scope}

\end{tikzpicture}
\end{center}
Although this change is immaterial for the data of the plex $a$ itself, the compositional properties of the the flattened plex $\underline{a}$ are modified. While the $I$ index operates as usual, $M$ is a multi-index and thus it forces incidences and contractions on pairs of indices. To denote this compositional asymmetry it suffices to notationally distinguish the single index vertex from the multi-index vertex of the plex $\underline{a}$. This can be simply accomplished by indicating an orientation along the edge representing the flattened $2$-plex $\underline{a}$ pointing from the multi-index to the single index:
\begin{center}
\begin{tikzpicture}[line join = round, line cap = round]

\begin{scope}[xshift=0, yshift=0]

\begin{scope}[rotate=180, scale=0.7, xshift=0, yshift=0]

\coordinate [] (6) at ({3*sqrt(3)},0);
\coordinate [] (5) at ({2*sqrt(3)},-1);
\coordinate [] (4) at ({2*sqrt(3)},1);
\coordinate [label=left:$I$] (3) at ({sqrt(3)},0);
\coordinate [label=above right:$J$] (2) at (0,-1);
\coordinate [label=below right:$K$] (1) at (0,1);
\coordinate [] (0) at (0,0);

\draw[gray, opacity=0.5,fill=lightgray,fill opacity=0.5] (1)--(2)--(3)--cycle;

\end{scope}

\node[] at (-0.4,0.04) {$a$};

\end{scope}

\begin{scope}[xshift=85, yshift=0]

\node[] at (-2,0) {$\equiv$};

\coordinate [label=right:$M$] (2) at (1,0);
\coordinate [label=left:$I$] (0) at (-1,0);

\begin{scope}[very thick,decoration={markings, mark=at position 0.6 with {\arrow{latex}}}]

\draw[lightgray, ultra thick, postaction={decorate}] (2)--(0);

\end{scope}

\node[] at (0,0.4) {$\underline{a}$};

\end{scope}

\end{tikzpicture}
\end{center}
The two binary products of $3$-plexes above can now be rewritten in terms of flattened $2$-plexes
\begin{center}
\begin{tikzpicture}[line join = round, line cap = round]

\begin{scope}[xshift=0, yshift=0]

\begin{scope}[scale=0.7, xshift=0, yshift=0]

\coordinate [] (4) at ({2*sqrt(3)},0);
\coordinate [] (3) at ({sqrt(3)},0);
\coordinate [] (2) at (0,-1);
\coordinate [] (1) at (0,1);
\coordinate [] (0) at (0,0);

\begin{scope}[very thick,decoration={markings, mark=at position 0.6 with {\arrow{latex}}}]

\draw[lightgray, ultra thick, postaction={decorate}] (0)--(3);
\draw[lightgray, ultra thick, postaction={decorate}] (4)--(3);

\end{scope}

\end{scope}

\filldraw[darkgray] (3) circle (0.07);

\end{scope}

\begin{scope}[xshift=130, yshift=0]

\begin{scope}[scale=0.7, xshift=0, yshift=0]

\coordinate [] (4) at ({2*sqrt(3)},0);
\coordinate [] (3) at ({sqrt(3)},0);
\coordinate [] (2) at (0,-1);
\coordinate [] (1) at (0,1);
\coordinate [] (0) at (0,0);

\begin{scope}[very thick,decoration={markings, mark=at position 0.6 with {\arrow{latex}}}]

\draw[lightgray, ultra thick, postaction={decorate}] (3)--(0);
\draw[lightgray, ultra thick, postaction={decorate}] (3)--(4);

\end{scope}

\end{scope}

\filldraw[darkgray] (3) circle (0.07);

\end{scope}

\end{tikzpicture}
\end{center}
These non-circulatory diagrams represent the only allowed binary contractions of flattened $2$-plexes, which contrasts with the circulatory nature of usual oriented binary compositions, e.g. the composition of morphisms in a category \cite{leinster2014basic}. Notably, each of these operations does not give a well-defined composition of flattened $2$-plexes since depending on the site of the index contraction the outcome is a $4$-plex (left diagram) or a $2$-plex (right diagram). However, when considered in combination, they give a characterization of the fish product.

\begin{prop}[Flat Fish Product] \label{flatfish}
The fish product is the minimal (lowest arity) composition on flattened $2$-plexes.
\end{prop}
\begin{proof}
This result follows simply by construction, as shown by the following diagrammatic equation:
\begin{center}
\begin{tikzpicture}[line join = round, line cap = round]

\begin{scope}[xshift=0, yshift=0]

\begin{scope}[scale=0.7, xshift=0, yshift=0]

\coordinate [] (5) at ({3*sqrt(3)},0);
\coordinate [] (4) at ({2*sqrt(3)},0);
\coordinate [] (3) at ({sqrt(3)},0);
\coordinate [] (2) at (0,-1);
\coordinate [] (1) at (0,1);
\coordinate [] (0) at (0,0);

\begin{scope}[very thick,decoration={markings, mark=at position 0.6 with {\arrow{latex}}}]

\draw[lightgray, ultra thick, postaction={decorate}] (0)--(3);
\draw[lightgray, ultra thick, postaction={decorate}] (4)--(3);
\draw[lightgray, ultra thick, postaction={decorate}] (4)--(5);

\end{scope}

\end{scope}

\filldraw[darkgray] (3) circle (0.07);
\filldraw[darkgray] (4) circle (0.07);

\end{scope}

\node[] at (5,0) {$=$};

\begin{scope}[xshift=180, yshift=0]

\begin{scope}[scale=0.7, xshift=0, yshift=0]

\coordinate [] (3) at ({sqrt(3)},0);
\coordinate [] (0) at (0,0);

\begin{scope}[very thick,decoration={markings, mark=at position 0.7 with {\arrow{latex}}}]

\draw[lightgray, ultra thick, postaction={decorate}] (0)--(3);

\end{scope}

\end{scope}

\end{scope}

\end{tikzpicture}
\end{center}
\end{proof}
The \emph{flattened perspective} on the fish product allows to regard $3$-plex semiheapoids as a partial algebraic structures articulated on directed graphs similarly to ordinary semigroupoids or categories. The key difference between semiheapoids and ordinary semigroupoids is the strictly non-circulatory nature of fish composition as shown in Proposition \ref{flatfish}. Furthermore, rephrasing Proposition \ref{fishsemiheap} in terms of the flat fish product allows to recover the semiheapoid analogue of the semiheap structure of dagger categories identified in the beginning of Section \ref{heaps}. To see this in detail, consider three flattened plexes $\underline{a}$, $\underline{b}$ and $\underline{c}$ defined on indices $I$ and $M$, where $M$ is a multi-index of two indices $J,K$. The conformability conditions of the fish product imply that the flattened plexes sit between the two indices according to the following non-circulatory diagram:
\begin{center}
\begin{tikzpicture}[line join = round, line cap = round]

\begin{scope}[scale=1.2, xshift=0, yshift=0]

\coordinate [label=right:$M$] (2) at (1,0);
\coordinate [label=left:$I$] (0) at (-1,0);

\begin{scope}[very thick,decoration={markings, mark=at position 0.6 with {\arrow{latex}}}]

\draw[lightgray, ultra thick, postaction={decorate}] (0) to[out=60, in=120] (2);
\draw[lightgray, ultra thick, postaction={decorate}] (2) to[out=180, in=0] (0);
\draw[lightgray, ultra thick, postaction={decorate}] (0) to[out=-60, in=-120] (2);

\end{scope}

\end{scope}

\node[] at (0,0.9) {$\underline{a}$};
\node[] at (0,0.3) {$\underline{b}$};
\node[] at (0,-0.3) {$\underline{c}$};

\end{tikzpicture}
\end{center}
This is indeed analogous to the commutative diagram that defines the semiheap structure on the morphism set of a dagger category (\ref{dag}). Note, however, that unlike the case of a dagger category, there is no involution operation in the fish product of flattened plexes. What the involution of a dagger category accomplishes for binary morphisms is enabled by the fact that plexes (before flattening) are defined on $3$ indices. The flat fish product appeared in the definition of pregroupoids in \cite{kock2007principal}.\newline

So far in this section we have shown how semiheapoids can be realized in a minimal way within plex algebra via the fish product. We now focus on the possibility of neutral elements for the fish product and the search for biunits in fish semiheaps. The general problem of the existence of neutral elements for the fish product can be formulated as follows: given an arbitrary $3$-plex $a$, are there canonical $3$-plexes $e$ and $e'$ such that they act as neutral elements in the fish product i.e. they satisfy some of the diagrammatic equations below?
\begin{center}
\begin{tikzpicture}[line join = round, line cap = round]

\begin{scope}[xshift=0, yshift=0]

\begin{scope}[scale=0.7, xshift=0, yshift=0]

\coordinate [] (6) at ({3*sqrt(3)},0);
\coordinate [] (5) at ({2*sqrt(3)},-1);
\coordinate [] (4) at ({2*sqrt(3)},1);
\coordinate [] (3) at ({sqrt(3)},0);
\coordinate [] (2) at (0,-1);
\coordinate [] (1) at (0,1);
\coordinate [] (0) at (0,0);

\draw[gray, opacity=0.5,fill=lightgray,fill opacity=0.5] (1)--(2)--(3)--cycle;
\draw[gray, opacity=0.5,fill=lightgray,fill opacity=0.5] (3)--(4)--(5)--cycle;
\draw[gray, opacity=0.5,fill=lightgray,fill opacity=0.5] (4)--(5)--(6)--cycle;

\end{scope}

\node[] at ({0.7*0.6},0) {$a$};
\node[] at ({0.7*(2*sqrt(3)-0.6)},0) {$e$};
\node[] at ({0.7*(2*sqrt(3)+0.6)},0.05) {$e'$};

\filldraw[darkgray] (5) circle (0.07);
\filldraw[darkgray] (4) circle (0.07);
\filldraw[darkgray] (3) circle (0.07);

\end{scope}

\begin{scope}[xshift=155, yshift=0]

\node[] at (-0.9,0) {$=$};

\begin{scope}[scale=0.7, xshift=0, yshift=0]

\coordinate [] (6) at ({3*sqrt(3)},0);
\coordinate [] (5) at ({2*sqrt(3)},-1);
\coordinate [] (4) at ({2*sqrt(3)},1);
\coordinate [] (3) at ({sqrt(3)},0);
\coordinate [] (2) at (0,-1);
\coordinate [] (1) at (0,1);
\coordinate [] (0) at (0,0);

\draw[gray, opacity=0.5,fill=lightgray,fill opacity=0.5] (1)--(2)--(3)--cycle;

\end{scope}

\node[] at ({0.7*0.6},0) {$a$};

\end{scope}

\end{tikzpicture}
\end{center}

\begin{center}
\begin{tikzpicture}[line join = round, line cap = round]

\begin{scope}[xshift=0, yshift=0]

\begin{scope}[scale=0.7, xshift=0, yshift=0]

\coordinate [] (6) at ({3*sqrt(3)},0);
\coordinate [] (5) at ({2*sqrt(3)},-1);
\coordinate [] (4) at ({2*sqrt(3)},1);
\coordinate [] (3) at ({sqrt(3)},0);
\coordinate [] (2) at (0,-1);
\coordinate [] (1) at (0,1);
\coordinate [] (0) at (0,0);

\draw[gray, opacity=0.5,fill=lightgray,fill opacity=0.5] (1)--(2)--(3)--cycle;
\draw[gray, opacity=0.5,fill=lightgray,fill opacity=0.5] (3)--(4)--(5)--cycle;
\draw[gray, opacity=0.5,fill=lightgray,fill opacity=0.5] (4)--(5)--(6)--cycle;

\end{scope}

\node[] at ({0.7*0.6},0) {$e$};
\node[] at ({0.7*(2*sqrt(3)-0.6)},0) {$a$};
\node[] at ({0.7*(2*sqrt(3)+0.6)},0.05) {$e'$};

\filldraw[darkgray] (5) circle (0.07);
\filldraw[darkgray] (4) circle (0.07);
\filldraw[darkgray] (3) circle (0.07);

\end{scope}

\begin{scope}[xshift=155, yshift=0]

\node[] at (-0.9,0) {$=$};

\begin{scope}[scale=0.7, xshift=0, yshift=0]

\coordinate [] (6) at ({3*sqrt(3)},0);
\coordinate [] (5) at ({2*sqrt(3)},-1);
\coordinate [] (4) at ({2*sqrt(3)},1);
\coordinate [] (3) at ({sqrt(3)},0);
\coordinate [] (2) at (0,-1);
\coordinate [] (1) at (0,1);
\coordinate [] (0) at (0,0);

\draw[gray, opacity=0.5,fill=lightgray,fill opacity=0.5] (1)--(2)--(3)--cycle;

\end{scope}

\node[] at ({0.7*0.6},0) {$a$};

\end{scope}

\end{tikzpicture}
\end{center}

\begin{center}
\begin{tikzpicture}[line join = round, line cap = round]

\begin{scope}[xshift=0, yshift=0]

\begin{scope}[scale=0.7, xshift=0, yshift=0]

\coordinate [] (6) at ({3*sqrt(3)},0);
\coordinate [] (5) at ({2*sqrt(3)},-1);
\coordinate [] (4) at ({2*sqrt(3)},1);
\coordinate [] (3) at ({sqrt(3)},0);
\coordinate [] (2) at (0,-1);
\coordinate [] (1) at (0,1);
\coordinate [] (0) at (0,0);

\draw[gray, opacity=0.5,fill=lightgray,fill opacity=0.5] (1)--(2)--(3)--cycle;
\draw[gray, opacity=0.5,fill=lightgray,fill opacity=0.5] (3)--(4)--(5)--cycle;
\draw[gray, opacity=0.5,fill=lightgray,fill opacity=0.5] (4)--(5)--(6)--cycle;

\end{scope}

\node[] at ({0.7*0.6},0) {$e$};
\node[] at ({0.7*(2*sqrt(3)-0.6)},0.05) {$e'$};
\node[] at ({0.7*(2*sqrt(3)+0.6)},0) {$a$};

\filldraw[darkgray] (5) circle (0.07);
\filldraw[darkgray] (4) circle (0.07);
\filldraw[darkgray] (3) circle (0.07);

\end{scope}

\begin{scope}[xshift=155, yshift=0]

\node[] at (-0.9,0) {$=$};

\begin{scope}[scale=0.7, xshift=0, yshift=0]

\coordinate [] (6) at ({3*sqrt(3)},0);
\coordinate [] (5) at ({2*sqrt(3)},-1);
\coordinate [] (4) at ({2*sqrt(3)},1);
\coordinate [] (3) at ({sqrt(3)},0);
\coordinate [] (2) at (0,-1);
\coordinate [] (1) at (0,1);
\coordinate [] (0) at (0,0);

\draw[gray, opacity=0.5,fill=lightgray,fill opacity=0.5] (1)--(2)--(3)--cycle;

\end{scope}

\node[] at ({0.7*0.6},0) {$a$};

\end{scope}

\end{tikzpicture}
\end{center}
Due to the asymmetric anatomy of the fish plex diagram each of the equations above has to be considered separately. In the most general situation, when $a$ is a $3$-plex on three distinct indices $I$, $J$, $K$, the canonical choices of $3$-plexes are of the following form:
\begin{center}
\begin{tikzpicture}[]

\begin{scope}[xshift=0, yshift=0]

\coordinate [label=above:$I$] (2a) at (0.5+0.03,{0.5*sqrt(3)});
\coordinate [] (2b) at (0.5-0.03,{0.5*sqrt(3)});
\coordinate [label=right:$I$] (1a) at ({1+0.03*0.5},{0.03*0.5*sqrt(3)});
\coordinate [] (1b) at ({1-0.03*0.5},{-0.03*0.5*sqrt(3)});
\coordinate [label=left:$I$] (0a) at ({0.03*0.5},{-0.03*0.5*sqrt(3)});
\coordinate [] (0b) at ({-0.03*0.5},{0.03*0.5*sqrt(3)});

\draw[lightgray, line width=1pt] (0b) to[out=30, in=270] (2b);
\draw[lightgray, line width=1pt] (2a) to[out=270, in=150] (1a);
\draw[lightgray, line width=1pt] (1b) to[out=150, in=30] (0a);

\node[] at (0.45,-1) {$3$-identity};

\end{scope}

\begin{scope}[xshift=100, yshift=0]

\coordinate [label=above:$J$] (2) at (0.5,{0.5*sqrt(3)});
\coordinate [label=right:$I$] (1) at (1,0);
\coordinate [label=left:$I$] (0) at (0,0);

\draw[lightgray, dashed] (1)--(0)--(2)--cycle;
\draw[lightgray, line width=1pt, double] (1)--(0);

\node[] at (0.45,-1) {broadened $2$-identity};

\end{scope}

\begin{scope}[xshift=220, yshift=0]

\coordinate [label=above:$K$] (2) at (0.5,{0.5*sqrt(3)});
\coordinate [label=right:$J$] (1) at (1,0);
\coordinate [label=left:$I$] (0) at (0,0);

\draw[lightgray, opacity=0.5,fill=lightgray,fill opacity=0.5] (1)--(0)--(2)--cycle;

\node[] at (0.5,0.3) {$0$};

\node[] at (0.45,-1) {empty $3$-plex};

\end{scope}

\begin{scope}[xshift=330, yshift=0]

\coordinate [label=above:$K$] (2) at (0.5,{0.5*sqrt(3)});
\coordinate [label=right:$J$] (1) at (1,0);
\coordinate [label=left:$I$] (0) at (0,0);

\draw[lightgray, opacity=0.5,fill=lightgray,fill opacity=0.5] (1)--(0)--(2)--cycle;

\node[] at (0.5,0.3) {$1$};

\node[] at (0.45,-1) {full $3$-plex};

\end{scope}

\end{tikzpicture}
\end{center}
The empty $3$-plex trivially acts as an absorbent element, as in any plex product, so it can be discarded as a candidate for neutral element. Contractions with the full $3$-plex are equivalent to self-contractions and thus they cannot act as neutral elements for the fish product if they are directly contracted with $a$. This leaves the $3$-identity and the broadened $2$-identities as the candidates for fish neutral elements. We proceed case by case:
\begin{itemize}
    \item $a$ as tail. This diagrammatic equation is equivalent to the condition that the body $e$ and the head $e'$ contract to the $2$-identity plex
    \begin{center}
\begin{tikzpicture}[]

\begin{scope}[xshift=0, yshift=0]

\begin{scope}[scale=0.7, xshift=0, yshift=0]

\coordinate [] (6) at ({3*sqrt(3)},0);
\coordinate [] (5) at ({2*sqrt(3)},-1);
\coordinate [] (4) at ({2*sqrt(3)},1);
\coordinate [] (3) at ({sqrt(3)},0);
\coordinate [] (2) at (0,-1);
\coordinate [] (1) at (0,1);
\coordinate [] (0) at (0,0);

\draw[gray, opacity=0.5,fill=lightgray,fill opacity=0.5] (3)--(4)--(5)--cycle;
\draw[gray, opacity=0.5,fill=lightgray,fill opacity=0.5] (4)--(5)--(6)--cycle;

\end{scope}

\filldraw[darkgray] (5) circle (0.07);
\filldraw[darkgray] (4) circle (0.07);

\node[] at ({0.7*(2*sqrt(3)-0.6)},0) {$e$};
\node[] at ({0.7*(2*sqrt(3)+0.6)},0.05) {$e'$};

\end{scope}

\begin{scope}[xshift=155, yshift=0]

\node[] at (-0.9,0) {$=$};

\begin{scope}[scale=0.7, xshift=0, yshift=0]

\coordinate [] (6) at ({3*sqrt(3)},0);
\coordinate [] (5) at ({2*sqrt(3)},-1);
\coordinate [] (4) at ({2*sqrt(3)},1);
\coordinate [] (3) at ({sqrt(3)},0);
\coordinate [] (2) at (0,-1);
\coordinate [] (1) at (0,1);
\coordinate [] (0) at (0,0);

\draw[lightgray, line width=1pt, double] (0)--(3);

\end{scope}

\end{scope}

\end{tikzpicture}
\end{center}
this can be accomplished in two ways: with two copies of the $3$-identity
\begin{center}
\begin{tikzpicture}[]

\begin{scope}[xshift=0, yshift=0]

\begin{scope}[scale=1.3, rotate=-30, xshift=24, yshift=14]

\coordinate [] (2a) at (0.5+0.03,{0.5*sqrt(3)});
\coordinate [] (2b) at (0.5-0.03,{0.5*sqrt(3)});
\coordinate [] (1a) at ({1+0.03*0.5},{0.03*0.5*sqrt(3)});
\coordinate [] (1b) at ({1-0.03*0.5},{-0.03*0.5*sqrt(3)});
\coordinate [] (0a) at ({0.03*0.5},{-0.03*0.5*sqrt(3)});
\coordinate [] (0b) at ({-0.03*0.5},{0.03*0.5*sqrt(3)});

\draw[lightgray, line width=1pt] (0b) to[out=30, in=270] (2b);
\draw[lightgray, line width=1pt] (2a) to[out=270, in=150] (1a);
\draw[lightgray, line width=1pt] (1b) to[out=150, in=30] (0a);

\end{scope}

\begin{scope}[scale=1.3, rotate=30, xshift=39, yshift=-39.5]

\coordinate [] (2a) at (0.5+0.03,{0.5*sqrt(3)});
\coordinate [] (2b) at (0.5-0.03,{0.5*sqrt(3)});
\coordinate [] (1a) at ({1+0.03*0.5},{0.03*0.5*sqrt(3)});
\coordinate [] (1b) at ({1-0.03*0.5},{-0.03*0.5*sqrt(3)});
\coordinate [] (0a) at ({0.03*0.5},{-0.03*0.5*sqrt(3)});
\coordinate [] (0b) at ({-0.03*0.5},{0.03*0.5*sqrt(3)});

\draw[lightgray, line width=1pt] (0b) to[out=30, in=270] (2b);
\draw[lightgray, line width=1pt] (2a) to[out=270, in=150] (1a);
\draw[lightgray, line width=1pt] (1b) to[out=150, in=30] (0a);

\end{scope}

\begin{scope}[scale=0.7, xshift=0, yshift=0]

\coordinate [] (6) at ({3*sqrt(3)},0);
\coordinate [] (5) at ({2*sqrt(3)},-1);
\coordinate [] (4) at ({2*sqrt(3)},1);
\coordinate [] (3) at ({sqrt(3)},0);
\coordinate [] (2) at (0,-1);
\coordinate [] (1) at (0,1);
\coordinate [] (0) at (0,0);

\end{scope}

\filldraw[darkgray] (5) circle (0.07);
\filldraw[darkgray] (4) circle (0.07);

\end{scope}

\begin{scope}[xshift=155, yshift=0]

\node[] at (-0.9,0) {$=$};

\begin{scope}[scale=0.7, xshift=0, yshift=0]

\coordinate [] (6) at ({3*sqrt(3)},0);
\coordinate [] (5) at ({2*sqrt(3)},-1);
\coordinate [] (4) at ({2*sqrt(3)},1);
\coordinate [] (3) at ({sqrt(3)},0);
\coordinate [] (2) at (0,-1);
\coordinate [] (1) at (0,1);
\coordinate [] (0) at (0,0);

\draw[lightgray, line width=1pt, double] (0)--(3);

\end{scope}

\end{scope}

\end{tikzpicture}
\end{center}
or with a $3$-identity and a broadened $2$-identity contracted at external indices:
\begin{center}
\begin{tikzpicture}[]

\begin{scope}[xshift=0, yshift=0]

\begin{scope}[scale=1.3, rotate=-30, xshift=24, yshift=14]

\coordinate [] (2a) at (0.5+0.03,{0.5*sqrt(3)});
\coordinate [] (2b) at (0.5-0.03,{0.5*sqrt(3)});
\coordinate [] (1a) at ({1+0.03*0.5},{0.03*0.5*sqrt(3)});
\coordinate [] (1b) at ({1-0.03*0.5},{-0.03*0.5*sqrt(3)});
\coordinate [] (0a) at ({0.03*0.5},{-0.03*0.5*sqrt(3)});
\coordinate [] (0b) at ({-0.03*0.5},{0.03*0.5*sqrt(3)});

\draw[lightgray, line width=1pt] (0b) to[out=30, in=270] (2b);
\draw[lightgray, line width=1pt] (2a) to[out=270, in=150] (1a);
\draw[lightgray, line width=1pt] (1b) to[out=150, in=30] (0a);

\end{scope}

\begin{scope}[scale=0.7, xshift=0, yshift=0]

\coordinate [] (6) at ({3*sqrt(3)},0);
\coordinate [] (5) at ({2*sqrt(3)},-1);
\coordinate [] (4) at ({2*sqrt(3)},1);
\coordinate [] (3) at ({sqrt(3)},0);
\coordinate [] (2) at (0,-1);
\coordinate [] (1) at (0,1);
\coordinate [] (0) at (0,0);

\draw[lightgray, dashed] (4)--(5)--(6)--cycle;
\draw[lightgray, line width=1pt, double] (5)--(6);

\end{scope}

\filldraw[darkgray] (5) circle (0.07);
\filldraw[darkgray] (4) circle (0.07);

\end{scope}

\begin{scope}[xshift=155, yshift=0]

\node[] at (-0.9,0) {$=$};

\begin{scope}[scale=0.7, xshift=0, yshift=0]

\coordinate [] (6) at ({3*sqrt(3)},0);
\coordinate [] (5) at ({2*sqrt(3)},-1);
\coordinate [] (4) at ({2*sqrt(3)},1);
\coordinate [] (3) at ({sqrt(3)},0);
\coordinate [] (2) at (0,-1);
\coordinate [] (1) at (0,1);
\coordinate [] (0) at (0,0);

\draw[lightgray, line width=1pt, double] (0)--(3);

\end{scope}

\end{scope}

\end{tikzpicture}
\end{center}
    
    \item $a$ as body. Contractions of $a$ with any of the candidate $3$-plexes alters it in general via self-contractions, therefore there are no canonical choices of $e$ and $e'$ that can make the diagrammatic equation hold.
    \item $a$ as head. The contraction of the tail $e$ and the body $e'$ is a $4$-plex that must contract with $a$ on two indices to give back $a$. The only $4$-plex that contracts in this manner is the tensor product of $2$-identities:
    \begin{center}
\begin{tikzpicture}[]

\begin{scope}[xshift=0, yshift=0]

\begin{scope}[scale=0.7, xshift=0, yshift=0]

\begin{scope}[xshift=0, yshift=0]

\draw[rounded corners] ({sqrt(3)-0.1},-1.2) rectangle ({2*sqrt(3)},1.2);

\draw[rounded corners] ({sqrt(3)-0.1},-0.3)--({sqrt(3)-0.7},-0.3)--({sqrt(3)-0.7},0.3)--({sqrt(3)-0.1},0.3);

\node[] at ({sqrt(3)-0.38},0) {$\otimes$};

\end{scope}

\coordinate [] (6) at ({3*sqrt(3)},0);
\coordinate [] (5) at ({2*sqrt(3)},-1);
\coordinate [] (4) at ({2*sqrt(3)},1);
\coordinate [] (3) at ({sqrt(3)},0);
\coordinate [] (2) at ({sqrt(3)},-1);
\coordinate [] (1) at ({sqrt(3)},1);
\coordinate [] (0) at (0,0);

\draw[lightgray, line width=1pt, double] (1)--(4);
\draw[lightgray, line width=1pt, double] (2)--(5);
\draw[gray, opacity=0.5,fill=lightgray,fill opacity=0.5] (4)--(5)--(6)--cycle;

\end{scope}

\node[] at ({0.7*(2*sqrt(3)+0.6)},0) {$a$};

\filldraw[darkgray] (5) circle (0.07);
\filldraw[darkgray] (4) circle (0.07);

\end{scope}

\begin{scope}[xshift=110, yshift=0]

\node[] at (0.5,0) {$=$};

\begin{scope}[scale=0.7, xshift=0, yshift=0]

\coordinate [] (6) at ({3*sqrt(3)},0);
\coordinate [] (5) at ({2*sqrt(3)},-1);
\coordinate [] (4) at ({2*sqrt(3)},1);
\coordinate [] (3) at ({sqrt(3)},0);
\coordinate [] (2) at ({sqrt(3)},-1);
\coordinate [] (1) at ({sqrt(3)},1);
\coordinate [] (0) at (0,0);

\draw[lightgray, line width=1pt, double] (1)--(4);
\draw[lightgray, line width=1pt, double] (2)--(5);
\draw[gray, opacity=0.5,fill=lightgray,fill opacity=0.5] (4)--(5)--(6)--cycle;

\end{scope}

\node[] at ({0.7*(2*sqrt(3)+0.6)},0) {$a$};

\filldraw[darkgray] (5) circle (0.07);
\filldraw[darkgray] (4) circle (0.07);

\end{scope}

\begin{scope}[xshift=270, yshift=0]

\node[] at (-0.9,0) {$=$};

\begin{scope}[scale=0.7, xshift=0, yshift=0]

\coordinate [] (6) at ({3*sqrt(3)},0);
\coordinate [] (5) at ({2*sqrt(3)},-1);
\coordinate [] (4) at ({2*sqrt(3)},1);
\coordinate [] (3) at ({sqrt(3)},0);
\coordinate [] (2) at (0,-1);
\coordinate [] (1) at (0,1);
\coordinate [] (0) at (0,0);

\draw[gray, opacity=0.5,fill=lightgray,fill opacity=0.5] (1)--(2)--(3)--cycle;

\end{scope}

\node[] at ({0.7*0.6},0) {$a$};

\end{scope}

\end{tikzpicture}
\end{center}
However, the $4$-plex must be obtained from the contraction of the tail $e$ and body $e'$ via one common index, which is easily checked to be impossible for any combinations of broadened $2$-identities and $3$-identities. 
\end{itemize}
These observations lead to the following result.
\begin{prop}[Fish Units] \label{fishunits}
The $3$-identity plexes and broadened $2$-identity plexes act as neutral elements when placed in body and head positions in the fish product.
\end{prop}
\begin{proof}
Consider an arbitrary $3$-plex defined on generic indices $I$, $J$, $K$, then we find the following diagrammatic equations by expanding identities on any of the indices:
\begin{center}
\begin{tikzpicture}[line join = round, line cap = round]

\begin{scope}[xshift=0, yshift=0]

\begin{scope}[scale=1.3, rotate=-30, xshift=24, yshift=14]

\coordinate [] (2a) at (0.5+0.03,{0.5*sqrt(3)});
\coordinate [] (2b) at (0.5-0.03,{0.5*sqrt(3)});
\coordinate [] (1a) at ({1+0.03*0.5},{0.03*0.5*sqrt(3)});
\coordinate [] (1b) at ({1-0.03*0.5},{-0.03*0.5*sqrt(3)});
\coordinate [] (0a) at ({0.03*0.5},{-0.03*0.5*sqrt(3)});
\coordinate [] (0b) at ({-0.03*0.5},{0.03*0.5*sqrt(3)});

\draw[lightgray, line width=1pt] (0b) to[out=30, in=270] (2b);
\draw[lightgray, line width=1pt] (2a) to[out=270, in=150] (1a);
\draw[lightgray, line width=1pt] (1b) to[out=150, in=30] (0a);

\end{scope}

\begin{scope}[scale=1.3, rotate=30, xshift=39, yshift=-39.5]

\coordinate [] (2a) at (0.5+0.03,{0.5*sqrt(3)});
\coordinate [] (2b) at (0.5-0.03,{0.5*sqrt(3)});
\coordinate [] (1a) at ({1+0.03*0.5},{0.03*0.5*sqrt(3)});
\coordinate [] (1b) at ({1-0.03*0.5},{-0.03*0.5*sqrt(3)});
\coordinate [] (0a) at ({0.03*0.5},{-0.03*0.5*sqrt(3)});
\coordinate [] (0b) at ({-0.03*0.5},{0.03*0.5*sqrt(3)});

\draw[lightgray, line width=1pt] (0b) to[out=30, in=270] (2b);
\draw[lightgray, line width=1pt] (2a) to[out=270, in=150] (1a);
\draw[lightgray, line width=1pt] (1b) to[out=150, in=30] (0a);

\end{scope}

\begin{scope}[scale=0.7, xshift=0, yshift=0]

\coordinate [label=right:$K$] (6) at ({3*sqrt(3)},0);
\coordinate [label=below:$K$] (5) at ({2*sqrt(3)},-1);
\coordinate [label=above:$K$] (4) at ({2*sqrt(3)},1);
\coordinate [label=above:$K$] (3) at ({sqrt(3)},0);
\coordinate [label=below left:$J$] (2) at (0,-1);
\coordinate [label=above left:$I$] (1) at (0,1);
\coordinate [] (0) at (0,0);

\draw[gray, opacity=0.5,fill=lightgray,fill opacity=0.5] (1)--(2)--(3)--cycle;

\end{scope}

\node[] at ({0.7*0.6},0) {$a$};

\filldraw[darkgray] (5) circle (0.07);
\filldraw[darkgray] (4) circle (0.07);
\filldraw[darkgray] (3) circle (0.07);

\end{scope}

\begin{scope}[xshift=170, yshift=0]

\node[] at (-1.1,0) {$=$};

\begin{scope}[scale=0.7, xshift=0, yshift=0]

\coordinate [] (6) at ({3*sqrt(3)},0);
\coordinate [] (5) at ({2*sqrt(3)},-1);
\coordinate [] (4) at ({2*sqrt(3)},1);
\coordinate [label=right:$K$] (3) at ({sqrt(3)},0);
\coordinate [label=below left:$J$] (2) at (0,-1);
\coordinate [label=above left:$I$] (1) at (0,1);
\coordinate [] (0) at (0,0);

\draw[gray, opacity=0.5,fill=lightgray,fill opacity=0.5] (1)--(2)--(3)--cycle;

\end{scope}

\node[] at ({0.7*0.6},0) {$a$};

\end{scope}

\end{tikzpicture}
\end{center}
\begin{center}
\begin{tikzpicture}[]

\begin{scope}[xshift=0, yshift=0]

\begin{scope}[scale=1.3, rotate=30, xshift=39, yshift=-39.5]

\coordinate [] (2a) at (0.5+0.03,{0.5*sqrt(3)});
\coordinate [] (2b) at (0.5-0.03,{0.5*sqrt(3)});
\coordinate [] (1a) at ({1+0.03*0.5},{0.03*0.5*sqrt(3)});
\coordinate [] (1b) at ({1-0.03*0.5},{-0.03*0.5*sqrt(3)});
\coordinate [] (0a) at ({0.03*0.5},{-0.03*0.5*sqrt(3)});
\coordinate [] (0b) at ({-0.03*0.5},{0.03*0.5*sqrt(3)});

\draw[lightgray, line width=1pt] (0b) to[out=30, in=270] (2b);
\draw[lightgray, line width=1pt] (2a) to[out=270, in=150] (1a);
\draw[lightgray, line width=1pt] (1b) to[out=150, in=30] (0a);

\end{scope}

\begin{scope}[scale=0.7, xshift=0, yshift=0]

\coordinate [label=right:$K$] (6) at ({3*sqrt(3)},0);
\coordinate [label=below:$K$] (5) at ({2*sqrt(3)},-1);
\coordinate [label=above:$K$] (4) at ({2*sqrt(3)},1);
\coordinate [label=above:$K$] (3) at ({sqrt(3)},0);
\coordinate [label=below left:$J$] (2) at (0,-1);
\coordinate [label=above left:$I$] (1) at (0,1);
\coordinate [] (0) at (0,0);

\draw[gray, opacity=0.5,fill=lightgray,fill opacity=0.5] (1)--(2)--(3)--cycle;
\draw[lightgray, dashed] (3)--(5)--(4)--cycle;
\draw[lightgray, line width=1pt, double] (3)--(5);

\end{scope}

\node[] at ({0.7*0.6},0) {$a$};

\filldraw[darkgray] (5) circle (0.07);
\filldraw[darkgray] (4) circle (0.07);
\filldraw[darkgray] (3) circle (0.07);

\end{scope}

\begin{scope}[xshift=170, yshift=0]

\node[] at (-1.1,0) {$=$};

\begin{scope}[scale=0.7, xshift=0, yshift=0]

\coordinate [] (6) at ({3*sqrt(3)},0);
\coordinate [] (5) at ({2*sqrt(3)},-1);
\coordinate [] (4) at ({2*sqrt(3)},1);
\coordinate [label=right:$K$] (3) at ({sqrt(3)},0);
\coordinate [label=below left:$J$] (2) at (0,-1);
\coordinate [label=above left:$I$] (1) at (0,1);
\coordinate [] (0) at (0,0);

\draw[gray, opacity=0.5,fill=lightgray,fill opacity=0.5] (1)--(2)--(3)--cycle;

\end{scope}

\node[] at ({0.7*0.6},0) {$a$};

\end{scope}

\end{tikzpicture}
\end{center}
\begin{center}
\begin{tikzpicture}[]

\begin{scope}[xshift=0, yshift=0]

\begin{scope}[scale=1.3, rotate=-30, xshift=24, yshift=14]

\coordinate [] (2a) at (0.5+0.03,{0.5*sqrt(3)});
\coordinate [] (2b) at (0.5-0.03,{0.5*sqrt(3)});
\coordinate [] (1a) at ({1+0.03*0.5},{0.03*0.5*sqrt(3)});
\coordinate [] (1b) at ({1-0.03*0.5},{-0.03*0.5*sqrt(3)});
\coordinate [] (0a) at ({0.03*0.5},{-0.03*0.5*sqrt(3)});
\coordinate [] (0b) at ({-0.03*0.5},{0.03*0.5*sqrt(3)});

\draw[lightgray, line width=1pt] (0b) to[out=30, in=270] (2b);
\draw[lightgray, line width=1pt] (2a) to[out=270, in=150] (1a);
\draw[lightgray, line width=1pt] (1b) to[out=150, in=30] (0a);

\end{scope}

\begin{scope}[scale=0.7, xshift=0, yshift=0]

\coordinate [label=right:$K$] (6) at ({3*sqrt(3)},0);
\coordinate [label=below:$K$] (5) at ({2*sqrt(3)},-1);
\coordinate [label=above:$K$] (4) at ({2*sqrt(3)},1);
\coordinate [label=above:$K$] (3) at ({sqrt(3)},0);
\coordinate [label=below left:$J$] (2) at (0,-1);
\coordinate [label=above left:$I$] (1) at (0,1);
\coordinate [] (0) at (0,0);

\draw[gray, opacity=0.5,fill=lightgray,fill opacity=0.5] (1)--(2)--(3)--cycle;
\draw[lightgray, dashed] (4)--(5)--(6)--cycle;
\draw[lightgray, line width=1pt, double] (5)--(6);

\end{scope}

\node[] at ({0.7*0.6},0) {$a$};

\filldraw[darkgray] (5) circle (0.07);
\filldraw[darkgray] (4) circle (0.07);
\filldraw[darkgray] (3) circle (0.07);

\end{scope}

\begin{scope}[xshift=170, yshift=0]

\node[] at (-1.1,0) {$=$};

\begin{scope}[scale=0.7, xshift=0, yshift=0]

\coordinate [] (6) at ({3*sqrt(3)},0);
\coordinate [] (5) at ({2*sqrt(3)},-1);
\coordinate [] (4) at ({2*sqrt(3)},1);
\coordinate [label=right:$K$] (3) at ({sqrt(3)},0);
\coordinate [label=below left:$J$] (2) at (0,-1);
\coordinate [label=above left:$I$] (1) at (0,1);
\coordinate [] (0) at (0,0);

\draw[gray, opacity=0.5,fill=lightgray,fill opacity=0.5] (1)--(2)--(3)--cycle;

\end{scope}

\node[] at ({0.7*0.6},0) {$a$};

\end{scope}

\end{tikzpicture}
\end{center}
where we have expanded on the index $K$ and thus the analogous diagrammatic equations hold by expanding on the indexes $I$ and $J$.
\end{proof}
By restricting to regular $3$-plexes defined on a single index set $I$ we can easily see that this result implies that the $3$-identity plex on $I$ becomes a one-sided\footnote{Right or left depending on the labelling convention.} biunit for the fish semiheap $A_{III}$. Crucially, however, no two-sided biunits can be canonically found in the fish semiheaps of regular $3$-plexes in general. This means that fish semiheaps are strictly inequivalent to involuted monoids as per Theorem \ref{monoidsemiheap} and thus appear as a novel family of examples of semiheap structures.\newline

Beyond the existence of canonical choices of $3$-plexes that behave like fish units, the diagrammatic equations can be taken as the defining conditions for specific $3$-plexes. In particular, as an analogue of semiheap biunits in the context of $3$-plex semiheapoids with the fish product, we define a \textbf{biunit pair} as two $3$-plexes $e$ and $e'$ satisfying
\begin{center}
\begin{tikzpicture}[line join = round, line cap = round]

\begin{scope}[xshift=0, yshift=0]

\begin{scope}[scale=0.7, xshift=0, yshift=0]

\coordinate [] (6) at ({3*sqrt(3)},0);
\coordinate [] (5) at ({2*sqrt(3)},-1);
\coordinate [] (4) at ({2*sqrt(3)},1);
\coordinate [] (3) at ({sqrt(3)},0);
\coordinate [] (2) at (0,-1);
\coordinate [] (1) at (0,1);
\coordinate [] (0) at (0,0);

\draw[gray, opacity=0.5,fill=lightgray,fill opacity=0.5] (1)--(2)--(3)--cycle;
\draw[gray, opacity=0.5,fill=lightgray,fill opacity=0.5] (3)--(4)--(5)--cycle;
\draw[gray, opacity=0.5,fill=lightgray,fill opacity=0.5] (4)--(5)--(6)--cycle;

\end{scope}

\node[] at ({0.7*0.6},0) {$a$};
\node[] at ({0.7*(2*sqrt(3)-0.6)},0) {$e$};
\node[] at ({0.7*(2*sqrt(3)+0.6)},0.05) {$e'$};

\filldraw[darkgray] (5) circle (0.07);
\filldraw[darkgray] (4) circle (0.07);
\filldraw[darkgray] (3) circle (0.07);

\end{scope}

\begin{scope}[xshift=155, yshift=0]

\node[] at (-0.9,0) {$=$};

\begin{scope}[scale=0.7, xshift=0, yshift=0]

\coordinate [] (6) at ({3*sqrt(3)},0);
\coordinate [] (5) at ({2*sqrt(3)},-1);
\coordinate [] (4) at ({2*sqrt(3)},1);
\coordinate [] (3) at ({sqrt(3)},0);
\coordinate [] (2) at (0,-1);
\coordinate [] (1) at (0,1);
\coordinate [] (0) at (0,0);

\draw[gray, opacity=0.5,fill=lightgray,fill opacity=0.5] (1)--(2)--(3)--cycle;

\end{scope}

\node[] at ({0.7*0.6},0) {$a$};

\end{scope}

\end{tikzpicture}
\end{center}
\begin{center}
\begin{tikzpicture}[line join = round, line cap = round]

\begin{scope}[xshift=0, yshift=0]

\begin{scope}[scale=0.7, xshift=0, yshift=0]

\coordinate [] (6) at ({3*sqrt(3)},0);
\coordinate [] (5) at ({2*sqrt(3)},-1);
\coordinate [] (4) at ({2*sqrt(3)},1);
\coordinate [] (3) at ({sqrt(3)},0);
\coordinate [] (2) at (0,-1);
\coordinate [] (1) at (0,1);
\coordinate [] (0) at (0,0);

\draw[gray, opacity=0.5,fill=lightgray,fill opacity=0.5] (1)--(2)--(3)--cycle;
\draw[gray, opacity=0.5,fill=lightgray,fill opacity=0.5] (3)--(4)--(5)--cycle;
\draw[gray, opacity=0.5,fill=lightgray,fill opacity=0.5] (4)--(5)--(6)--cycle;

\end{scope}

\node[] at ({0.7*0.6},0) {$e$};
\node[] at ({0.7*(2*sqrt(3)-0.6)},0.05) {$e'$};
\node[] at ({0.7*(2*sqrt(3)+0.6)},0) {$a$};

\filldraw[darkgray] (5) circle (0.07);
\filldraw[darkgray] (4) circle (0.07);
\filldraw[darkgray] (3) circle (0.07);

\end{scope}

\begin{scope}[xshift=155, yshift=0]

\node[] at (-0.9,0) {$=$};

\begin{scope}[scale=0.7, xshift=0, yshift=0]

\coordinate [] (6) at ({3*sqrt(3)},0);
\coordinate [] (5) at ({2*sqrt(3)},-1);
\coordinate [] (4) at ({2*sqrt(3)},1);
\coordinate [] (3) at ({sqrt(3)},0);
\coordinate [] (2) at (0,-1);
\coordinate [] (1) at (0,1);
\coordinate [] (0) at (0,0);

\draw[gray, opacity=0.5,fill=lightgray,fill opacity=0.5] (1)--(2)--(3)--cycle;

\end{scope}

\node[] at ({0.7*0.6},0) {$a$};

\end{scope}

\end{tikzpicture}
\end{center}
When $e=e'$ the notion of biunit pair recovers the usual notion of biunit for the fish semiheaps of Proposition \ref{fishsemiheap}. This further motivates the definition of \textbf{heapoid} in Section \ref{heapoids} below. A $3$-plex heapoid is a $3$-plex semiheapoid where any $3$-plex $a$ has an associated $3$-plex $a'$ such that the pair is a fish biunit.\newline

We conclude our discussion of the fish product by presenting a summary of our results expressed in terms of array algebra by choosing some index sequentializations. This will serve two purposes: on the one hand, it will allow us to use standard index notation to state our results and, on the other, it will make the connection to existing formalisms explicit. For concreteness, we shall focus on a set of $3$-plexes on fixed indices $I$, $J$, $K$ and the ternary operation $\eta_{IJK}$ as defined in Proposition \ref{fishsemiheap}. Sequentializing $3$-plexes amounts to choosing a linear order for the three distinct indices. We set:
\begin{equation*}
    (I,J,K) \qquad \mapsto \qquad (1,2,3)
\end{equation*}
that is, given a $3$-plex $a$ on the indices $I$, $J$, $K$ we have the following correspondence with arrays:
\begin{center}
\begin{tikzpicture}[line join = round, line cap = round]

\begin{scope}[scale=1.2, xshift=0, yshift=-10]

\coordinate [label=above:$k$] (2) at (0.5,{0.5*sqrt(3)});
\coordinate [label= below right:$j$] (1) at (1,0);
\coordinate [label=below left:$i$] (0) at (0,0);

\draw[lightgray, opacity=0.5,fill=lightgray,fill opacity=0.5] (1)--(0)--(2)--cycle;

\end{scope}

\node[] at (0.6,0) {$a$};

\node[] at (2.5,0) {$\mapsto$};

\node[] at (4,0) {$a_{ijk}$};

\end{tikzpicture}
\end{center}
The fish product $\eta_{IJK}$ thus becomes a ternary operation for $3$-arrays:
\begin{center}
\begin{tikzpicture}[line join = round, line cap = round]

\begin{scope}[xshift=0, yshift=0]

\begin{scope}[scale=0.7, xshift=0, yshift=0]

\coordinate [label=right:$k$] (6) at ({3*sqrt(3)},0);
\coordinate [label=below:$r$] (5) at ({2*sqrt(3)},-1);
\coordinate [label=above:$q$] (4) at ({2*sqrt(3)},1);
\coordinate [label=above:$p$] (3) at ({sqrt(3)},0);
\coordinate [label=below left:$j$] (2) at (0,-1);
\coordinate [label=above left:$i$] (1) at (0,1);
\coordinate [] (0) at (0,0);

\draw[gray, opacity=0.5,fill=lightgray,fill opacity=0.5] (1)--(2)--(3)--cycle;
\draw[gray, opacity=0.5,fill=lightgray,fill opacity=0.5] (3)--(4)--(5)--cycle;
\draw[gray, opacity=0.5,fill=lightgray,fill opacity=0.5] (4)--(5)--(6)--cycle;

\end{scope}

\node[] at ({0.7*0.6},0) {$a$};
\node[] at ({0.7*(2*sqrt(3)-0.6)},0.05) {$b$};
\node[] at ({0.7*(2*sqrt(3)+0.6)},0) {$c$};

\filldraw[darkgray] (5) circle (0.07);
\filldraw[darkgray] (4) circle (0.07);
\filldraw[darkgray] (3) circle (0.07);

\end{scope}

\node[] at (5.2,0) {$\mapsto$};

\node[] at (7,0) {$(abc)_{ijk}$};

\node[] at (8,0) {$=$};

\node[] at (10,-0.1) {$\displaystyle\sum_{pqr} a_{ijp}\cdot b_{qrp} \cdot c_{qrk}$};

\end{tikzpicture}
\end{center}
where we have grouped all the summed indices under a single summation sign for brevity -- note that each of the summed indices ranges over a distinct index set $p\in K$, $q\in I$, $r\in J$ and, consequently, they occur in the same positions in different arrays. This is indeed the ternary operation defined for the composition of triadas in \cite{longyear1972further} and the multiplication of cubic matrices in \cite{abramov2009algebras} that we identified as the fish product (\ref{fishmultplication}) in Section \ref{intro}. We must point out that our definition of fish product is more general than those found in \cite{longyear1972further} and \cite{abramov2009algebras}, since these references assumed $I=J=K$ in the definition of the ternary operation of $3$-arrays. Assuming that all $3$-arrays are regular is compatible with the conformability conditions of the fish product -- although by no means necessary -- and it leads to some additional freedom in the internal contraction of the indices. This was reflected in Theorem 4.1 of \cite{abramov2009algebras} that identified four possible sequentializations of the fish product for regular $3$-arrays:
\begin{center}
    \includegraphics[scale=0.26]{fish3.png}
\end{center}
Note that $\textit{(1)}$ and $\textit{(4)}$ correspond to the same fish diagram just by relabelling the $A$ $B$ $C$ arrays in either tail-body-head or head-body-tail sequential order; and similarly for $\textit{(2)}$ and $\textit{(3)}$. Therefore, we only need to distinguish between $\textit{(3)}$ and $\textit{(4)}$. Observe that $\textit{(4)}$ has a permutation of internal indices (the first to indices of $B$) with respect to $\textit{(3)}$. This simply corresponds to the fact that in the case of regular arrays the fish plex diagram defines two ternary operations by twisting.\newline

It follows from Theorem \ref{fishassoc} that this ternary operation of $3$-arrays satisfies the semiheap associativity property (\ref{semiheap}). In terms of array algebra, this is an immediate consequence of the associativity and distributivity properties of the value semiring. We can obtain the semiheap associativity equations (\ref{semiheap}) by bracketing products of three factors and distributing sums in the explicit array expression corresponding to the long fish hypergraph:
\begin{equation*}
    (abcde)_{ijk} = \sum_{pqrnml} a_{ijp}\cdot b_{qrp} \cdot c_{qrn} \cdot d_{mln} \cdot e_{mlk}.
\end{equation*}

The identity $3$-plexes are now expressed in terms of Kronecker arrays: the $3$-array identity on an index set $K$ is denoted by $t_{k_1k_2k_3}:=\delta_{k_1k_2k_3}$ and the identity $2$-array on $I$ broadened to $J$ is denoted by $u_{i_1i_2j}:=\delta_{i_1i_2}$, and similarly for other index sets by adding some upper script label. The diagrammatic equations of Proposition \ref{fishunits} now become the following array equations that are easily checked by direct computation:
\begin{align*}
    (att)_{ijk} &= \sum_{pqr} a_{ijp}\cdot \delta_{qrp} \cdot \delta_{qrk} = a_{ijk} \\
    (aut)_{ijk} &= \sum_{pqr} a_{ijp}\cdot \delta_{rp} \cdot \delta_{qrk} = a_{ijk} \\
    (atu)_{ijk} &= \sum_{pqr} a_{ijp}\cdot \delta_{qrp} \cdot \delta_{rk} = a_{ijk}
\end{align*}
Note that the first equation above amounts to the fact that the $3$-array identity behaves like a right biunit in the semiheap of regular $3$-arrays. The array equations for fish products are particularly useful to rewrite the biunit conditions. Recall that an element $e$ is a fish biunit when
\begin{equation*}
    (aee')=a=(ee'a).
\end{equation*}
By writing the corresponding array equations after sequentialization we see that these two equations are equivalent to:
\begin{align*}
    \sum_{p} e_{pij}\cdot e'_{pkl} &= \delta_{ik} \cdot \delta_{jl}  \\
    \sum_{qr} e_{iqr}\cdot e'_{jqr} &= \delta_{ij}.
\end{align*}
which should be read as a $3$-index analogue of the quadratic equation of invertible matrices. A collection of $3$-arrays satisfying the above quadratic equations forms the structure of a heapoid with the fish product as ternary composition.\newline

The flattened fish product becomes particularly transparent in explicit array notation: a flattened $3$-array $a_{ijk}$ on indices $(I,J,K)$ where $J$ and $K$ are considered as a multi-index is simply written as a $2$-array
\begin{equation*}
    a_{im}
\end{equation*}
where $m=(i,k)\in J\times K$. The fish product then becomes
\begin{equation*}
    (abc)_{im}=\sum_{pn} a_{in} \cdot b_{pn} \cdot c_{pm}
\end{equation*}
which we recognize as the concatenation of binary compositions of $2$-arrays and a transposition:
\begin{equation*}
    \sum_{pn} a_{in} \cdot b_{pn} \cdot c_{pm} = (a \circ b^\top \circ c)_{im}.
\end{equation*}
This explicitly recovers the semiheap operation defined on the morphism sets of dagger categories in Section \ref{heaps} for the particular case of the category of $2$-arrays with binary composition and transposition, generalizing the usual category of binary relations on sets or matrices over a ring. This result enables us to contextualize the question of existence of two-sided biunits in fish semiheaps in existing formalisms. Recall that semiheaps of heterogeneous relations only become heaps when isomorphic sets are considered \cite{hollings2017wagner}, therefore a fish semiheap could be restricted to a fish heap only when
\begin{equation*}
    I\cong J\times K
\end{equation*}
as sets. Indeed, the fish biunit equations for a pair of $3$-arrays $e$, $e'$ above are now rewritten for the flattened $2$-arrays in terms of ordinary binary composition as:
\begin{align*}
    e' \circ e^\top &= \Id_J  \otimes \Id_K = \Id_{J\times K} \\
    e^\top \circ e' &=\Id_I
\end{align*}
which we recognize as the condition that $e^\top$ and $e'$ are inverses and thus give an isomorphism of sets between $I$ and $J\times K$. This recovers the standard example of heaps as the morphisms sets of a groupoid.

\section{Heapoids and Fish Categories}\label{heapoids}

The behaviour $3$-plex diagrams discussed in Section \ref{fish} motivates the definition of a general category-like structure that subsumes the algebra of $3$-arrays endowed with the fish product. Our approach is to emulate how the standard definition of category can be abstracted from the basic example of binary relations between sets.\newline

As in the case of ordinary categories, we begin by considering a collection of \textbf{objects} without any internal structure denoted by $\{x,y,z, \dots\} = \mathcal{H}$. A general \textbf{ternary morphism}\footnote{Higher arity generalizations of morphisms were considered by V.V. Topentcharov \cite{topentcharov1988structures} and, more recently, by N. Baas \cite{baas2016higher} where they were called `bonds'.} is formally defined as a labelled multiset of three objects, that is, a generally unordered collection of three objects some of which may be repeated. A ternary morphism $a$ binding the trio of objects $x,y,z$ is depicted as a hyperedge in a similar fashion to the plex diagrams of Section \ref{plex}, graphically:
\begin{center}
\begin{tikzpicture}[line join = round, line cap = round]

\coordinate [label=above:$z$] (2) at (0.5,{0.5*sqrt(3)});
\coordinate [label= below right:$y$] (1) at (1,0);
\coordinate [label=below left:$x$] (0) at (0,0);

\draw[lightgray, opacity=0.5,fill=lightgray,fill opacity=0.5] (1)--(0)--(2)--cycle;

\node[] at (0.5,0.3) {$a$};

\end{tikzpicture}
\end{center}
The collection of all morphisms between a trio of objects $x,y,z \in \mathcal{H}$ is called the \textbf{morphism set} $\mathcal{H}\big( \raisebox{-3pt}[0pt][4pt]{{\stackon{$x\,y$}{$z$}}}\big)$ on those objects. Note that $\raisebox{-3pt}[0pt][4pt]{{\stackon{$x\,y$}{$z$}}}$ denotes any multiset of three objects and thus no sequential order is specified\footnote{This is directly inspired by the pre-sequential nature of plex diagrams in Section \ref{plex}.}. We shall remain agnostic with respect to sequential order and consider multisets instead of tuples. This will suffice to highlight the compositional aspects of generalized category-like structures below.\newline

The ordinary definition of composition in a semigroupoid or a category $\mathcal{C}$ involves a partial operation on the sets of morphisms subject to the source-and-target conformability conditions, that is, (binary) morphisms are composed via the formal assignment:
\begin{equation*}
    \mathcal{C}(xy) \times \mathcal{C}(yz) \to \mathcal{C}(xz)
\end{equation*}
Diagrammatically, composition of (binary) morphisms in a semigroupoid or category is an operation on simple graphs:
\begin{center}
\begin{tikzpicture}[line join = round, line cap = round]

\begin{scope}[xshift=0, yshift=0]

\coordinate [label=below right:$z$] (2) at (1,0);
\coordinate [label=above:$y$] (1) at (0,1);
\coordinate [label=below left:$x$] (0) at (-1,0);

\draw[opacity=0.5, lightgray, ultra thick] (0)--(1);
\draw[opacity=0.5, lightgray, ultra thick] (1)--(2);

\node[] at (-0.8,0.6) {$f$};
\node[] at (0.8,0.65) {$g$};

\end{scope}

\begin{scope}[xshift=150, yshift=10]

\node[] at (-2.5,0) {$\mapsto$};

\coordinate [label=right:$z$] (2) at (1,0);
\coordinate [] (1) at (0,1);
\coordinate [label=left:$x$] (0) at (-1,0);

\draw[opacity=0.5, lightgray, ultra thick] (0)--(2);

\node[] at (0,0.4) {$f\circ g$};

\end{scope}

\end{tikzpicture}
\end{center}
This suggests a natural generalization for higher order compositions of morphisms of arbitrary arity as operations on hypergraphs. In particular, we can define the \textbf{fish composition} of ternary morphisms as the formal assignment
\begin{equation*}
    \mathcal{H}\big( \raisebox{-3pt}[0pt][4pt]{{\stackon{$x\,y$}{$p$}}}\big) \times \mathcal{H}\big( \raisebox{-3pt}[0pt][4pt]{{\stackon{$q\,r$}{$p$}}}\big) \times \mathcal{H}\big( \raisebox{-3pt}[0pt][4pt]{{\stackon{$q\,r$}{$z$}}}\big) \to \mathcal{H}\big( \raisebox{-3pt}[0pt][4pt]{{\stackon{$x\,y$}{$z$}}}\big)
\end{equation*}
defined by the fish diagram:
\begin{center}
\begin{tikzpicture}[line join = round, line cap = round]

\begin{scope}[xshift=0, yshift=0]

\begin{scope}[scale=0.7, xshift=0, yshift=0]

\coordinate [label=right:$z$] (6) at ({3*sqrt(3)},0);
\coordinate [label=below:$r$] (5) at ({2*sqrt(3)},-1);
\coordinate [label=above:$q$] (4) at ({2*sqrt(3)},1);
\coordinate [label=above:$p$] (3) at ({sqrt(3)},0);
\coordinate [label=below left:$y$] (2) at (0,-1);
\coordinate [label=above left:$x$] (1) at (0,1);
\coordinate [] (0) at (0,0);

\draw[gray, opacity=0.5,fill=lightgray,fill opacity=0.5] (1)--(2)--(3)--cycle;
\draw[gray, opacity=0.5,fill=lightgray,fill opacity=0.5] (3)--(4)--(5)--cycle;
\draw[gray, opacity=0.5,fill=lightgray,fill opacity=0.5] (4)--(5)--(6)--cycle;

\end{scope}

\node[] at ({0.7*0.6},0) {$a$};
\node[] at ({0.7*(2*sqrt(3)-0.6)},0.05) {$b$};
\node[] at ({0.7*(2*sqrt(3)+0.6)},0) {$c$};

\end{scope}

\begin{scope}[xshift=180, yshift=0]

\node[] at (-1.3,0) {$\mapsto$};

\begin{scope}[scale=0.7, xshift=0, yshift=0]

\coordinate [] (6) at ({3*sqrt(3)},0);
\coordinate [] (5) at ({2*sqrt(3)},-1);
\coordinate [] (4) at ({2*sqrt(3)},1);
\coordinate [label=right:$z$] (3) at ({sqrt(3)},0);
\coordinate [label=below left:$y$] (2) at (0,-1);
\coordinate [label=above left:$x$] (1) at (0,1);
\coordinate [] (0) at (0,0);

\draw[gray, opacity=0.5,fill=lightgray,fill opacity=0.5] (1)--(2)--(3)--cycle;

\end{scope}

\node[] at ({0.7*0.6},0.04) {$abc$};

\end{scope}

\end{tikzpicture}
\end{center}
Enforcing the same rule of diagrammatic non-multiplicity that leads to the axiom of binary associativity in ordinary semigroupoids and categories, i.e. the fact that there exists a concurrent rewrite system with the composition of binary morphisms as rewrite rule as established in Section \ref{confl}, we are lead to the axiom of \textbf{fish associativity} as shown in Theorem \ref{fishassoc}. Symbolically, this corresponds to the fact that the long fish diagram for $5$ ternary morphisms:
\begin{center}
\begin{tikzpicture}[line join = round, line cap = round]

\begin{scope}[scale=0.7, xshift=0, yshift=0]

\coordinate [] (9) at ({5*sqrt(3)},0);
\coordinate [] (8) at ({4*sqrt(3)},-1);
\coordinate [] (7) at ({4*sqrt(3)},1);
\coordinate [] (6) at ({3*sqrt(3)},0);
\coordinate [] (5) at ({2*sqrt(3)},-1);
\coordinate [] (4) at ({2*sqrt(3)},1);
\coordinate [] (3) at ({sqrt(3)},0);
\coordinate [] (2) at (0,-1);
\coordinate [] (1) at (0,1);
\coordinate [] (0) at (0,0);

\draw[gray, opacity=0.5,fill=lightgray,fill opacity=0.5] (1)--(2)--(3)--cycle;
\draw[gray, opacity=0.5,fill=lightgray,fill opacity=0.5] (3)--(4)--(5)--cycle;
\draw[gray, opacity=0.5,fill=lightgray,fill opacity=0.5] (4)--(5)--(6)--cycle;
\draw[gray, opacity=0.5,fill=lightgray,fill opacity=0.5] (6)--(7)--(8)--cycle;
\draw[gray, opacity=0.5,fill=lightgray,fill opacity=0.5] (8)--(7)--(9)--cycle;

\end{scope}

\node[] at ({0.7*0.6},0) {$a$};
\node[] at ({0.7*(2*sqrt(3)-0.6)},0.05) {$b$};
\node[] at ({0.7*(2*sqrt(3)+0.6)},0) {$c$};
\node[] at ({0.7*(4*sqrt(3)-0.6)},0.05) {$d$};
\node[] at ({0.7*(4*sqrt(3)+0.6)},0) {$e$};

\end{tikzpicture}
\end{center}
evaluates to a single ternary morphism. This means that the multiple internal fish compositions must all be equal, thus leading to the semiheap property (\ref{semiheap}) introduced in Section \ref{heaps}:
\begin{equation*}
    ((abc)de)=(a(dcb)e)=(ab(cde)).
\end{equation*}
A collection of objects and ternary morphisms $\mathcal{H}$ endowed with such a ternary composition defined from the fish diagram is called a \textbf{semiheapoid}. The conformability condition of ternary morphism composition captured by the fish diagram is such that it restricts to a well-defined ternary operation on morphism sets. Indeed, consider a morphism set $\mathcal{H}\big( \raisebox{-3pt}[0pt][4pt]{{\stackon{$x\,y$}{$z$}}}\big)$, then we can define the ternary operation:
\begin{center}
\begin{tikzpicture}[line join = round, line cap = round]

\begin{scope}[xshift=0, yshift=0]

\begin{scope}[scale=0.7, xshift=0, yshift=0]

\coordinate [label=right:$z$] (6) at ({3*sqrt(3)},0);
\coordinate [label=below:$y$] (5) at ({2*sqrt(3)},-1);
\coordinate [label=above:$x$] (4) at ({2*sqrt(3)},1);
\coordinate [label=above:$z$] (3) at ({sqrt(3)},0);
\coordinate [label=below left:$y$] (2) at (0,-1);
\coordinate [label=above left:$x$] (1) at (0,1);
\coordinate [] (0) at (0,0);

\draw[gray, opacity=0.5,fill=lightgray,fill opacity=0.5] (1)--(2)--(3)--cycle;
\draw[gray, opacity=0.5,fill=lightgray,fill opacity=0.5] (3)--(4)--(5)--cycle;
\draw[gray, opacity=0.5,fill=lightgray,fill opacity=0.5] (4)--(5)--(6)--cycle;

\end{scope}

\node[] at ({0.7*0.6},0) {$a$};
\node[] at ({0.7*(2*sqrt(3)-0.6)},0.05) {$b$};
\node[] at ({0.7*(2*sqrt(3)+0.6)},0) {$c$};

\end{scope}

\begin{scope}[xshift=180, yshift=0]

\node[] at (-1.3,0) {$=:$};

\node[] at ({0.7*0.6},0.04) {$\eta_{xyz} (a,b,c)$};

\end{scope}

\end{tikzpicture}
\end{center}
the axiom of fish associativity then implies that the morphism set carries a semiheap structure
\begin{equation*}
    \big(\mathcal{H}\big( \raisebox{-3pt}[0pt][4pt]{{\stackon{$x\,y$}{$z$}}}\big), \eta_{xyz}\big)
\end{equation*}
for all trios of objects $x,y,z \in \mathcal{H}$. These semiheap structures are analogous to the fish semiheaps of Proposition \ref{fishsemiheap}.\newline

A ternary morphism $e$ in a semiheapoid $\mathcal{H}$ is called a \textbf{biunit} if there exists another ternary morphism $e'$ such that the following diagrammatic equations hold:
\begin{center}
\begin{tikzpicture}[line join = round, line cap = round]

\begin{scope}[xshift=0, yshift=0]

\begin{scope}[scale=0.7, xshift=0, yshift=0]

\coordinate [] (6) at ({3*sqrt(3)},0);
\coordinate [] (5) at ({2*sqrt(3)},-1);
\coordinate [] (4) at ({2*sqrt(3)},1);
\coordinate [] (3) at ({sqrt(3)},0);
\coordinate [] (2) at (0,-1);
\coordinate [] (1) at (0,1);
\coordinate [] (0) at (0,0);

\draw[gray, opacity=0.5,fill=lightgray,fill opacity=0.5] (1)--(2)--(3)--cycle;
\draw[gray, opacity=0.5,fill=lightgray,fill opacity=0.5] (3)--(4)--(5)--cycle;
\draw[gray, opacity=0.5,fill=lightgray,fill opacity=0.5] (4)--(5)--(6)--cycle;

\end{scope}

\node[] at ({0.7*0.6},0) {$a$};
\node[] at ({0.7*(2*sqrt(3)-0.6)},0) {$e$};
\node[] at ({0.7*(2*sqrt(3)+0.6)},0.05) {$e'$};

\end{scope}

\begin{scope}[xshift=155, yshift=0]

\node[] at (-0.9,0) {$=$};

\begin{scope}[scale=0.7, xshift=0, yshift=0]

\coordinate [] (6) at ({3*sqrt(3)},0);
\coordinate [] (5) at ({2*sqrt(3)},-1);
\coordinate [] (4) at ({2*sqrt(3)},1);
\coordinate [] (3) at ({sqrt(3)},0);
\coordinate [] (2) at (0,-1);
\coordinate [] (1) at (0,1);
\coordinate [] (0) at (0,0);

\draw[gray, opacity=0.5,fill=lightgray,fill opacity=0.5] (1)--(2)--(3)--cycle;

\end{scope}

\node[] at ({0.7*0.6},0) {$a$};

\end{scope}

\end{tikzpicture}
\end{center}
\begin{center}
\begin{tikzpicture}[line join = round, line cap = round]

\begin{scope}[xshift=0, yshift=0]

\begin{scope}[scale=0.7, xshift=0, yshift=0]

\coordinate [] (6) at ({3*sqrt(3)},0);
\coordinate [] (5) at ({2*sqrt(3)},-1);
\coordinate [] (4) at ({2*sqrt(3)},1);
\coordinate [] (3) at ({sqrt(3)},0);
\coordinate [] (2) at (0,-1);
\coordinate [] (1) at (0,1);
\coordinate [] (0) at (0,0);

\draw[gray, opacity=0.5,fill=lightgray,fill opacity=0.5] (1)--(2)--(3)--cycle;
\draw[gray, opacity=0.5,fill=lightgray,fill opacity=0.5] (3)--(4)--(5)--cycle;
\draw[gray, opacity=0.5,fill=lightgray,fill opacity=0.5] (4)--(5)--(6)--cycle;

\end{scope}

\node[] at ({0.7*0.6},0) {$e$};
\node[] at ({0.7*(2*sqrt(3)-0.6)},0.05) {$e'$};
\node[] at ({0.7*(2*sqrt(3)+0.6)},0) {$a$};

\end{scope}

\begin{scope}[xshift=155, yshift=0]

\node[] at (-0.9,0) {$=$};

\begin{scope}[scale=0.7, xshift=0, yshift=0]

\coordinate [] (6) at ({3*sqrt(3)},0);
\coordinate [] (5) at ({2*sqrt(3)},-1);
\coordinate [] (4) at ({2*sqrt(3)},1);
\coordinate [] (3) at ({sqrt(3)},0);
\coordinate [] (2) at (0,-1);
\coordinate [] (1) at (0,1);
\coordinate [] (0) at (0,0);

\draw[gray, opacity=0.5,fill=lightgray,fill opacity=0.5] (1)--(2)--(3)--cycle;

\end{scope}

\node[] at ({0.7*0.6},0) {$a$};

\end{scope}

\end{tikzpicture}
\end{center}
for all morphisms $a$ that are conformable as tail or head in the fish diagram with $e$ and $e'$. The pair $e, e'$ is called a \textbf{biunit pair}. A semiheapoid in which all ternary morphisms have a biunit pair is called a \textbf{heapoid}. A morphism $e$ that is its own biunit partner, i.e. when $e=e'$, is called a \textbf{Malcev} biunit and heapoids whose morphisms are Malcev biunits are called \textbf{Malcev heapoids}. The morphisms sets of Malcev heapoids $\big(\mathcal{H}\big( \raisebox{-3pt}[0pt][4pt]{{\stackon{$x\,y$}{$z$}}}\big), \eta_{xyz}\big)$ are heaps.\newline

As suggested by our discussion of the identity plexes in Section \ref{plex}, we can consider the ternary generalization of identity morphisms to be as follows: given any objects $x,y\in \mathcal{H}$ there always exist at least one morphism of the form
\begin{center}
\begin{tikzpicture}[line join = round, line cap = round]

\coordinate [label=above:$x$] (2) at (0.5,{0.5*sqrt(3)});
\coordinate [label= below right:$x$] (1) at (1,0);
\coordinate [label=below left:$x$] (0) at (0,0);

\draw[lightgray, opacity=0.5,fill=lightgray,fill opacity=0.5] (1)--(0)--(2)--cycle;

\node[] at (0.5,0.3) {$I$};

\end{tikzpicture}
\end{center}
called the \textbf{tridentity} on $x$, and one morphism of the form
\begin{center}
\begin{tikzpicture}[line join = round, line cap = round]

\coordinate [label=above:$y$] (2) at (0.5,{0.5*sqrt(3)});
\coordinate [label= below right:$x$] (1) at (1,0);
\coordinate [label=below left:$x$] (0) at (0,0);

\draw[lightgray, opacity=0.5,fill=lightgray,fill opacity=0.5] (1)--(0)--(2)--cycle;

\node[] at (0.5,0.3) {$i$};

\end{tikzpicture}
\end{center}
called the \textbf{partial identity} on $x$ extended to $y$. These singled-out ternary morphisms, collectively called \textbf{ternary identities}, are the analogues of the unique binary identity morphisms in ordinary categories. Following the diagrammatic equations of Proposition \ref{fishunits} we define a \textbf{fish category} as a semiheapoid where ternary identities exist and they satisfy the following diagrammatic equations:
\begin{center}
\begin{tikzpicture}[line join = round, line cap = round]

\begin{scope}[xshift=0, yshift=0]

\begin{scope}[scale=0.7, xshift=0, yshift=0]

\coordinate [] (6) at ({3*sqrt(3)},0);
\coordinate [] (5) at ({2*sqrt(3)},-1);
\coordinate [] (4) at ({2*sqrt(3)},1);
\coordinate [] (3) at ({sqrt(3)},0);
\coordinate [] (2) at (0,-1);
\coordinate [] (1) at (0,1);
\coordinate [] (0) at (0,0);

\draw[gray, opacity=0.5,fill=lightgray,fill opacity=0.5] (1)--(2)--(3)--cycle;
\draw[gray, opacity=0.5,fill=lightgray,fill opacity=0.5] (3)--(4)--(5)--cycle;
\draw[gray, opacity=0.5,fill=lightgray,fill opacity=0.5] (4)--(5)--(6)--cycle;

\end{scope}

\node[] at ({0.7*0.6},0) {$a$};
\node[] at ({0.7*(2*sqrt(3)-0.6)},0) {$I$};
\node[] at ({0.7*(2*sqrt(3)+0.6)},0) {$I$};

\end{scope}

\begin{scope}[xshift=155, yshift=0]

\node[] at (-0.9,0) {$=$};

\begin{scope}[scale=0.7, xshift=0, yshift=0]

\coordinate [] (6) at ({3*sqrt(3)},0);
\coordinate [] (5) at ({2*sqrt(3)},-1);
\coordinate [] (4) at ({2*sqrt(3)},1);
\coordinate [] (3) at ({sqrt(3)},0);
\coordinate [] (2) at (0,-1);
\coordinate [] (1) at (0,1);
\coordinate [] (0) at (0,0);

\draw[gray, opacity=0.5,fill=lightgray,fill opacity=0.5] (1)--(2)--(3)--cycle;

\end{scope}

\node[] at ({0.7*0.6},0) {$a$};

\end{scope}

\end{tikzpicture}
\end{center}
\begin{center}
\begin{tikzpicture}[line join = round, line cap = round]

\begin{scope}[xshift=0, yshift=0]

\begin{scope}[scale=0.7, xshift=0, yshift=0]

\coordinate [] (6) at ({3*sqrt(3)},0);
\coordinate [] (5) at ({2*sqrt(3)},-1);
\coordinate [] (4) at ({2*sqrt(3)},1);
\coordinate [] (3) at ({sqrt(3)},0);
\coordinate [] (2) at (0,-1);
\coordinate [] (1) at (0,1);
\coordinate [] (0) at (0,0);

\draw[gray, opacity=0.5,fill=lightgray,fill opacity=0.5] (1)--(2)--(3)--cycle;
\draw[gray, opacity=0.5,fill=lightgray,fill opacity=0.5] (3)--(4)--(5)--cycle;
\draw[gray, opacity=0.5,fill=lightgray,fill opacity=0.5] (4)--(5)--(6)--cycle;

\end{scope}

\node[] at ({0.7*0.6},0) {$a$};
\node[] at ({0.7*(2*sqrt(3)-0.6)},0) {$I$};
\node[] at ({0.7*(2*sqrt(3)+0.6)},0) {$i$};

\end{scope}

\begin{scope}[xshift=155, yshift=0]

\node[] at (-0.9,0) {$=$};

\begin{scope}[scale=0.7, xshift=0, yshift=0]

\coordinate [] (6) at ({3*sqrt(3)},0);
\coordinate [] (5) at ({2*sqrt(3)},-1);
\coordinate [] (4) at ({2*sqrt(3)},1);
\coordinate [] (3) at ({sqrt(3)},0);
\coordinate [] (2) at (0,-1);
\coordinate [] (1) at (0,1);
\coordinate [] (0) at (0,0);

\draw[gray, opacity=0.5,fill=lightgray,fill opacity=0.5] (1)--(2)--(3)--cycle;

\end{scope}

\node[] at ({0.7*0.6},0) {$a$};

\end{scope}

\end{tikzpicture}
\end{center}
\begin{center}
\begin{tikzpicture}[line join = round, line cap = round]

\begin{scope}[xshift=0, yshift=0]

\begin{scope}[scale=0.7, xshift=0, yshift=0]

\coordinate [] (6) at ({3*sqrt(3)},0);
\coordinate [] (5) at ({2*sqrt(3)},-1);
\coordinate [] (4) at ({2*sqrt(3)},1);
\coordinate [] (3) at ({sqrt(3)},0);
\coordinate [] (2) at (0,-1);
\coordinate [] (1) at (0,1);
\coordinate [] (0) at (0,0);

\draw[gray, opacity=0.5,fill=lightgray,fill opacity=0.5] (1)--(2)--(3)--cycle;
\draw[gray, opacity=0.5,fill=lightgray,fill opacity=0.5] (3)--(4)--(5)--cycle;
\draw[gray, opacity=0.5,fill=lightgray,fill opacity=0.5] (4)--(5)--(6)--cycle;

\end{scope}

\node[] at ({0.7*0.6},0) {$a$};
\node[] at ({0.7*(2*sqrt(3)-0.6)},0) {$i$};
\node[] at ({0.7*(2*sqrt(3)+0.6)},0) {$I$};

\end{scope}

\begin{scope}[xshift=155, yshift=0]

\node[] at (-0.9,0) {$=$};

\begin{scope}[scale=0.7, xshift=0, yshift=0]

\coordinate [] (6) at ({3*sqrt(3)},0);
\coordinate [] (5) at ({2*sqrt(3)},-1);
\coordinate [] (4) at ({2*sqrt(3)},1);
\coordinate [] (3) at ({sqrt(3)},0);
\coordinate [] (2) at (0,-1);
\coordinate [] (1) at (0,1);
\coordinate [] (0) at (0,0);

\draw[gray, opacity=0.5,fill=lightgray,fill opacity=0.5] (1)--(2)--(3)--cycle;

\end{scope}

\node[] at ({0.7*0.6},0) {$a$};

\end{scope}

\end{tikzpicture}
\end{center}

As illustrated by the ternary algebra of $3$-plexes at the end of Section \ref{fish}, the asymmetric anatomy of the fish composition implies that heapoids and fish categories are generically distinct classes of semiheapoids. In particular, the morphism sets of heapoids and fish categories give rise to inequivalent classes of semiheaps without two-sided biunits. To our knowledge, these notions are defined here for the first time. Heapoids and fish categories will be studied in more detail in future work.

\section{Towards Higher Associativity} \label{towards}

In this paper we have studied the behaviour of a particular ternary plex composition -- the fish product -- and we were able to rediscover a known form of generalized associativity -- the semiheap property (\ref{semiheap}) -- by analyzing the hypergraph rewrite system that results from the diagrammatic treatment of plexes. In turn, this lead to the definition of heapoids and a ternary notion of category -- what we called fish category -- and recovered familiar notions of invertibility conditions in the context of ternary morphisms, which generalize ternary relations and cubic matrices.\newline

This presents a clear path for future research on generalized associativity and generalized categories: by choosing a plex diagram representing a composition we can study its hypergraph rewrite properties and search for concurrency, as defined in Section \ref{confl}, to identify higher arity associativity axioms. Computational methods may be particularly useful since finding such rewrite systems amounts to enumerating some well-defined families of simple connected hypergraphs. We are currently working on the implementation of such algorithms \cite{collab2022arity}.\newline

More concretely, some plex diagrams are specially deserving of our initial attention when searching for ternary generalizations of associativity. As discussed in Section \ref{fish}, the fish product has a strongly sequential nature due to the asymmetric positions of the composed $3$-plexes. This prompts us to search for ternary plex diagrams whose composed $3$-plexes appear in symmetric positions. It turns out that there are only three such ternary compositions \cite{zapata2022invitation}. They are defined by the following plex diagrams:
\begin{center}
\begin{tikzpicture}[line join = round, line cap = round]

\begin{scope}[xshift=0, yshift=0]

\coordinate [] (3) at ({0.5*sqrt(3)},-0.5);
\coordinate [] (2) at (-{0.5*sqrt(3)},-0.5);
\coordinate [] (1) at (0,1);
\coordinate [] (0) at (0,0);

\draw[gray, opacity=0.5,fill=lightgray,fill opacity=0.5] (1)--(0)--(2)--cycle;
\draw[gray, opacity=0.5,fill=lightgray,fill opacity=0.5] (2)--(0)--(3)--cycle;
\draw[gray, opacity=0.5,fill=lightgray,fill opacity=0.5] (1)--(0)--(3)--cycle;

\filldraw[darkgray] (0) circle (0.07);

\end{scope}

\begin{scope}[xshift=100, yshift=0]

\coordinate [] (3) at ({0.5*sqrt(3)},-0.5);
\coordinate [] (2) at (-{0.5*sqrt(3)},-0.5);
\coordinate [] (1) at (0,1);
\coordinate [] (0a) at (0.3,0);
\coordinate [] (0b) at (-0.3,0);

\draw[gray, opacity=0.5,fill=lightgray,fill opacity=0.5] (1)--(0b)--(0a)--cycle;
\draw[gray, opacity=0.5,fill=lightgray,fill opacity=0.5] (2)--(0b)--(0a)--cycle;
\draw[gray, opacity=0.5,fill=lightgray,fill opacity=0.5] (3)--(0b)--(0a)--cycle;

\filldraw[darkgray] (0a) circle (0.07);
\filldraw[darkgray] (0b) circle (0.07);

\end{scope}

\begin{scope}[xshift=200, yshift=0]

\coordinate [] (3) at ({0.5*sqrt(3)},-0.5);
\coordinate [] (2) at (-{0.5*sqrt(3)},-0.5);
\coordinate [] (1) at (0,1);
\coordinate [] (0a) at (0.43,0.24);
\coordinate [] (0b) at (-0.43,0.24);
\coordinate [] (0c) at (0,-0.5);

\draw[gray, opacity=0.5,fill=lightgray,fill opacity=0.5] (1)--(0a)--(0b)--cycle;
\draw[gray, opacity=0.5,fill=lightgray,fill opacity=0.5] (2)--(0b)--(0c)--cycle;
\draw[gray, opacity=0.5,fill=lightgray,fill opacity=0.5] (3)--(0a)--(0c)--cycle;

\filldraw[darkgray] (0c) circle (0.07);
\filldraw[darkgray] (0b) circle (0.07);
\filldraw[darkgray] (0a) circle (0.07);

\end{scope}

\end{tikzpicture}
\end{center}
which we jokingly call the \textbf{holy trinity} of ternary plex compositions. The plex diagram on the left corresponds to the so-called Bhattacharya-Mesner product \cite{mesner1990association,gnang2014combinatorial,gnang2020bhattacharya} and the plex diagram on the right was identified as an operation of ternary relations in the cybernetics literature \cite{mcculloch1967triadas,longyear1972further} and as a cubic matrix multiplication \cite{kerner2008ternary}. To the best of our knowledge, the middle plex diagram represents a ternary composition that has not been previously identified in the literature.\newline

No associativity-like properties are known for these ternary operations or, indeed, any other higher order composition of arrays. We believe that the hypergraph rewrite approach via plex diagrams that we propose here will serve as a powerful tool to unravel the mystery of ternary and higher associativity.

\section*{Acknowledgements}

We would like to thank Irida Altman, Nils Baas, Jos\'e Figueroa-O'Farrill, Tom Brzezinski, Bernard Rybolowicz, Mark Lawson, Prathyush Pramod, Maximilian Schich, Pau Enrique, \'Alvaro Moreno, Rafael N\'u\~nez, Stephen Wolfram, Mundy Otto Reimer and Silvia Butti for all the feedback and eye-opening conversations.

\printbibliography

\end{document}